\documentclass{amsart}
\usepackage{mathabx}
\usepackage{amsmath, amsthm, amssymb}
\usepackage{mdwlist} 
\usepackage{graphicx}
\usepackage{ytableau}
\usepackage[utf8]{inputenc}
\usepackage[author-year]{amsrefs}
\usepackage{enumerate}

\usepackage[mathscr]{euscript}
\usepackage{mathrsfs}

\usepackage{pictexwd,dcpic}

\newtheorem{main}{Main Theorem}

\newtheorem{theorem}{Theorem}[subsection]
\newtheorem{corollary}[theorem]{Corollary}
\newtheorem{proposition}[theorem]{Proposition}
\newtheorem{lemma}[theorem]{Lemma}

\theoremstyle{definition}
\newtheorem{definition}[theorem]{Definition}
\newtheorem{problem}{Problem}

\theoremstyle{remark}
\newtheorem{remark}[theorem]{Remark}
\newtheorem{example}[theorem]{Example}

\newtheorem{rexample}{Running example,}

\renewcommand{\Re}{\mathrm{Re}}

\newcommand{\proj}{\mathrm{proj}}
\newcommand{\changed}[1]{{#1}}

\newcommand{\finsler}[2]{\left\vvvert #1 \right\vvvert_{#2}}

\newcommand{\quotient}[1]{\left[#1\right]}
\newcommand{\conv}[1]{\mathrm{Conv}(#1)}
\newcommand{\diam}[1]{\mathrm{diam}(#1)}

\newcommand{\avg}[1]{\mathrm{Avg}_{#1}}
\newcommand{\vol}{\mathrm{Vol}}
\newcommand{\dd}{\,\mathrm{d}}

\newcommand{\binomial}[2]{\ensuremath{\left( \begin{matrix}#1 \\ #2 \end{matrix} \right)}}

\newcommand{\defun}[5]{\ensuremath{\begin{array}{lrcl}
#1:&#2 & \longrightarrow & #3\\&#4 & \longmapsto & #5\end{array}}}
\newcommand{\bundle}[4]{\ensuremath{#1 \rightarrow #2 \stackrel{#3}{\rightarrow}{#4}}}
\newcommand{\longbundle}[4]{\ensuremath{#1 \rightarrow #2 \stackrel{#3}{\xrightarrow{\hspace{2em}}}{#4}}}

\newcommand{\defeq}{\stackrel{\mathrm{def}}{=}}

\newcommand{\ghostrule}[1]{\rule{0cm}{#1}}

\author{Gregorio Malajovich}
\title[Complexity of sparse polynomial solving]
{Complexity of sparse polynomial solving:
homotopy on toric varieties and the condition metric}
\address{Departamento de Matemática Aplicada, Instituto de Matemática, Universidade Federal do
Rio de Janeiro. Caixa Postal 68530, Rio de Janeiro RJ 21941-909, Brasil.}
\email{gregorio.malajovich@gmail.com}
\date{November 2$^{\mathrm{nd}}$, 2017}
\thanks{Part of these results were obtained while visiting the Simons 
Institute for the Theory of Computing in the University of California at 
Berkeley. This visit was funded by CAPES (Coordenação de
Aperfeiçoamento de Pessoal de Nível Superior, Brazil. Proc. BEX 2388/14-6). 
This research was also funded by CNPq grants 441678/2014-9 and 306673/2013-4.}

\begin{document}
\begin{abstract}
This paper investigates the cost of solving systems of sparse polynomial
equations by homotopy continuation. 
First, a space of systems of $n$-variate polynomial equations is
specified through $n$ monomial bases. The natural locus for the roots of
those systems is known to be a certain toric variety. This variety is
a compactification of $(\mathbb C\setminus\{0\})^n$, dependent on the
monomial bases. A toric Newton operator
is defined on that toric variety. Smale's alpha theory is
generalized to provide criteria of quadratic convergence. 
Two condition numbers are defined and
a higher derivative estimate is obtained in this setting. The Newton operator
and related condition numbers turn out to be invariant through a 
group action related to the momentum map. 
A homotopy algorithm is given, and is proved to terminate
after a number of Newton steps which is
linear on the condition length of the lifted homotopy path. This
generalizes a result from \ocite{Bezout6}.
\end{abstract}
\maketitle
\setcounter{tocdepth}{1}
\tableofcontents

\section{Introduction}

The solution of Smale's 17th problem by Beltrán and Pardo \ycites{Beltran-Pardo-2009,Beltran-Pardo-2011} and \ocite{Lairez} was
a tremendous breakthrough in the theory of solving polynomial
systems. Roughly, the cost of 
finding an {\em approximate solution} for a random system of 
$n$ polynomial equations on $n$ variables
is bounded by a polynomial in the input size. 
\par
Yet, several unanswered questions may prevent the immediate application
of those results and supporting algorithms.
One of the main obstructions comes from the way the input size was defined
by ~\ocite{Smale-next-century}. First the total degree $d_i$ of each equation $f_i$
is prescribed. Then,
\begin{quotation}\em
A probability measure must be put on the space of all such 
$\mathbf f$, for each 
$\mathbf d= (d_1, \dots, d_n)$ and the time of an algorithm is averaged over the 
space of $\mathbf f$. Is there such an algorithm
where the average time is bounded by a polynomial in the number of coefficients of $\mathbf f$ (the input size)?
\end{quotation}
Usually, the probability measure is assumed to be the normal distribution
with $0$ average and identity covariance with respect to Weyl's 
$U(n+1)$-invariant inner product. The input size of such a system is
therefore $\sum_{i=1}^n \binomial{d_i+n}{d_i}$.
\medskip
\par
Instead, a lot of the current numerical interest concentrates 
on systems of equations of the form 
\begin{equation}\label{sparse}
\begin{array}{lcl}
F_1(\mathbf Z) &=& \sum_{\mathbf a \in A_1} f_{i,\mathbf a} Z_1^{a_1}Z_2^{a_2}\cdots Z_n^{a_n}
\\
&\vdots&
\\
F_n(\mathbf Z) &=& \sum_{\mathbf a \in A_n} f_{i,\mathbf a} Z_1^{a_1}Z_2^{a_2}\cdots Z_n^{a_n},
\end{array}
\end{equation}
where each $A_i$ is a finite set. The natural input size for those systems
is $\sum_i \#A_i$ which can be exponentially smaller than 
$\sum_{i=1}^n \binomial{d_i+n}{d_i}$. 
\par
One of the main reasons
to find roots of a random system is to use them as a starting point
for a homotopy algorithm.
Sometimes, only the `finite' roots of a sparse system are needed.
Those are the roots in
$(\mathbb C \setminus \{0\})^n$. 
A famous theorem by Bernstein, Kushnirenko and Khovanskii\ycite{BKK} 
bounds the number
of such roots in terms of the {\em mixed volume} of the 
convex hulls of the $A_i$. This bound is tighter than 
Bézout's Theorem. The bound is exact once the roots are taken in
the proper compactification of $(\mathbb C \setminus\{0\})^n$ 
and counted with multiplicity. This compactification is a particular
toric variety. Properly detecting and finding `infinite'
roots in this toric variety is also an interesting problem.
Finding just 
{\em one} root of a random dense system could be very expensive
and would not necessarily provide a finite root of the sparse
target system, or even a legitimate `infinite' root in the toric variety.
Those considerations lead to the following theoretical questions:
\begin{problem}\label{problem1}
Can a finite zero of a random sparse polynomial system as in equation
\eqref{sparse}
be found approximately, on the average, in
time polynomial in $\sum_i \#A_i$ with a uniform algorithm?
\end{problem}
\begin{problem}\label{problem2}
Can every finite zero of a random polynomial system as in equation
\eqref{sparse}
be found approximately, on the average, in
time polynomial in $\sum_i \#A_i$ with a uniform algorithm
running in parallel, one parallel process for every expected zero?
\end{problem}
To simplify Problem~\ref{problem2}, one can assume that some preliminary information such as
a lower mixed subdivision is given as input to the algorithm.
An algorithm to find this mixed subdivision in time bounded in terms
of mixed volumes and other quermassintegrals was given by 
\ocite{Malajovich-Mixed}. 
Implementation issues were also discussed. \ocite{Jensen} provides an
alternative symbolic method which can also be used to recover this mixed
subdivision.
\medskip
\par
As a first step towards an investigation of problems ~\ref{problem1} 
and~\ref{problem2}, this paper attempts to develop a theory of 
homotopy algorithms
for sparse polynomial systems by following a parallel with the theory
for dense polynomial systems.
A key result in the theory  
was obtained by
\ocite{Bezout6}: the cost of homotopy is bounded above
by the {\em condition length} of the homotopy path
(see Section~\ref{proj-newton}).
The aim of this paper is to obtain a similar theorem
for sparse polynomial systems. 
\par
One of the cornerstones of that theory is the concept of $U(n+1)$ invariance
\cites{Bezout1, Bezout2, Bezout3, Bezout5, Bezout4, BCSS, Bezout6, 
Bezout7, Beltran-Pardo-2009, Beltran-Pardo-2011, Burgisser-Cucker, Adaptive}.
Unfortunately, unitary action does not preserve the structure
of equation ~(\ref{sparse}). In this paper,
the $U(n+1)$ invariance will be replaced
by another group action explained in Section~\ref{main-results}. 
\par
It is convenient to identify
sparse polynomials to exponential sums.
More formally, let $\mathscr F_{A_i}$ be the set of expressions of the
form $f_i(\mathbf z) = \sum_{\mathbf a \in A_i} f_{i\mathbf a} e^{\mathbf a \mathbf z}$ and let $\mathbf f \in \mathscr F_{A_1}
\times \cdots \times\mathscr F_{A_n}$ .
If $\mathbf f(\mathbf z)=0$ and 
$e^{\mathbf z} = \mathbf Z \in (\mathbb C\setminus\{0\})^n$ then $\mathbf Z$ is a finite zero of 
equation \eqref{sparse}. In section~\ref{main-results} 
we will construct the {\em toric variety} 
$\mathscr V$ as 
the Zariski closure of a non-unique embedding of $(\mathbb C \setminus \{0\})^n$ 
into $\mathbb P(\mathscr F_{A_1}^*) \times \cdots \times \mathbb P(\mathscr F_{A_n}^ *)$.
Actual computations require the use of some local chart. We 
will use a system of `logarithmic coordinates' $[V]:\mathscr M \rightarrow \mathscr V$   
where $\mathscr M$ is the quotient of the $z$-space $\mathbb C^n$ 
that makes the embedding
injective. To every point $x \in \mathscr M$ we will associate
the local norm $\|\cdot\|_{\mathbf x}$ induced by the pull-back of Fubini-Study
metric from $\mathscr V$. Another possibility discussed in
section \ref{sec:Finsler} is to endow $\mathscr M$ with a 
Finsler structure.
We will also define a Newton operator
on $\mathscr V$ which will actually operate on $\mathscr M$
as a (locally) linear space.
This will avoid all the technicalities associated to
Newton iteration on manifolds such as estimating covariant derivatives
or approximating geodesics, as required in previous work from 
\ocite{Dedieu-Priouret-Malajovich}. However $\mathscr M$ is still a 
manifold, with a metric structure associated to each point.
We may estimate the distance between two points 
$\mathbf x$, $\mathbf z$ through the norm $\|\mathbf x - \mathbf z\|_{\mathbf x}$ on the tangent space $T_{\mathbf x}\mathscr M$. 
The subtraction operator above is provided by the linear structure of 
$\mathbb C^n$, and it is assumed that representatives $\mathbf x$ and $\mathbf z$ in $\mathbb C^n$ minimize the norm.

In this paper, the {\em solution variety} is
\[
\mathscr S_0 = \left\{\ghostrule{3ex} (\mathbf f, \mathbf x) \in \mathbb P(\mathscr F_{A_1}) \times \cdots \times 
\mathbb P(\mathscr F_{A_n}) \times \mathscr M: \mathbf f(\mathbf x)=0 \right\}
\]
We will define two condition numbers $\mu:\mathbb P(\mathscr F_{A_1}) \times \cdots \times 
\mathbb P(\mathscr F_{A_n}) \times \mathscr M  \rightarrow
[1, \infty]$ and $\nu: \mathscr M \rightarrow [1,\infty]$. 
Let $\Sigma'$ be the set of ill-posed pairs, that is the set of all $(\mathbf f,\mathbf x)$ 
with $\mu(\mathbf f,\mathbf x) \nu(\mathbf x)=\infty$. 
The condition numbers induce
a length structure on $\mathscr S_0 \setminus \Sigma'$: the
{\em condition length} of a 
rectifiable path
$(\mathbf f_t,\mathbf z_t)_{t \in [t_0,t_1]}$
is defined by
\[
\mathscr L\left(\ghostrule{2.5ex}(\mathbf f_t, \mathbf z_t); t_0, t_1\right) =
\int_{t_0}^{t_1}   
\mu(\mathbf f_t,\mathbf z_t) \nu(\mathbf z_t) \sqrt{\| \dot {\mathbf f}_t \|_{\mathbf f_t}^2 + \| \dot {\mathbf z}_t\|_{\mathbf z_t}^2}
\ \dd t
.\]
This gives $\mathscr S_0 \setminus \Sigma'$ the structure of a 
path-metric space.

\begin{main}\label{th-homotopy}
Let $(\mathbf f_t,\mathbf z_t)_{t \in [0,T]}$ 
be a rectifiable path in $\mathscr S_0 \setminus \Sigma'$.
Let $\mathbf x_0$ be an approximation for $\mathbf z_0$, satisfying 
\[
\frac{1}{2}\mu(\mathbf f_{0},\mathbf z_{0}) \nu(\mathbf z_{0}) 
\|\mathbf z_{0} - \mathbf x_0\|_{\mathbf z_{0}} 
\le u_0  
\]
for the constant $u_0=\frac{3-\sqrt{7}}{2} \simeq 0.090994\cdots$. 
Then, there is a time mesh
$0=t_0 < t_1 < \cdots < t_N=T$ with
\[
N \le \left \lceil 38 \ \mathscr L\left(\ghostrule{2.5ex}(\mathbf f_t,\mathbf z_t);0,T\right) \right\rceil
\]
so that the approximation
\[
\mathbf x_{i+1} = {\mathbf N}(\mathbf f_{t_i}, \mathbf x_{i})
\]
produces 
$\mathbf y_0=\mathbf x_N$ with
\[
\frac{1}{2}
\mu(\mathbf f_{T},\mathbf z_{T}) \nu(\mathbf z_{T}) 
\|\mathbf z_{T} - \mathbf y_0\|_{\mathbf z_{T}} 
\le u_0
\]
for the same constant $u_0$. 
Moreover,
the sequence $\mathbf y_{i+1} =  {\mathbf N}(\mathbf f_T,\mathbf y_i)$ is well-defined
and satisfies
\[
\| \mathbf y_i - \mathbf z_T\|_{\mathbf z_T} \le 2^{-2^i+1} \| \mathbf y_0 - \mathbf z_T\|_{\mathbf z_T} . 
\]
\end{main}

Main Theorem A is not effective, in the sense that
the time mesh above is just said to exist. 
One can get an adaptive criterion for the step size at the
price of increasing the complexity bound.

\begin{main}\label{cor:main}
There are constants 
\[
\alpha_1 \simeq 0.081239483\cdots 
\hspace{2em}\text{and}\hspace{2em}
u_1 \simeq 0.039745185\cdots
\]
with the following properties:
Let $(\mathbf f_t,\mathbf z_t)_{t \in [0,T]}$ 
be a rectifiable path in $\mathscr S_0 \setminus \Sigma'$.
Let $\mathbf x_0$ be an approximation for $\mathbf z_0$, satisfying 
\[
\frac{1}{2}\mu(\mathbf f_{0},\mathbf z_{0}) \nu(\mathbf z_{0}) 
\|\mathbf z_{0} - \mathbf x_0\|_{\mathbf z_{0}} 
\le u_1.
\]
Then one can 
define $(\mathbf x_i)$ and $(t_i)$ recursively by
\[
\left\{
\begin{array}{lcl}
\mathbf x_{i+1} &=&  {\mathbf N}(\mathbf f_{t_i}, \mathbf x_{i})\\
t_{i+1} &=& \min\left( T, \inf \left\{ \ghostrule{2.5ex} 
t>t_i: \right. \right. \\
& & \left. \left.
\hspace{5em}\ghostrule{2.5ex}\frac{1}{2}{\mu}(\mathbf f_{t}, \mathbf x_{i+1}) \nu(\mathbf x_{i+1}) 
\|  {\mathbf N}_{\mathbf f_{t}}(\mathbf x_{i+1})-\mathbf x_{i+1}\|_{\mathbf x_{i+1}}
\ge \alpha_1 \right\}\right).
\end{array} \right.
\]
Then, $t_N = T$ for some $N\le \left\lceil {59} 
\mathscr L\left( \ghostrule{2.5ex}(\mathbf f_t,\mathbf z_t);0,T\right) \right\rceil$.
Moreover, 
the sequence $\mathbf y_0=\mathbf x_N$, $\mathbf y_{i+1} =  {\mathbf N}_{\mathbf f_T}(\mathbf y_i)$ is well-defined
and satisfies for $i \ge 1$
\[
\| \mathbf y_i - \mathbf z_T\|_{\mathbf z_T} \le 2^{-2^{i-1} -2} \|\mathbf y_0-\mathbf y_1\| 
.
\]
\end{main}
The calculation of $t_{i+1}$ requires a subroutine to find the smallest
solution $t > t_i$ of 
the equation 
\[
\frac{1}{2} {\mu}(\mathbf f_{t}, \mathbf x_{i+1}) \nu(\mathbf x_{i+1}) 
\|  {\mathbf N}_{\mathbf f_{t}}(\mathbf x_{i+1})-\mathbf x_{i+1}\|_{\mathbf x_{i+1}}
=  \alpha_1.
\]
Obvious modifications in the algorithm allow 
for approximate computations in that subroutine. Similar results
were known for the dense setting
\cites{Beltran2011,Adaptive,Beltran-Leykin}. The constants in Main 
Theorem B are not supposed to be sharp.
\medskip
\par
Last but not least, the methods in this paper may offer a better 
alternative than projective Newton for
approximating certain roots at `toric infinity'. 
We will show this through an example. 

\begin{rexample}
The family 
\begin{equation}\label{eq-example}
\mathbf f_t(X,Y) = \begin{pmatrix}
tX - tXY + Y^2 - t^2 Y^3 \\
X + XY -Y^2 -Y^3.
\end{pmatrix}
\end{equation}
admits two `finite' solutions on the toric variety $\mathscr V$, namely
$(t^{-2},t^{-1})$ and $(-\frac{1+t^2}{2t},$ $-1)$. When $t \rightarrow 0$,
both solutions converge to different points at toric `infinity' and 
those can be efficiently approximated. We will show in Section~\ref{main-results}
that
\[
\mathscr L\left( \ghostrule{2.5ex}\left (\mathbf f_t,(x_t,y_t)\right);\epsilon,1\right)
\in \Theta(\log(1/\epsilon)) 
\]
where $(x,y)=(\log(X),\log(Y))$. In comparison, we show in Section~\ref{proj-newton}
that the condition length $L$
for the homogeneous setting as in \cite{Bezout6} satisfies
\[
L\left( \ghostrule{2.5ex}(\mathbf f_t,[X_t:Y_t:1]);\epsilon,1\right)
\in {\Omega(1/\epsilon)}. 
\]
This amounts to an exponentially worse bound on the number of homotopy
steps,
due to the fact that in projective space the two solutions 
are the undistinguishable on the limit.
Indeed, $\lim_{t \rightarrow 0} [X_t:Y_t:1] = [1:0:0]$ 
for both curves.
\end{rexample}
\medskip
\par

This paper is organized as follows. Section~\ref{proj-newton} revisits
known results about alpha-theory, for reference and conceptual
clarification.  
All the main results and constructs of this paper are contained in 
Section~\ref{main-results}. Among them, the
construction of the toric variety, the Newton operator and
the momentum map action. Main theorems
A and B are proved, but the proofs of intermediate results are postponed. 
Section~\ref{distortion} contains distortion
bounds that allow to switch between charts in $\mathscr M$. 
The remaining technical results are proved in Section~\ref{proof:tech}.
\par
In section \ref{sec:Finsler} an alternative, more natural Finsler
structure on the toric variety $\mathscr V$ is introduced. All
the theorems in this paper are also valid if the Hermitian structure
is replaced by this Finsler structure, and some bounds actually become
sharper. A short summary and some short remarks close the paper
in section \ref{sec:conclusions}
\bigskip
\paragraph{\bf Acknowledgements:} The author would like to thank 
Carlos Beltrán, Bernardo Freitas Paulo da Costa, Felipe Bottega Diniz and two anonymous referees for their suggestions and improvements.

\section{Projective Newton iteration revisited}
\label{proj-newton}

In this section we revisit some classical results about
Newton iteration, such as Smale's quadratic convergence 
theorems. Then
we recall the corresponding results for projective Newton iteration. 
By understanding projective Newton as an algorithm operating on vector bundles,
we highlight some subtle differences between the 
gamma theorem
which extends naturally to projective space, and the alpha theorem.

\subsection{Classical theorems}

Let $\mathbf f: \mathbb E \rightarrow \mathbb F$ be an analytic mapping
between real or complex Banach spaces. Whenever $D\mathbf f(\mathbf x)$ is invertible,
Newton iteration is defined by
\[
\defun{ {\mathbf N}_{\mathbf f}}{\mathbb E}{\mathbb F}{\mathbf x}{\mathbf x-D\mathbf f(\mathbf x)^{-1}\mathbf f(\mathbf x)}
\] 
Smale's invariants for Newton iterations are:
\[
\beta(\mathbf f,\mathbf x) = 
\left\|D\mathbf f(\mathbf x)^{-1}
\mathbf f(\mathbf x) \right\|,
\]
\[
\gamma(\mathbf f,\mathbf x) = 
\max_{k \ge 2} 
\left(\frac{1}{k!}\left\|D\mathbf f(\mathbf x)^{-1}
D^k\mathbf f(\mathbf x)\right\|
\right)^{\frac{1}{k-1}}
\]
and $\alpha(\mathbf f,\mathbf x) = \beta(\mathbf f,\mathbf x) \gamma(\mathbf f,\mathbf x)$. If {$D\mathbf f(\mathbf x)$} fails to be
surjective at $\mathbf x$, then $\alpha(\mathbf f,\mathbf x)=\beta(\mathbf f,\mathbf x)=\gamma(\mathbf f,\mathbf x)=\infty$.
Recall also the teminology: a zero $\mathbf z$ of $\mathbf f$ is said
to be degenerate if $D\mathbf f(\mathbf z)$ is not surjective, otherwise it is non-degenerate. 
The domain of $\mathbf f$ will be denoted $\mathcal D_{\mathbf f}$ and $B(\mathbf x,r)$ will be
the radius $r$ ball around $\mathbf x$. The following two results are due
to \ocite{Smale-PE}. The constant $\alpha_0$ below is due to
\ocite{Wang-Xing-Hua}. Proofs can be found on textbooks or lecture notes
such as  
\cites{BCSS,Malajovich-nonlinear, Malajovich-UIMP}.

\begin{theorem}[$\gamma$-theorem]\label{th-gamma}
Let $\boldsymbol \zeta \in \mathbb E$ be a non-degenerate zero of $f$.
If $x_0 \in \mathbb E$ satisfies
\[
\|\boldsymbol \zeta - \mathbf x_0\| \gamma(\mathbf f, \boldsymbol \zeta) \le \frac{3 - \sqrt{7}}{2}
\]
and $B(\boldsymbol \zeta, \|\boldsymbol \zeta - \mathbf x_0\|)\subseteq \mathcal D_{\mathbf f}$,
then the sequence $\mathbf x_{i+1} =  {\mathbf N}_{\mathbf f}(\mathbf x_i)$ is well-defined
and
\[
\| \boldsymbol \zeta - \mathbf x_i\| \le 2^{-2^i+1} \|\boldsymbol \zeta - \mathbf x_0\| . 
\]
\end{theorem}
\begin{theorem}[$\alpha$-theorem]\label{th-alpha}
Let \[
\alpha 
\le
\alpha_0 = \frac{13 - 3 \sqrt{17}}{4},
\]
\[
r_0 = 
\frac{1+\alpha-\sqrt{1-6\alpha+\alpha^2}}{4\alpha} 
\text{ and }
r_1 =
\frac{1-3\alpha-\sqrt{1-6\alpha+\alpha^2}}{4\alpha} 
.
\]
If $\mathbf x_0 \in \mathbb E$ satisfies $\alpha(\mathbf f,\mathbf x_0) \le \alpha$,
and $B({\mathbf x_0}, r_0 \beta(\mathbf f,\mathbf x_0))\subseteq \mathcal D_{\mathbf f}$,
then the sequence defined recursively by
$\mathbf x_{i+1} =  {\mathbf N}_{\mathbf f}(\mathbf x_i)$
is well-defined and converges to a limit $\boldsymbol \zeta$ so that
${\mathbf f}(\boldsymbol \zeta)=0$. Furthermore,
\begin{enumerate}[(a)]
\item
$
\| \mathbf x_i - \boldsymbol \zeta \| \le 2^{-2^i+1} \| \mathbf x_1 - \mathbf x_0 \|
$
\item 
$
\| \mathbf x_0 - \boldsymbol \zeta\| \le r_0 \beta(\mathbf f, \mathbf x_0)
$
\item
$
\| \mathbf x_1-\boldsymbol \zeta \|
\le r_1 \beta(\mathbf f, \mathbf x_0).
$
\end{enumerate}
\end{theorem}

\subsection{The case for projective Newton}
Polynomial equations in $\mathbb C^n$ are poorly conditionned when a
root `approaches' infinity. For instance, the affine system of equations
\[
\left\{
\begin{array}{ccc}
\epsilon x - 1 & = & 0\\
y -1 & = & 0
\end{array}
\right.
\]
has solution $(\epsilon^{-1},1)$. A small perturbation of the first coefficient
by (say) $\delta$ may change the solution to $((\epsilon-\delta)^{-1},1)$.
The absolute condition number is by definition 
$\left| \frac{\partial}{\partial \delta}_{|\delta=0}\ 
\frac{1}{\epsilon - \delta}\right| =
\epsilon^{-2}$, while
the relative condition number is the absolute condition number divided by
the limit value $\epsilon^{-1}$, namely $\epsilon^{-1}$.
\par
This source of ill-posedness was noticed by \ocite{Bezout1}*{section I-4}.
In comparison, the theory was greatly simplified by homogenizing 
equations and then performing
Newton iteration on projective space. On the previous example, the 
homogenized system is
\[
\left\{
\begin{array}{ccc}
\epsilon x - z & = & 0\\
y -z & = & 0
\end{array}
\right.
\]
and the solution $\quotient{\epsilon^{-1}:1:1} = \quotient{1:\epsilon:\epsilon}$ has a 
well-defined limit as $\epsilon \rightarrow {0}$. 
\par
Those ideas require the introduction of an appropiate Newton operator.
One possibility is to perform Newton iteration in $\mathbb C^{n+1}$ using
the Moore-Penrose pseudo-inverse as \ocite{AllgowerGeorg}, or charts as
\ocite{Morgan}.
However, the projective Newton operator introduced by \ocite{ShubProjective} allowed
for a more natural development of the theory. 
\medskip
\par
\subsection{The line bundle $\mathcal O(d)$}
Homogeneous polynomials do not have a well-defined value on projective
space $\mathbb P^n$. A classical construction in algebraic geometry is
to represent homogeneous degree $d$ polynomials as sections of the line
bundle $\bundle{\mathbb C}{\mathcal O(d)}{\pi}{\mathbb P^n}$ with total space $\mathcal O(d)$
equal to the
quotient of
$(\mathbb C^{n+1} \setminus \{0\}) \times \mathbb C$
by the $\mathbb C_{\times}$ group action
\[
\lambda (\mathbf x, y) = (\lambda \mathbf x, \lambda^d y) 
.
\]
When no confusion can arise, we will use the same notation for a
fiber bundle and its total space.  Through this paper, brackets 
denote the equivalence class
under a prescribed group action. For instance, $\quotient{\mathbf x} \in \mathbb P^n$ will be
the equivalence class of $\mathbf x \in \mathbb C^{n+1}\setminus\{0\}$ with respect to scalings, and
$\quotient{\mathbf x,y} \in \mathcal O(d)$ will be the equivalence 
class of $(\mathbf x,y)$ under the
group action above. Under this notation, the projection operator
$\pi:\mathcal O(d) \rightarrow \mathbb P^n$ is just $\quotient{\mathbf x,y} \mapsto \quotient{\mathbf x}$.
\par
\par
To a homogeneous degree $d$ polynomial $f$, one associates the
section $s_f: \quotient{\mathbf x} \mapsto \quotient{\mathbf x,f(\mathbf x)}$. 
The reader should check that this
is independent of the choice of the representative $\mathbf x$ for $\quotient{\mathbf x}$.

\subsection{Systems of equations}
Let $d_1, \dots, d_n \in \mathbb N$ be fixed through this section.
We consider the vector bundle $\mathscr E =
\mathcal O(d_1) \oplus \cdots \oplus \mathcal O(d_n)$.
Denoting also by $\mathscr E$ its total space, we may
write this bundle as
$\bundle{\mathbb C^n}{\mathscr E}{\pi}{\mathbb P^ n}$.
The total space
$\mathscr E$ is the quotient of 
$(\mathbb C^{n+1} \setminus \{0\}) \times \mathbb C^n$
by the $\mathbb C_{\times}$ group action
\[
\lambda(\mathbf x,y_1, \dots, y_n) =
(\lambda \mathbf x, \lambda^{d_1}y_1, \dots, \lambda^{d_n}y_n).
\]
The projection map takes $[\mathbf x,\mathbf y]$ into $[\mathbf x]$.

To a system $(f_1, \dots, f_n)$ of homogeneous polynomials of degree
$(d_1, \dots, d_n)$, one associates the section of the vector bundle
\[
\defun{
s_{(f_1, \dots, f_n)}
}
{\mathbb P^n}
{\mathscr E}
{\quotient{\mathbf x}}{\quotient{\mathbf x, f_1(\mathbf x), \dots, f_n(\mathbf x)}}
.
\]
The brackets on the right denote quotient with respect to the
multiplicative group action $\lambda(\mathbf x,y_1, \dots, y_n) =
(\lambda \mathbf x, \lambda^{d_1}y_1, \dots, \lambda^{d_n}y_n)$.
The tangent space of $\mathbb P^n$ at $\mathbf x$ is
the linear space $\mathbf x^{\perp} \subset \mathbb C^{n+1}$ 
with the inner product
$\|\mathbf x\|^{-2} \langle \cdot,\cdot \rangle$.
We can define a local map from $T_{\quotient{\mathbf x}}\mathbb P^n$ into the fiber
above $\quotient{\mathbf x}$, namely
\[
\defun {S_{\mathbf f,\mathbf x}}{T_{\quotient{\mathbf x}}\mathbb P^n = {\mathbf x^{\perp}}}{\pi^{-1}( \quotient{\mathbf x})
\cong \mathbb C^ n}
{\dot {\mathbf x}}{ f_1(\mathbf x+\dot {\mathbf x}), \dots, f_n(\mathbf x+\dot {\mathbf x})}
\]
\medskip
\par
Since this $S_{\mathbf f,\mathbf x}$ is a function between linear spaces, 
we can 
define the local Newton operator associated to $s_{\mathbf f}$ as the Newton operator
for $S_{\mathbf f,\mathbf x}$:
\[
\defun{ {\mathbf N}_{\mathbf f,\mathbf x}}{T_{\quotient{\mathbf x}}\mathbb P^n}{T_{\quotient{\mathbf x}}\mathbb P^n}
{\dot {\mathbf x}}
{\dot {\mathbf x} - \left(D\mathbf f(\mathbf x+\dot {\mathbf x})_{|{{\mathbf x}^{\perp}}}\right)^{-1} \mathbf f(\mathbf x + \dot {\mathbf x})}
\]
The projective Newton operator is
\[
\defun{ {\mathbf N}^{\proj}_{\mathbf f}}{\mathbb P^n}{\mathbb P^n}{\quotient{\mathbf x}}
{\quotient{\mathbf x+ {\mathbf N}_{\mathbf f,\mathbf x}(0)}.}
\]

\begin{remark}
Explicit expressions for the projective Newton operator are
\[
 {\mathbf N}^{\proj}_{\mathbf f}(\quotient{\mathbf x}) = 
\quotient{\mathbf x - D\mathbf f(\mathbf x)_{|\mathbf x^{\perp}}^ {-1} \mathbf f(\mathbf x) }
=
\quotient{
\mathbf x - 
\begin{pmatrix}D\mathbf f(\mathbf x) \\ \mathbf x^*\end{pmatrix}^{-1}
\begin{pmatrix}\mathbf f(\mathbf x) \\ 0\end{pmatrix}
}
.
\]
\end{remark}

\subsection{Alpha theory}

Smale's invariants for the projective Newton operator are
\[
\beta(\mathbf f,\quotient{\mathbf x}) = \frac{1}{\|\mathbf x\|}
\left\|D\mathbf f(\mathbf x)_{|\mathbf x^{\perp}}^{-1}
\mathbf f(\mathbf x) \right\|,
\]
\[
\gamma(\mathbf f,\quotient{\mathbf x}) = \|\mathbf x\|
\max_{k \ge 2} 
\left(\frac{1}{k!}\left\|D\mathbf f(\mathbf x)_{|\mathbf x^{\perp}}^{-1}
D^k\mathbf f(\mathbf x)\right\|
\right)^{\frac{1}{k-1}}
\]
and of course $\alpha(\mathbf f,\quotient{\mathbf x}) = \beta(\mathbf f,\quotient{\mathbf x}) \gamma(\mathbf f,\quotient{\mathbf x})$. 

We will denote by $d(\quotient{\mathbf x},\quotient{\mathbf y})$ the Riemannian 
(Fubini-Study) distance in projective
space and by $d_T(\quotient{\mathbf x},\quotient{\mathbf y})=\tan d(\quotient{\mathbf x},\quotient{\mathbf y})$
the `tangential distance'. This is
not a metric, since the triangle inequality fails. However, if $\dot {\mathbf x} \perp \mathbf x$,
then
\[
d_T(\quotient{\mathbf x},\quotient{\mathbf x + \dot {\mathbf x}})
=
\frac{\| \dot {\mathbf x}\|}{\|\mathbf x\|}
\]
is the norm in $T_{\quotient{\mathbf x}}\mathbb P^n$. 

\begin{theorem}[$\gamma$-theorem]
Let $\quotient{\boldsymbol \zeta} \in \mathbb P^n$ be a non-degenerate zero of $f$.
If $\quotient{\mathbf x_0} \in \mathbb P^n$ satisfies
\[
d_T(\quotient{\boldsymbol \zeta},\quotient{\mathbf x_0}) \gamma(\mathbf f, \quotient{\boldsymbol \zeta}) \le \frac{3 - \sqrt{7}}{2},
\]
then the sequence $\quotient{\mathbf x_{i+1}} =  {\mathbf N}^{\proj}_{{\mathbf f}}(\quotient{\mathbf x_i})$ is well-defined
and
\[
d_T(\quotient{\boldsymbol \zeta}, \quotient{\mathbf x_i}) \le 2^{-2^i+1}d_T(\quotient{\boldsymbol \zeta}, \quotient{\mathbf x_0}) . 
\]
\end{theorem}

This first appeared in the book by \ocite{BCSS}*{Th.1 p.263}. 
\ocite{Burgisser-Cucker}*{Th.16.38} provided a 
refinement of this theorem, not necessary for this paper.
One can also state
an alpha-theorem for the projective Newton iteration, but 
the sharpest $\alpha_0$ constant seems to be unknown. Instead we can
apply Theorem~\ref{th-alpha} to the local Newton operator. 
\begin{theorem}[Tangential $\alpha$-theorem]\label{th-alpha-proj}
Let \[
\alpha 
\le
\alpha_0 = \frac{13 - 3 \sqrt{17}}{4}.
\]
Let 
\[
r_0 = 
\frac{1+\alpha-\sqrt{1-6\alpha+\alpha^2}}{4\alpha} 
\text{ and }
r_1 =
\frac{1-3\alpha-\sqrt{1-6\alpha+\alpha^2}}{4\alpha} 
.
\]

If $\quotient{{\mathbf x_0}} \in \mathbb P^n$ satisfies $\alpha(\mathbf f,\quotient{\mathbf x_0}) \le \alpha$,
then the sequence defined recursively by
$\dot {\mathbf x}_0=0$, $\dot {\mathbf x}_{i+1} =  {\mathbf N}_{\mathbf f,\mathbf x}(\dot {\mathbf x}_i)$
is well-defined and converges to a limit $\dot {\mathbf x}^*$ so that
$\quotient{\boldsymbol \zeta} \defeq \quotient{{\mathbf x_0}+ \dot {\mathbf x}^*}$ 
is a zero of $\mathbf f$. Furthermore,
\begin{enumerate}[(a)]
\item
$
\| \dot {\mathbf x}_i - \dot {\mathbf x}^* \| \le 2^{-2^i+1} \| \dot {\mathbf x}^*\|
$
\item 
$ d_T(\quotient{ {\mathbf x_0} + \dot {\mathbf x}_i},\quotient{\boldsymbol \zeta}) \le 
2^{-2^i+1} d_T(\quotient{{\mathbf x_0}}, \quotient{\boldsymbol \zeta})$
\item
$
d_T(\quotient{\mathbf x_0},\quotient{\boldsymbol \zeta}) \le r_0 \beta(\mathbf f, \quotient{\mathbf x_0})
$
\item
$
d_T( {\mathbf N}^{\proj}_{\mathbf f}(\quotient{\mathbf x_0}),\quotient{\boldsymbol \zeta}) 
\le r_1 \beta(\mathbf f, \quotient{\mathbf x_0}).
$
\end{enumerate}
\end{theorem}
We will need to borrow Lemma 2(4) p.264 from \ocite{BCSS}. 
Since I am not satisfied with the published proof,
I included an alternate one in the appendix.
\begin{lemma}\label{plane-swap}
Suppose that $\mathbf x, \mathbf y, \boldsymbol \zeta \in \mathbb C^{n+1}$ with
$\boldsymbol \zeta-\mathbf x \perp \mathbf x$, $\mathbf y-\mathbf x \perp \mathbf x$ and $\|\mathbf y-\boldsymbol \zeta\|\le\|\mathbf x-\boldsymbol \zeta\|$.
Then,
\[
\frac{\| \pi(\mathbf y)-\boldsymbol \zeta \|}{\|\boldsymbol \zeta\|}
\le
\frac{\|\mathbf y-\boldsymbol \zeta\|}{\|\mathbf x\|}
\]
where $\pi(\mathbf y) = 
\frac{\|\boldsymbol \zeta\|^2}{\langle \mathbf y,\boldsymbol \zeta \rangle} \mathbf y
$ is the radial projection
onto the affine plane $\boldsymbol \zeta+\boldsymbol \zeta^{\perp}$.
\end{lemma}

\begin{proof}[Proof of Theorem \ref{th-alpha-proj}]
Item (a) is Theorem~\ref{th-alpha} in $\mathbf x_0 + \mathbf x_0^ \perp \cong \mathbb C^n$. 
Item (b)
is a particular case of the Lemma~\ref{plane-swap} above, namely
\[
d_T(\quotient{{\mathbf x_0}+\dot {\mathbf x}_i},
\quotient{\boldsymbol \zeta}) 
=
\frac{\| \pi(\mathbf y)-\boldsymbol \zeta \|}{\|\boldsymbol \zeta\|}
\le 
\frac{\|\mathbf y-\boldsymbol \zeta\|}{\|\mathbf x\|}
=
\|\dot {\mathbf x}_i -
\dot {\mathbf x}^*\|/ \| \mathbf x_0\|
\]
for $\mathbf x =
\mathbf x_0$, $\mathbf y={\mathbf x_0}+\dot {\mathbf x}_1$ and $\boldsymbol \zeta=\mathbf x+\dot {\mathbf x}^ *$. 
Items (c) and (d) follow from Theorem ~\ref{th-alpha}{(b,c)}
and from estimates
\[
d_T(\quotient{\mathbf x_0}, \quotient{\boldsymbol \zeta)}) = \| \dot {\mathbf x}^*\|{/\| \mathbf x_0\|}
\hspace{2em}
\text{ and }
\hspace{2em}
d_T(\quotient{\mathbf x_1}, \quotient{\boldsymbol \zeta)}) \le \| \dot {\mathbf x}_1 - \dot {\mathbf x}^*\|{/\| \mathbf x_0\|}
\]
the last one as above with $\mathbf y=\mathbf x+\dot {\mathbf x}_1=\mathbf x_1$.
\end{proof}

\subsection{Homotopy and the condition length}

Let $\mathscr H_{d}$ be the complex space of degree $d$ homogeneous
polynomials on $n+1$ variables, endowed with 
Weyl's
$U(n+1)$-invariant
inner product. Let $\mathscr H_{(d_1, \dots, d_n)}=
\mathscr H_{d_1} \times \cdots \times \mathscr H_{d_n}$.
The invariant condition number 
$\mu: \mathbb P(\mathscr H_{(d_1, \dots, d_n)}) \times
\mathbb P^n \rightarrow [\sqrt{n}, \infty]$
defined by \ocite{Bezout1}
is
\begin{equation}\label{proj-cond-number}
\mu(\mathbf f,\mathbf x) =
\|\mathbf f\|
\left\|
D\mathbf f(\mathbf x)_{\mathbf x^{\perp}} ^{-1}
\begin{pmatrix}
\|\mathbf x\|^{d_1-1} \sqrt{d_1}\\
&\ddots
\\
&&
\|\mathbf x\|^{d_n-1} \sqrt{d_n}
\end{pmatrix}
\right\|
\end{equation}
with the operator $2$-norm assumed. 
The minimum of $\mu(\mathbf f,\mathbf x)=\sqrt{n}$ is actually attained for 
$f_i(\mathbf x) = \sqrt{d_i} \mathbf x_{0}^{d_i-1} \mathbf x_i$ at $\mathbf x = \mathrm e_0$.
At this system, $\|f_i\|=1$ in Weyl's metric and therefore $\|\mathbf f\|=\sqrt{n}$.
The main complexity result that we want to emulate is:
\begin{theorem}\cite{Bezout6}*{Th.3}
There is a constant $C_1 > 0$, such that: if $(\mathbf f_t , \mathbf z_t)$, $t_0 \le t \le
t_1$ is a $\mathcal C^1$ path in $\mathscr S_0=\{([\mathbf f],[\mathbf z]): \mathbf f(\mathbf z)=0\}$, 
then
\[
C_1 (\max d_i)^{3/2} 
\int_{t_0}^{t_1}   
\mu(\mathbf f_t,\mathbf z_t) \sqrt{\| \dot {\mathbf f}_t \|_{\mathbf f_t}^2 + \| \dot {\mathbf z}_t\|_{\mathbf z_t}^2}
\ \dd t
\]
steps of the projective Newton method are sufficient to continue an 
approximate zero $\mathbf x_0$ 
of $\mathbf f_{t_0}$ with associated zero $\mathbf z_0$
to an approximate zero $\mathbf x_1$ of $\mathbf f_{t_1}$ with associated zero 
$\mathbf z_{t_1}$.
\end{theorem}

In the context of dense polynomial systems, the condition length
relates algorithmic issues to geometrical 
properties of the solution variety \cites{Bezout7,Boito-Dedieu,BDMS1,BDMS2}.
Adaptive algorithms exploiting the condition length were presented by
\ocite{Beltran-Leykin} and \ocite{Adaptive}. \ocite{Hauenstein-Liddell} obtained a similar algorithm for constant term homotopy. This allowed them to replace the condition number by Smale's $\gamma$ invariant in the definition of condition length. \ocite{ABBCS} used the condition
length complexity estimates to derive an average complexity result.
Condition metrics can also be studied for their own sake as in
\cites{BDMS1, CriadoDelRey}.

\begin{rexample}
We estimate now the condition length for the two solution paths in the 
example of equation \eqref{eq-example}.
Let $\mathbf Z_t = (X_t,Y_t,1)$ so that
\[
\mathbf Z_t^{(1)} = \begin{pmatrix}
t^{-2} \\ t^{-1} \\1 \end{pmatrix}
\hspace{1em}
\text{and}
\hspace{1em}
\mathbf Z_t^{(2)} = 
\begin{pmatrix}
- \frac{t^{2}+1}{2t}\\
-1 \\1 \end{pmatrix} .
\]
In norms,
\[
\|\mathbf Z_t^{(1)}\|^2 = t^{-4} + t^{-2} + 1
\hspace{1em}
\text{and}
\hspace{1em}
\|\mathbf Z_t^{(2)}\|^2 = \frac{1}{4}t^{-2} + \frac{5}{2} + \frac{1}{4}t^2 
\]
The Weyl norm of $\mathbf f_t$ satisfies
$\|\mathbf f_t\|^2 = \frac{13}{6} + \frac{1}{2}t^2 + t^4$.
Instead of evaluating the norm of 
$(D\mathbf f_t(\mathbf Z^{(i)}(t))_{(\mathbf Z^{(i)}(t))^{\perp}})^{-1}$, 
we compute
{\small
\[
D\mathbf f_t(\mathbf Z^{(1)}(t)) D\mathbf f_t(\mathbf Z^{(1)}(t))^* =  
\begin{pmatrix}
t^2-2t+10-6t^{-1}+5t^{-2} &
t+5t^{-1} +4t^{-2}+2t^{-4} \\
t+5t^{-1}+4t^{-2}+2t^{-4} &
1+2t^{-1}+5t^{-2}+8t^{-3}+5t^{-4}+2t^{-5}+t^{-6}
\end{pmatrix}
\]}
and
{\small
\[
D\mathbf f_t(\mathbf Z^{(2)}(t))D\mathbf f_t(\mathbf Z^{(2)}(t))^*   =  
\begin{pmatrix}
\frac{17}{2}t^4 + 13t^2+\frac{5}{2}
& 2t^3 + 4t^2+3t+2+t^{-1} \\
2t^3+4t^2+3t+2+t^{-1} &
\frac{1}{2}t^2+2t+3+2t^{-1}+\frac{1}{2}t^{-2}
\end{pmatrix}.
\]
}
Since 
\[
(\mu^{(i)})^ 2 = 
\left( \mu^{(i)}(\mathbf f_t, \mathbf Z^{(i)}(t)) \right)^2
= 3\|\mathbf Z^{(i)}(t)\|^4\|\mathbf f_t\|^2
\left\| \left(D\mathbf f_t(\mathbf Z^{(i)}(t))D\mathbf f_t(\mathbf Z^{(i)}(t))^*\right)^{-1} \right\|
,
\]
we first expand the inverse of
\[
D\mathbf f_t(\mathbf Z^{(i)}(t))D\mathbf f_t(\mathbf Z^{(i)}(t))^* 
\]
into its Laurent series around zero using the Maxima computer
algebra system \cite{MAXIMA}. 
The condition length for
paths $(\mathbf f_t,\mathbf Z_t^{(1)})$ and ,$(\mathbf f_t,\mathbf Z_t^{(2)})$ is computed in
Table~\ref{table-cond-length-h}. Overall, the condition length $L$ 
satisfies
\[
L((\mathbf f_t,\mathbf Z_t^{(1)}), \epsilon, 1)
=
\int_{\epsilon}^{1}   
\mu(\mathbf f_t,\mathbf z_t) \sqrt{\| \dot {\mathbf f}_t \|_{\mathbf f_t}^2 + \| \dot {\mathbf z}_t\|_{\mathbf z_t}^2}
\ \dd t
 \in \Theta(\epsilon^{-2})
\]
as claimed in the introduction. This is also the best known
upper bound for the number of projective Newton steps in a homotopy
algorithm going from $f_1$ to $f_{\epsilon}$.

\begin{table}
\centerline{
\begin{tabular}{||c||c|c||}
\hline \hline
$i$ & $1$ & $2$ \\
\hline \hline
$d_{1,2}$ & $3$ & $3$ \\
$\|\mathbf Z^{(i)}(t)\|^4$ & $t^{-8} + O(t^{-6})$ & $\frac{1}{16}t^{-4} + O(t^{-3})$\\
$\|\mathbf f_t\|^2$ & $\frac{13}{6} + O(t^2)$ & $\frac{13}{6} + O(t^2)$ \\ 
$\| \left(D\mathbf f_t D\mathbf f_t^*\right)^{-1} \|$ & 
{$t^2 +O(t^3)$}
 & 
{$2 +O(t)$}
\\
$(\mu^{(i)})^2$ & 
{$\frac{13}{2} t^{-6} + O(t^{-5})$}
& 
{$\frac{13}{16}t^{-4} + O(t^{-3})$}
\\
\hline
$\left\| \frac{\partial}{\partial t} \mathbf f_t \right\|_{\mathbf f_t}^2$ 
&$\frac{3}{13}+O(t^2)$ & $\frac{3}{13}+O(t^2)$\\
$\left\| \frac{\partial}{\partial t} \mathbf Z^{(i)}(t) \right\|_{\mathbf Z^{(i)}(t)}^2$
&
{$1+O(t^2)$} 
& 
{$8+O(t^2)$}
\\
$\left( \mu^{(i)} \left\|  \frac{\partial}{\partial t} 
\left(\mathbf f_t, \mathbf Z^{(i)}(t)\right) \right\|_{\left(\mathbf f_t, \mathbf Z^{(i)}(t)\right)}\right)^2$ 
&
{$8t^ {-6}+O(t^{-5})$}
&
{$\frac{107}{16} t^ {-4}+O(t^ {-3}) $}
\\
\hline
$\mu^ {(i)} \left\| 
\frac{\partial}{\partial t} 
\left(\mathbf f_t, \mathbf Z^{(i)}(t)\right) 
\right\|_{\left(\mathbf f_t, \mathbf Z^{(i)}(t)\right) }
$ &
{$\sqrt{8}t^{-3} + O(t^{-2})$} 
& 
{$\sqrt{\frac{107}{16}}t^{-2} + O(t^{-1})$}
\\
\hline \hline
$L((\mathbf f_t,\mathbf Z_t^{(i)}); \epsilon,1) = \int_{\epsilon}^1 \mu^{(i)} \|
\cdots\|\dd t$ 
& 
{$\sqrt{2}\epsilon^{-2} + O(\epsilon^{-1})$} 
& 
{$\sqrt{\frac{107}{16}}\epsilon^{-1} + O(\log(\epsilon^{-1}))$}
\\
\hline \hline
\end{tabular}
}
\caption{\label{table-cond-length-h}Computation of the condition length
in the homogeneous setting.}
\end{table}
\end{rexample}

\section{Toric Newton iteration, condition and homotopy}
\label{main-results}

The two solution
paths for equation \eqref{eq-example} from the running example
converge to the same point in projective space. Indeed, the
solution paths 
$(t^{-2},t^{-1})$ and $(-\frac{1+t^2}{2t},-1)$
correspond to solution paths
$[1:t:t^2]$ and $[1+2t^2:2t:-2t]$ in $\mathbb P^2$. When $t=0$
they converge to the same point. In this section we will embed
the solution paths in $\mathbb P^3$ instead of $\mathbb P^2$.
For instance, we consider the embedding 
$(X,Y) \mapsto [X:XY:Y^2:Y^3]$. Under this embedding, the
solution paths become
$[t:1:t:1]$ and $[-(1+t^2):(1+t^2):2t:2t]$. When $t \rightarrow 0$,
those solutions converge to $[0:1:0:1]$ and $[-1:1:0:0]$.

It turns out that sparse polynomial systems are better studied as
spaces of exponential sums with integer coefficients. This amounts
to representing the solutions in logarithmic coordinates. If
$s = -\log(t)$,
\[
\lim_{s \rightarrow \infty} 
\frac{1}{s}
\log 
\begin{pmatrix}
t^{-2} \\
t^{-1} 
\end{pmatrix}
= 
\begin{pmatrix}
2 \\
1
\end{pmatrix}
s
+
o(s)
\hspace{1em}
\text{and}
\hspace{1em}
\lim_{s \rightarrow \infty} 
\log 
\begin{pmatrix}
-\frac{1+t^2}{2t}\\-1
\end{pmatrix} = 
\begin{pmatrix}
1 \\
0
\end{pmatrix}
s
+o(s).
\]
The vectors $\begin{pmatrix}
2 \\
1
\end{pmatrix}
$ and $\begin{pmatrix}
1 \\
0
\end{pmatrix}$ are outer normals to the support polygon,
whose vertices are $(1,0)$,$(1,1)$,$(0,2)$ and $(0,3)$.
See Fig.~\ref{fig:momentum}.

\subsection{Spaces of complex fewnomials}
The group action that we will introduce in this section requires us
to take an extra step. We are required to allow for spaces of 
exponential sums with {\em real} exponents. 
All those spaces are particular examples of a more
general class 
of function spaces with an inner product,
studied by
\ocite{Malajovich-Fewspaces} 
in connection with a generalization of
the theorem by \ocite{BKK}. We will need here the basic definitions
and the reproducing kernel properties. 
\begin{definition}\label{fewspaces}
A fewnomial space $\mathscr F$ of functions over
a complex manifold $\mathscr M$ is a Hilbert space of holomorphic functions from $\mathscr M$
to $\mathbb C$, such that the evaluation form 
\[
\defun{V}{\mathscr M}{\mathscr F^*}{\mathbf x}{V(\mathbf x)
\text{ such that } V(\mathbf x) (f) = f(\mathbf x)}
\]
satisfies:
\begin{enumerate}
\item[i.] For all $\mathbf x \in \mathscr M$, $V(\mathbf x)$ is a continuous linear form.
\item[ii.] For all $\mathbf x \in \mathscr M$, $V(\mathbf x)$ is not the zero form.
\suspend{enumerate}
The fewnomial space $\mathscr F$ is said to be {\em non-degenerate} 
if and only if, \noindent \par
\resume{enumerate}
\item[iii.]
For all $\mathbf x \in \mathscr M$, the composition of $DV(\mathbf x)$ with the orthogonal projection onto $V(\mathbf x)^{\perp}$ has full rank.
\end{enumerate}
\end{definition}

Fewnomial spaces are reproducing kernel spaces, with 
reproducing kernel $K(\mathbf x,$ $\mathbf y) = V(\mathbf x)
(V(\mathbf y)^*)$. 
The pull-back of the Fubini-Study metric in $\mathbb P(\mathscr F^*)$
defines a Hermitian
structure on $\mathscr M$, denoted by  $\langle \cdot , \cdot \rangle_{\mathscr F,\mathbf x}$. 
Below are a few examples.

\begin{example}[Bergman space]
Let $\mathscr M \subset \mathbb C^n$ be open and bounded.
Let $\mathcal A(\mathscr M)$ be the space of holomorphic functions defined on $\mathscr M$ with
finite $\mathscr L^2$ norm, endowed with the $\mathscr L^ 2$ inner product.
Then $\mathcal A(\mathscr M)$ is a non-degenerate fewnomial space.
\end{example}
\begin{example}
Let $\mathscr M = \mathbb C^{n+1}\setminus \{0\}$. Let $\mathscr H_d$ be the
space of homogeneous polynomials on $\mathscr M$ of degree $d$, endowed with
the $U(n+1)$-invariant inner product. Then 
$\mathscr H_d$ is a non-degenerate fewnomial space. 
\end{example}
\begin{example}[Sparse polynomials]
Let $A \subset \mathbb Z^n$ be finite and let $\rho:A \rightarrow (0, \infty)$
be arbitrary. Let $\mathscr M = \mathbb C^n$.
Let $\mathscr P_{A}$ be the complex vector space spanned by monomials 
$\mathbf x^{\mathbf a}$,
endowed with the Hermitian inner product that makes $(\dots, \rho_{\mathbf a} \mathbf x^{\mathbf a}, \dots)_{\mathbf a \in A}$ 
an orthonormal basis.
Then $\mathscr P_{A}$ is a (possibly degenerate) fewnomial space.
\end{example}
\begin{example}[Exponential sums, integer coefficients]
Let $A \subset \mathbb Z^n$ be finite and let $\rho:A \rightarrow (0, \infty)$
be arbitrary. Let $\mathscr M = \mathbb C^n \mod 2 \pi \sqrt{-1}\ \mathbb Z^n$.
Let $\mathscr F_{A}$ be the complex vector space with orthonormal basis
$(\dots, \rho_{\mathbf a} e^{\mathbf a \mathbf x}, \dots)_{\mathbf a \in A}$ 
Then {$\mathscr F_{A}$} is a fewnomial space. 
Interest arises because
if $f=f(\mathbf z) \in \mathscr P_{A}$, then $f \circ \exp \in \mathscr F_{A}$.
\end{example}
\begin{example}[Exponential sums, real coefficients]
Let $A \subset \mathbb R^n$ be finite and let $\rho:A \rightarrow (0, \infty)$
be arbitrary. Let $\mathscr M = \mathbb C^n$.
Let $\mathscr F_{A}$ be the complex vector space with orthonormal basis
$(\dots, \rho_{\mathbf a} e^{\mathbf a \mathbf x}, \dots)_{\mathbf a \in A}$
\end{example}

\begin{remark}
While in this paper we take the $\rho_{\mathbf a}$ as arbitrary, 
there is a natural
product operation on the set of all fewnomial spaces that induces
specific choices, see \cite{Malajovich-Fewspaces}. 
\end{remark}

\subsection{Group actions and the momentum map}
Arguably, the most important tool in the theory of 
homotopy algorithms for
homogeneous polynomial systems is the invariance by 
$U(n+1)$-action.  We cannot use this technique here. Thus we need
an alternative tool.

The additive group $((\mathbb R^n)^*,+)$ acts on 
the set of {\bf all} exponential sums
by
\[
{\mathbf g} , \sum_{\mathbf a \in A} f_{\mathbf a} \rho_{\mathbf a} e^{\mathbf a \mathbf x} 
\mapsto
{\mathbf g} \left(\sum_{\mathbf a \in A} f_{\mathbf a} \rho_{\mathbf a} e^{\mathbf a \mathbf x}\right) 
\defeq \sum_{\mathbf a \in A} f_{\mathbf a} \rho_{\mathbf a} e^{(\mathbf a-{\mathbf g}) \mathbf x} = 
e^{-{\mathbf g}\mathbf x}  \sum_{\mathbf a \in A} f_{\mathbf a} \rho_{\mathbf a} e^{\mathbf a \mathbf x}
.
\]
This is equivalent to shifting the support of an exponential
sum, sending $\mathscr F_{A}$ to $\mathscr F_{A-\mathbf g}$ where
$A-\mathbf g \defeq \{\mathbf a-\mathbf g:\mathbf a \in A\}$. Shifting sends each basis vector
$\rho_{\mathbf a} e^{\mathbf a\mathbf x}$ of $\mathscr F_{\mathbf a}$ 
into a basis vector $\rho_{\mathbf a'} e^{(\mathbf a-{\mathbf g}) \mathbf x}$ of $\mathscr F_{A-\mathbf g}$. 
We require the $\rho_{\mathbf a'}$'s to be proportional to the $\rho_{\mathbf a}$'s.
This restriction amounts to say that 
the group acts by homothety.
The Hermitian structure in $\mathscr F_{A-{\mathbf g}}$ 
is therefore the same (up to a constant) than
the pull-forward of the Hermitian structure of $\mathscr F_{\mathbf a}$.

For each ${\mathbf g} \in (\mathbb R^n)^*$, define
\[
\defun{W_{{\mathbf g}}}{\mathbb C^n}{\mathscr F^*}
{\mathbf x}{W_{\mathbf g}(\mathbf x)=e^{-{\mathbf g} \mathbf x} V(\mathbf x)}
\]
and notice that always $[V(\mathbf x)]=[W_{\mathbf g}(\mathbf x)]$. The metric obtained by pulling
Fubini-Study metric from $\mathbb P(\mathscr F_{A})$ or from
$\mathbb P(\mathscr F_{A-{\mathbf g}})$
is exactly the same. 
From the point of view of this paper, $V$ and $W_{{\mathbf g}}$ and undistinguishable.
\begin{remark}
Properly speaking, a group acts on a set. Here, the set is the disjoint
union of all the complex fewnomial spaces over $\mathbb C^n$.
\end{remark}
\medskip
\par
\begin{figure}
\centerline{\resizebox{\textwidth}{!}{\includegraphics{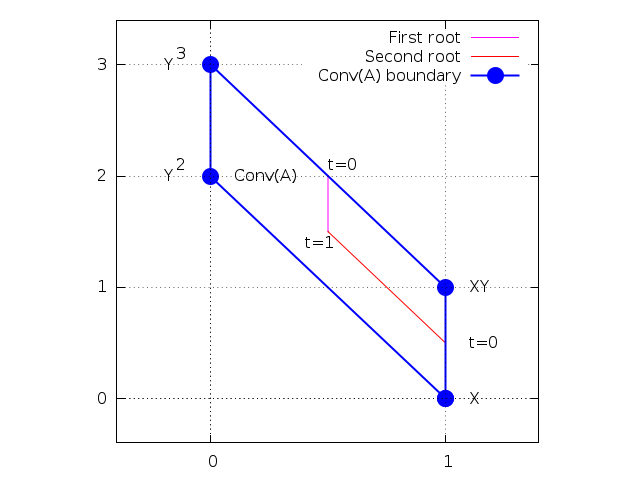}}}
\caption{The momentum map for the two solutions of the running example
between $t=1$ (center) and $t=0$ (on the boundary). \label{fig:momentum}}
\end{figure}
A particular choice of ${\mathbf g} \in (\mathbb R^n)^*$ plays
the rôle of the canonical basis in the $U(n+1)$-invariant 
homogeneous theory.
This particular choice is related to an invariant of the toric action
on $\mathbb C^n$:
each $\theta \in (S^1)^n=
\mathbb R^n \bmod \mathbb Z^n$ maps $\mathbf x$ to $\mathbf x+2\pi \theta \sqrt{-1}$. 
The reproducing kernel $K(\mathbf x,\mathbf x)$ is invariant through this action,
and the Hermitian metric happens to be equivariant. 
The {\em momentum map} associated to the toric action is
\[
\defun{\mathbf m}{\mathbb C^n}{\conv{A} \subseteq (\mathbb R^n)^*}
{\mathbf x}{\mathbf m(\mathbf x) = \frac{1}{2} D\log(K(\mathbf x,\mathbf x)) = \frac{1}{\|V(\mathbf x)\|^2} V(\mathbf x)^* DV(\mathbf x) }
.
\]
At each point $\mathbf x$, the momentum map $\mathbf m(\mathbf x)$ is also a convex linear 
combination of the points
in $A$. Points at toric infinity map to points on the boundary of
$\conv{A}$ (Figure~\ref{fig:momentum}). 

At a fixed point $\mathbf x_0 \in \mathbb C^n$, we set $\mathbf g = 
 \mathbf m {\defeq} \mathbf m(\mathbf x_0)$ and
$W(\mathbf x) \defeq W_{\mathbf m}(\mathbf x) = e^{-\mathbf m\mathbf x} V(\mathbf x)$. 
The derivative of
each $[V(\mathbf x)]$ at $\mathbf x_0$ can be written in normalized coordinates as
\[
\defun{D[V](\mathbf x_0)}{T_{\mathbf x_0}\mathbb C^n}{T_{[V(\mathbf x_0)]} \mathscr F^*}
{\dot {\mathbf x}}{
\frac{1}{\|V(\mathbf x_0)\|}
\left(I - \frac{1}{\|V(\mathbf x_0)\|^2}V(\mathbf x_0)V(\mathbf x_0)^*\right) DV(\mathbf x_0) \dot {\mathbf x}
}
\]
while introducing $W=W_{\mathbf m}$ one has $W^*(\mathbf x_0) DW(\mathbf x_0)=0$ so
\[
D[V](\mathbf x_0) = D[W](\mathbf x_0): \dot {\mathbf x} \mapsto \frac{1}{\|W(\mathbf x_0)\|}DW(\mathbf x_0) \dot {\mathbf x}
.
\]

The Lemma below also allows us to assume without loss of generality that
$\mathbf m(\mathbf x_0)=0$ at some special point $\mathbf x_0$. 

\begin{lemma}\label{action-on-m}
On a neighborhood of $\mathbf x_0$, define $\hat {\mathbf m}(\mathbf x) = 
\frac{1}{\|W(\mathbf x)\|^2} W^*(\mathbf x) DW(\mathbf x)$.
Then $\hat {\mathbf m}(\mathbf x) = \mathbf m(\mathbf x) - \mathbf m(\mathbf x_0)$.
\end{lemma}
\begin{proof}
We use the formula $\mathbf m(\mathbf x) = \frac{1}{2} D\log(K(\mathbf x,\mathbf x))$. 
The reproducing kernel associated to $W$ is $K(\mathbf x,\mathbf x) e^{-2\mathbf m(\mathbf x_0) \mathrm{Re}(\mathbf x)}$
so $\hat {\mathbf m}(\mathbf x) =\mathbf  m(\mathbf x) - \mathbf m(\mathbf x_0)$.
\end{proof}

\subsection{Systems of equations}\label{subsec:hermitian}
From now on, we assume that each $\mathscr F_i = \mathscr F_{A_i}$ is a 
finite dimensional space of exponential sums over $\mathbb C^n$,
with 
orthonormal basis 
\[
( \dots, \rho_{\mathbf a} e^{\mathbf a \mathbf x}, \dots)_{\mathbf a \in A_i}
\]
where the coefficients $\rho_{\mathbf a} > 0$ are arbitrary.
The evaluation map for each $\mathscr F_i$ will be denoted by $V_i$ and its
reproducing kernel by $K_i(\mathbf x,\mathbf y)$. 
Let $\mathscr V_0 \subset \mathbb P(\mathscr F_1^*) \times \cdots
\times \mathbb P(\mathscr F_n^*)$ be the image of $\quotient{V} =
\left(\quotient{V_1}, \dots, \quotient{V_n}\right)$. Let $\mathscr V=
\overline{\mathscr V}_0$
be the Zariski closure of $\mathscr V_0$. Points at $\mathscr V \setminus
\mathscr V_0$ are said to be {\em at toric infinity}.

Let $\langle \cdot , \cdot \rangle_{\mathbf x}$ be the
pull-back by $[V]$ at $x$ of the Fubini-Study Hermitian product on
$\mathscr V_0 \subset \mathbb P(\mathscr F_1) \times \cdots \times \mathbb P(\mathscr F_n)$.
Namely,
\[
\langle \cdot , \cdot \rangle_{\mathbf x} =
\langle \cdot , \cdot \rangle_{1,\mathbf x} +
\cdots +
\langle \cdot , \cdot \rangle_{n,\mathbf x}
\]
and $\langle \mathbf u,\mathbf u \rangle_{i,\mathbf x} \le \langle \mathbf u,\mathbf u \rangle_{\mathbf x}$ for all $\mathbf u$, where $\langle \cdot , \cdot \rangle_{i,\mathbf x}$ and $\|\cdot\|_{i,\mathbf x}$ are the Hermitian inner product and norm
associated to the $i$-th space $\mathscr F_i$. 
A metric structure on $\mathscr V$ is given by
the induced norm for the Hermitian inner product,
\[
\| \cdot \|_{\mathbf x} = \sqrt{\langle \cdot , \cdot \rangle_{\mathbf x}}.
\]
This is not the only possibility. In Section~\ref{sec:Finsler}
we replace this norm on $\mathcal V$ with the Finsler structure 
$\vvvert \cdot \vvvert = \max_i \|\cdot\|_{\mathbf x,i}$.

It is convenient to parameterize $\mathscr V_0 \subset \mathscr V$ 
through an isometric chart. 
Let $\mathbb C^n/[V]$ be the quotient obtained by 
identifying two points of $\mathbb C^n$ whenever they have the same
image by $[V]$. Let
$\mathscr M = 
(\mathbb C^n/[V], \langle \cdot , \cdot \rangle_{\mathbf x})$.
\begin{lemma}
$\mathscr M$ is a Hermitian manifold, isometric to 
$\mathscr V_0$.
\end{lemma}
\begin{proof}
Without loss of generality, assume that each $A_i \ni 0$. Let
$N$ be the space of all $\mathbf u \in \mathbb C^n$ such that
$\mathbf a \mathbf u = 0$ for all $a \in A_i$, $i=1, \dots, n$.
Let $W$ be such that $\mathbb C^n = N \oplus W$. Then $\mathbb C^n/[V]$
and $W/[V]$ are the same.

Two points $\mathbf x$ and $\mathbf z \in W$ share the same
image by $[V]$ if and only if there are constants $c_1, \dots, c_n \in \mathbb C$ so that for any $i$ and for any $\mathbf a \in A_i$,
\[
e^{\mathbf a (\mathbf x - \mathbf z)} = e^{c_i}
.
\]
For all $\mathbf a \in A_i$ we will have
\[
\mathbf a (\mathbf x - \mathbf z) \equiv c_i \mod 2 \pi \sqrt{-1}
\]
Since $0 \in A_i$, we can take $c_i = 0$. 

By construction of $W$, there is a subset $\{\mathbf a_1, \dots, \mathbf a_r\}$
of $\cup A_i$
that is a basis of $W$ as a complex vector space.
Let $W_{\mathbb R} = \{ \Re(u): u \in W\}$ be the real projection
of $W$. Since the $\mathbf a_j$ are real vectors,
the same subset of $\cup A_i$ is a basis of the real vector space
$W_{\mathbb R}$.
As a consequence 
\[
\Lambda = \left\{ \mathbf u \in W: \mathbf a \mathbf u \equiv 0  \mod 2 \pi 
\right\}
\]
is an $r$-dimensional lattice. As a topological space, 
$\mathcal M$ is the quotient
of $W$ by the equivalence relation 
\[
\mathbf x \equiv \mathbf y
\Leftrightarrow \mathbf x - \mathbf y = \mathbf u \sqrt{-1}
\text{ for some } \mathbf u \in \Lambda
.
\]
Therefore $\mathscr M = W_{\mathbb R} \times W_{\mathbb R}/\Lambda$ is a smooth complex manifold of
dimension $r$. The isometry property follows from the construction
of the inner product.
\end{proof}
\begin{remark}
Most theorems
in this paper assume or imply the existence of nondegenerate roots,
so that 
the mixed volume $V(\conv{A_1}, \dots, \conv{A_n})$ does not vanish.
In particular there is a mixed cell. Above, we can
make this mixed cell to be in the form $[0, \mathbf a_1] \times
\cdots \times [0,\mathbf a_n]$ so that $(\mathbf a_1, \dots, \mathbf a_n)$
is a basis for $W$ with $\mathbf a_i \in A_i$. In this case,
$\mathscr M$ is a $n$-dimensional Hermitian manifold.
See~\cite{Malajovich-Mixed} for details and references on mixed volume,
mixed cells and such.
\end{remark}
\begin{remark}
The Lemma above can also be restated in terms of 
{\em non-degenerate} fewnomial spaces. If one of the
$\mathcal F_{A_i}$ is non-degenerate and $0 \in A_i$,
then $A_i$ contains a basis for $\mathbb R^n$, etc...
\end{remark}
\begin{remark}
While $\mathscr M$ is also a smooth manifold, the closure $\mathscr V$
of $\mathscr V_0$ is not necessarily smooth. Just consider the span of
$e^{3x}$, $e^{2x}$ and $1$. Then $\mathscr V$ is the projective curve
$Y^2Z-X^3 = 0$ which has a singularity at $(0:0:1)$.
\end{remark}
\medskip
\par
As in the previous section, a system $(f_1, \dots, f_n) \in
\mathscr F_1, \dots, \mathscr F_n$ does not have a well-defined 
value at some $([V(\mathbf x)])$. Instead, it defines a section of 
the vector bundle $\pi:\mathscr E \rightarrow \mathbb P(\mathscr F_1^*) 
\times \dots 
\times \mathbb P(\mathscr F_n^*)$ with total space
\[
\mathscr E = 
\quotient{
(\mathscr F_1^* \setminus \{0\}) \times \dots \times (\mathscr F_n^*  \setminus \{0\}) \times \mathbb C^n 
}
\]
where the quotient is taken with respect to the 
${\mathbb C_{\times}^n}$-action 
\[
\lambda (\mathbf V,\mathbf y) = ( \lambda_1 V_1, \dots, \lambda_n V_n, \lambda_1 y_1, \dots, \lambda_n y_n).
\]

This bundle restricts to a vector bundle 
$\bundle{\mathbb C^ n}{\pi^{-1}(\mathscr V_0) \subseteq \mathscr E}{\pi}{\mathscr V_0}$, and pulls back to a bundle 
$\longbundle{\mathbb C^ n}{\mathscr E_0 =
\pi^{-1}(\mathscr V_0)}{[V]^ {-1} \circ \pi}{\mathscr M}$. 
The group $((\mathbb R^n)^*)^n$ acts coordinatewise on exponential sums: 
each $\mathbf M = (\mathbf m_1, \dots, \mathbf m_n) \in ((\mathbb R^n)^*)^n$ maps
$\mathscr F_{A_1} \times \cdots \times \mathscr F_{A_n}$ into
$\mathscr F_{A_1-\mathbf m_1} \times \cdots \times \mathscr F_{A_n-\mathbf m_n}$. 
\par

To define a local trivialization, fix an arbitrary $\mathbf x_0 \in \mathscr M$. Let
$U_0 = \{ \mathbf x \in \mathscr M: V_i(\mathbf x) \not \perp V_i(\mathbf x_0) \}$. 
Also, let $\mathbf m_i = \mathbf m_i(\mathbf x_0)$ be the momentum
map at $\mathbf x_0$. Let $W_i(\mathbf x) = e^{-\mathbf m_i(\mathbf x_0) (\mathbf x)}V_i(x)$. 
Then set
\[
\defun{ \phi_{\mathbf x_0}}{  U_0 \times \mathbb C^n }{\mathscr E_0}
{\mathbf x, \mathbf y}{\quotient{W(\mathbf x), \mathbf y}}
\]
To each $\mathbf f \in \mathscr F_1 \times \cdots \times \mathscr F_n$ we associate
the section 
\[
\defun{s_{\mathbf f}}{U_0 \subseteq \mathscr M}{\mathscr E_0}{\mathbf x_0+\dot{\mathbf x}}
{\phi_{\mathbf x_0} (\mathbf x_0+\dot{\mathbf x}, \mathbf f \cdot W(\mathbf x_0+\dot{\mathbf x}))}
\]
where the notation $\mathbf f \cdot \mathbf W$ stands for 
the map $\mathscr M \rightarrow \mathbb C^n$ given by
\[
\mathbf f \cdot \mathbf W =
\begin{pmatrix} f_1 \cdot W_1 \\ \vdots \\ f_n \cdot W_n\end{pmatrix} \defeq
	\begin{pmatrix} W_1(\mathbf x)(f_1) \\ \vdots \\ W_n(\mathbf x)(f_n)\end{pmatrix}.
\]
The local function is now
\[
\defun {S_{\mathbf f,\mathbf x_0}}{U_0 \subseteq T_{\mathbf x}\mathscr M}{\mathbb C^n}
{\dot{\mathbf x}}
{ 
\pi \circ \phi_{\mathbf x_0}^{-1} 
\left( \left[ \mathbf W(\mathbf x_0+\dot{\mathbf x}), \mathbf f \cdot \mathbf W(\mathbf x_0+\dot{\mathbf x}) \right]\right)
}
\]
where $\pi_2$ is the projection onto the second coordinate. 
In normalized coordinates,
\begin{equation}\label{local:function:V}
S_{\mathbf f,\mathbf x_0}(\dot{\mathbf x}) = 
\begin{bmatrix}
f_1 \cdot \left( \frac{1}{\|W_1(\mathbf x_0)\|} W_1(\mathbf x_0+\dot{\mathbf x})\right) \\
\vdots \\
f_n \cdot \left( \frac{1}{\|W_n(\mathbf x_0)\|} W_n(\mathbf x_0+\dot{\mathbf x})\right) 
\end{bmatrix}
\end{equation}

The local Newton operator is 
\[
\defun{ {\mathbf N}_{\mathbf f,\mathbf x_0}}{T_{\mathbf x_0}\mathscr M}
{T_{x_0}\mathscr M}
{\dot{\mathbf x}}{\dot{\mathbf x} - 
DS_{\mathbf f,\mathbf x_0}(\dot{\mathbf x}) ^{-1} 
S_{\mathbf f,\mathbf x_0}(\dot{\mathbf x})  
}
.
\]
In order to define a global Newton operator, one needs a map
from $T\mathscr M$ onto $\mathscr M$. We will use the sum from $\mathbb C^n$.
The map $(\mathbf x_0,\dot{\mathbf x}) \mapsto \mathbf x_0+\dot{\mathbf x}$ is the parallel
transport associated to the trivial (zero) connection on $\mathscr M$. 
The global Newton operator on $\mathscr M$ using that map is
\[
\defun{ {\mathbf N}_{\mathbf f}}{\mathscr M}{\mathscr M}{\mathbf x_0}{ \mathbf x_0 +  {\mathbf N}_{\mathbf f,\mathbf x_0}(0) .}
\]
If $ {\mathbf N}_{\mathbf f}(\mathbf x_0) \not \in \mathscr M$ we say that $ {\mathbf N}_{\mathbf f}(\mathbf x_0)$ is not defined. 
\medskip
\par
The group $((\mathbb R^n)^*)^n$ acts coordinatewise on exponential sums: 
each $\mathbf M = (\mathbf m_1,$ $ \dots, \mathbf m_n) \in ((\mathbb R^n)^*)^n$ maps
$\mathscr F_{A_1} \times \cdots \times \mathscr F_{A_n}$ into
$\mathscr F_{A_1-\mathbf m_1} \times \cdots \times \mathscr F_{A_n-\mathbf m_n}$. 
If we are given some $\mathbf x_0 \in \mathscr M$, we can always 
{\bf assume without loss of generality} that $\mathbf m_i(\mathbf x_0)=0$ for all $i$.
This simplifies the formulas for $\mathbf V$, $S_{\mathbf f,\mathbf x_0}$ and derivatives.
For instance,
\[
DS_{\mathbf f,\mathbf x_0}(0) =  
\begin{bmatrix}
f_1 \cdot \left( \frac{1}{\|V_1(\mathbf x_0)\|} DV_1(\mathbf x_0)\right) \\
\vdots \\
f_n \cdot \left( \frac{1}{\|V_n(\mathbf x_0)\|} DV_n(\mathbf x_0)\right) 
\end{bmatrix}
.
\]

\subsection{Condition number theory}

Assume now that $\mathbf m(\mathbf x_0)=0$, $\mathbf f \cdot \mathbf V(\mathbf x_0) = 0$ and 
$DS_{\mathbf f_0,\mathbf x_0}$ is non-degenerate. The implicit function
theorem asserts that there is a smooth function $G:U \subseteq
\mathbb P(\mathscr F_{1}) \times \cdots 
\times \mathbb P(\mathscr F_n) \rightarrow \mathscr M$ with
$\mathbf f \cdot \mathbf V(G(\mathbf f)) \equiv 0$, defined on a neighborhood $U \ni \mathbf f_0$.
Its derivative at $\mathbf f_0$ is 
\[
DG(\mathbf f_0) \dot{\mathbf f} = DS_{\mathbf f_0,\mathbf x_0}(0)^{-1} \dot{\mathbf f}(\mathbf x) .
\]
Using the reproducing kernel notation and assuming $\dot{\mathbf f}_i \perp f_i$,
\[
DG(\mathbf f_0) \dot{\mathbf f} = DS_{\mathbf f_0,\mathbf x_0}(0)^{-1} 
\frac{\|f_1\|}{\|K_1(\cdot, \mathbf x)\|} K_1(\cdot, \mathbf x)^*
\oplus
\cdots
\oplus
\frac{\|f_n\|}{\|K_n(\cdot, \mathbf x)\| } K_n(\cdot, \mathbf x)^*
.
\]
This motivates the following definition:

\begin{definition} The {\em toric condition number} of $\mathbf f$ at $\mathbf x$ is
\[
\mu(\mathbf f,\mathbf x) = \|DG(\mathbf f)\|_{\mathbf x} =\left\|
DS_{\mathbf f,\mathbf x}(0)^{-1}
\begin{pmatrix}
\|f_1\| \\
& \ddots \\
& & \|f_n\|
\end{pmatrix}
\right\|_{\mathbf x}
\]
where the operator norm from $\mathbb C^n$ (with canonical inner product)
into $(\mathscr M, \| \cdot \|_{\mathbf x})$ is assumed.
\end{definition}

The condition number is invariant through scaling of each of the
$f_i$. Therefore we also write $\mu([\mathbf f],\mathbf x) = \mu(\mathbf f,\mathbf x)$. Notice that
because of the normalization,
\begin{equation}\label{min:mu}
\mu(\mathbf f,\mathbf x) \ge 1
\end{equation}
always. 

A condition number theorem for $\mu(\mathbf f,\mathbf x)$ in terms of inverse distances
is known. In the language of this paper, it reads:
\begin{theorem}\cite{Malajovich-Rojas}*{Th.4} 
Let $\Sigma_{\mathbf x} = \{ \mathbf f: S_{\mathbf f,\mathbf x}(0)=0 \text{ and } \det{DS_{\mathbf f,\mathbf x}(0)} = 0 \}$.
Then,
\[
\max_{\|\dot{\mathbf f}\| \le \|\mathbf f\|}
\min_i
\left\|
DG(\mathbf f)\dot{\mathbf f} 
\right\|_{i,\mathbf x}
\le 
d_P(\mathbf f, \Sigma_{\mathbf x})^{-1}
\le
\mu(\mathbf f,\mathbf x)
\]
where $d_P$ is the projective (sine) metric.
\end{theorem}
\medskip
\par

\begin{figure}
\centerline{\resizebox{\textwidth}{!}{
\includegraphics{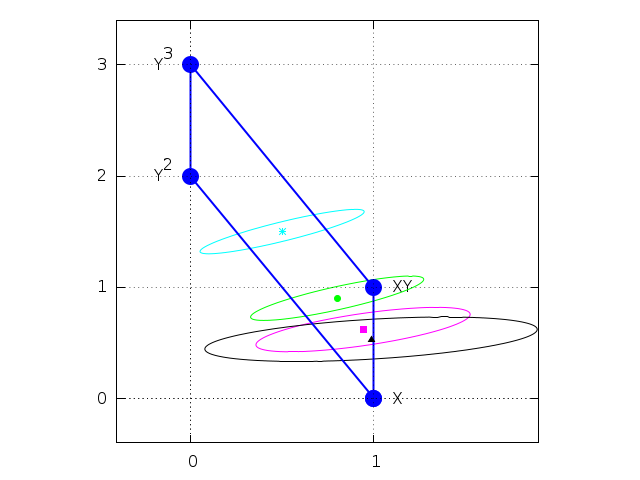}
\includegraphics{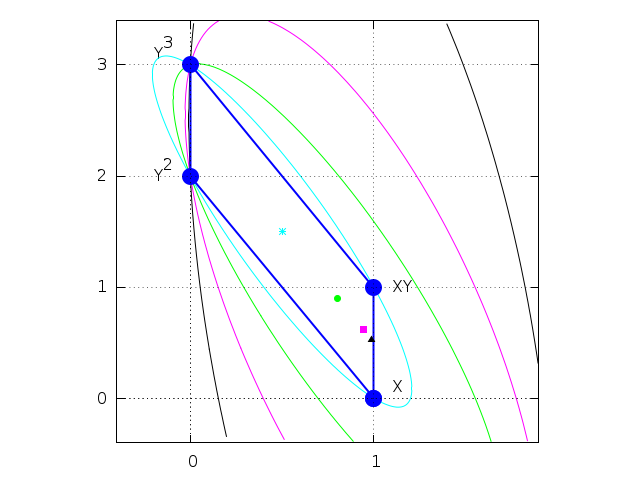}
}}
\caption{Left: Unit circles for the Hermitian metric 
$\langle \cdot, \cdot \rangle_{i,\mathbf x}$ from the running example, 
at several points.  The circles are centered  
at $\mathbf m_i(\mathbf x)$ and shrinked by a factor of 10 to fit in the picture.
Right: radius $\nu_i$ circles of the {\em dual} metric. Both
pictures are independent of the value of $i$.
\label{fig:unit-balls}}
\end{figure}

The condition numbers $\nu_i(\mathbf x)$ defined below
play an important rôle in this paper.

\begin{definition}
The {\em $\|\cdot\|_{i,\mathbf x}$ -- circumscribed radius of 
$\conv{A_i-\mathbf m_i(\mathbf x)}$} is
\[
\nu_{i}(\mathbf x) = 
\max_{\mathbf a \in A_i} \sup_{\|\mathbf u\|_{i,\mathbf x} \le 1} |(\mathbf a-\mathbf m_i(\mathbf x)) \mathbf u|
.
\]
Also, we set
\[
\nu(\mathbf x) = \max_i \nu_{i}(\mathbf x)
\]
\end{definition}

Figure~\ref{fig:unit-balls} shows the unit balls $\|\mathbf u\|_{i,\mathbf x} \le 1$ from the
running example at a few points. It also shows the radius $\nu_i(\mathbf x)$-balls
from the dual metric.

\begin{remark}
There is no guarantee that the unit ball for a $\|\cdot\|_{i,\mathbf x}$ is compact.
If $\mathrm{Span}( \mathbf a - \mathbf m_i(\mathbf x))$ is a proper subspace of $\mathbb R^n$, then
any vector $\mathbf u$ can be decomposed as $\mathbf u=\mathbf u_1 + \mathbf u_2$ with $\mathbf u_1 \in 
\mathrm{Span}( \mathbf a - \mathbf m_i(\mathbf x))$ and $\mathbf u_2 \perp \mathrm{Span}( \mathbf a - \mathbf m_i(\mathbf x))$.
In that case $\|\mathbf u\|_{i,\mathbf x}=\|\mathbf u_1\|_{i,\mathbf x}$ and 
$(\mathbf a - \mathbf m_i(\mathbf x)) \mathbf u = (\mathbf a - \mathbf m_i(\mathbf x))\mathbf u_1$.
\end{remark}
The reader should check that $1 \le \nu_i(\mathbf x)$ and that
\begin{equation}\label{lower:bound:nu}
\max_{\mathbf a \in A_i} \sup_{\|\mathbf u\|_{\mathbf x} \le 1} |(\mathbf a-\mathbf m_i(\mathbf x))\mathbf u|
\le \nu_i(\mathbf x) \le \nu(\mathbf x).
\end{equation}

As mentioned before, we are avoiding to use geodesics and parallel transport
to move from one point to another. Instead, we use the 
trivial transport operator
$\mathbf u \in T_{\mathbf x}\mathscr M \mapsto \mathbf u \in T_{\mathbf y}\mathscr M$. This operator is not isometric, but the
distortion it introduces can be bounded in terms of $\nu(\mathbf x)$:

\begin{lemma}\label{lem:metric}
Let $s = \nu(\mathbf x) \|\mathbf y-\mathbf x\|_{\mathbf x}$. Then for all $i$,
\[
(2-e^{s}) \|\mathbf u\|_{i,\mathbf x} \le \|\mathbf u\|_{i,\mathbf y} \le e^{s} \|u\|_{i,\mathbf x} 
.
\] 
Moreover,
\[
(2-e^{s}) \|\mathbf u\|_{\mathbf x} \le \|\mathbf u\|_{\mathbf y} \le e^{s} \|\mathbf u\|_{\mathbf x} 
.
\]
\end{lemma}
The exponential bounds above are not as inconvenient as they look. Typically,
$s$ is small. If $s<1$, $1+s \le e^s < 1/(1-s)$.

\medskip
\par

One of the main tools in recent homotopy papers such as
\cites{Bezout6, Bezout7, Adaptive, Burgisser-Cucker} is an 
estimate on the sensitivity of the condition number.
In this paper we will use the following bound instead:

\begin{theorem}\label{main-bound}
Assume that  
$\theta=(\|\mathbf x-\mathbf y\|_{\mathbf x} + d_P([\mathbf f],[\mathbf g])) \mu(\mathbf f,\mathbf x) \nu(\mathbf x) < 1/5$. 
Then,
\[
\mu(\mathbf f,\mathbf x)\nu(\mathbf x) (1-5\theta)
\le
\mu(\mathbf g,\mathbf y)\nu(\mathbf y) \le 
\frac{\mu(\mathbf f,\mathbf x)\nu(\mathbf x)}{1-5\theta}.
\]
where $d_P$ is the multiprojective (sine) distance.
\end{theorem}

The multiprojective distance is defined by
\[
d_P([\mathbf f],[\mathbf g])^ 2 = \sum_i \inf_{\lambda \in \mathbb C} \frac{\| \mathbf f_i - \lambda \mathbf g_i \|^2}{\|\mathbf f_i\|^ 2}
.
\]
In the definition of $\theta$, the multiprojective distance can be replaced
by the Riemannian distance which is larger.

\subsection{Quadratic convergence}

The invariants for the toric Newton operator are:
\[
\beta(\mathbf f,\mathbf x) \defeq \|  {\mathbf N}_{\mathbf f}(\mathbf x) - \mathbf x\|_{\mathbf x} =
\| DS_{\mathbf f,\mathbf x}(0) ^{-1} S_{\mathbf f,\mathbf x}(0)\|_{\mathbf x},
\]
\[
\gamma(\mathbf f,\mathbf x) \defeq \max_{k \ge 2} 
\left(
\frac{1}{k!}
\left \|DS_{\mathbf f,\mathbf x}(0) ^{-1} D^k S_{\mathbf f,\mathbf x}(0) \right\|_{\mathbf x}
\right)^{1/(k-1)}
\]
and of course $\alpha(\mathbf f,\mathbf x) = \beta(\mathbf f,\mathbf x)\gamma(\mathbf f,\mathbf x)$.

We assume that $\mathbf z$ is a non-degenerate
zero of the line bundle section given by $\mathbf f$.
All norms will be taken with respect to $T_{\mathbf z}\mathscr M$. 

\begin{theorem}[$\gamma$-theorem]\label{toric:gamma}
Let $\mathbf z \in \mathscr M$ be a non-degenerate zero of $\mathbf f$.
If $\mathbf x_0 \in \mathscr M$ satisfies
\[
\| \mathbf x_0-\mathbf z\|_{\mathbf z} \left( \gamma(\mathbf f,\mathbf z) + 
\mu(\mathbf f,\mathbf z) \max_i \sup_{\|\mathbf u\|_{\mathbf z} \le 1}|(\mathbf m_i(\mathbf z)-\mathbf m_i(\mathbf x))\mathbf u|
\right)
\le \frac{3 - \sqrt{7}}{2},
\]
then the sequence $\mathbf x_{i+1} =  {\mathbf N}_{\mathbf f}(\mathbf x_i)$ is well-defined
and
\[
\| \mathbf x_i - \mathbf z\|_{\mathbf z} \le 2^{-2^i+1} \| \mathbf x_0 - \mathbf z\|_{\mathbf z} . 
\]
\end{theorem}

A trivial bound for $\sup_{\|\mathbf u\|_{\mathbf z} \le 1}|(\mathbf m_i(\mathbf z)-\mathbf m_i(\mathbf x))\mathbf u|$ is
the circumscribed radius $\nu_i(\mathbf z)$.
We will obtain a sharper bound 
in Theorem~\ref{variation:moment}. 
Theorem~\ref{toric:gamma} is proved in Section~\ref{proof:toric:gamma}.

\begin{theorem}[$\alpha$-theorem]\label{th-alpha-toric}
Let \[
\alpha 
\le
\alpha_0 = \frac{13 - 3 \sqrt{17}}{4}.
\]
Let 
\[
r_0 = 
\frac{1+\alpha-\sqrt{1-6\alpha+\alpha^2}}{4\alpha} 
\text{ and }
r_1 =
\frac{1-3\alpha-\sqrt{1-6\alpha+\alpha^2}}{4\alpha} 
.
\]

If $\mathbf x_0 \in \mathscr M$ satisfies $\alpha(\mathbf f,\mathbf x_0) \le \alpha$,
then the sequence defined recursively by
$\mathbf x_{i+1} = \mathbf x_0+ {\mathbf N}_{\mathbf f,\mathbf x_0}(\mathbf x_i-\mathbf x_0)$
is well-defined and converges to a zero $\boldsymbol \zeta \in \mathscr M$
of $f$. Furthermore,
\begin{enumerate}[(a)]
\item
$
\| \mathbf x_i - \boldsymbol \zeta \|_{\mathbf x_0} \le 2^{-2^i+1} \| \mathbf x_1 - \mathbf x_0\|_{\mathbf x_0}
$
\item
$
\| \mathbf x_i - \boldsymbol \zeta \|_{\boldsymbol \zeta} \le 2^{-2^i+1} \| \mathbf x_1 - \mathbf x_0\|_{\mathbf x_0}
$
\item
$
\| \mathbf x_0 - \boldsymbol \zeta\|_{\mathbf x_0} \le r_0 \beta(\mathbf f, \quotient{\mathbf x_0})
$
\item
$
\| \mathbf x_0 - \boldsymbol \zeta\|_{\boldsymbol \zeta} \le r_0 \beta(\mathbf f, \quotient{\mathbf x_0})
$
\item
$
\| \mathbf x_1 - \boldsymbol \zeta\|_{\boldsymbol \zeta} 
\le r_1 \beta(\mathbf f, \quotient{\mathbf x_0}).
$
\end{enumerate}
\end{theorem}

\begin{proof}
Items (a) and (c) are just Theorem~\ref{th-alpha}(a,c) applied to
$S_{\mathbf f,\mathbf x_0}:T_{\mathbf x_0}\mathscr M=T_{\mathbf V(\mathbf x_0)}\mathscr V \rightarrow \mathbb C^n$.
The proof of item (b) mimics the proof of Theorem~\ref{th-alpha-proj}(b).
For all $1 \le j \le n$, we claim that
\begin{equation}\label{th-alpha-toric-partial}
\|\dot{\mathbf x}_i\|_{j,\mathbf x} \le 2^{-2^i+1} \| \dot{\mathbf x}^*\|_{j,\mathbf x}
.
\end{equation}
Indeed, assume without loss of generality that 
$\mathbf m_j(\mathbf x)=0$.
Then we set $v_j(\mathbf x)=\frac{1}{V_j(\mathbf x)}V_j(\mathbf x)$, so that
\[
Dv_j(\mathbf x)=\frac{1}{V_j(\mathbf x)} DV_j(\mathbf x)
.
\]
By definition, $\|\mathbf u\|_{j,\mathbf x} = \|Dv_j(\mathbf x) \mathbf u\|$. Moreover, $Dv_j(\mathbf x) \mathbf u \perp v_j(\mathbf x)$.
Let $\mathbf X=v_j(\mathbf x_0)$,
$\mathbf Y=v_j(\mathbf x_0) + Dv_j(\mathbf x_0) \dot{\mathbf x}_i$ and $Z=Dv_j(\mathbf x_0) \dot{\mathbf x}^*$. 
By item (a), $\|Y-Z\| \le \|X-Z\|$. Therefore, Lemma~\ref{plane-swap} 
implies that
\[
\frac{\| \pi(Y) - Z \|}{\|Z\|}
\le
\frac{\| Y - Z\|}{\|X\|}
\le 
2^{-2^i+1} 
\frac{\| X - Z\|}{\|X\|}
\]
where $\pi$ is the projection onto $Z+Z^{\perp}$.
This establishes equation \eqref{th-alpha-toric-partial}.
Squaring, adding for all $j$ and taking square roots, one gets:
\[
\|\dot{\mathbf x}_i\|_{\mathbf x} \le 2^{-2^i+1} \| \dot{\mathbf x}^*\|_{\mathbf x}
\]
The proof of items (d) and (e) is similar.
\end{proof}

\subsection{The higher derivative estimate}

A most important bound in modern homotopy papers is the
higher derivative estimate. While $\gamma$ is an awkward invariant
to approximate, there is a convenient upper bound:

\begin{theorem}\label{higher}
\[
\gamma(\mathbf f,\mathbf x) \le \frac{1}{2} \mu(\mathbf f,\mathbf x) \nu(\mathbf x)
\]
\end{theorem}

This can be compared to the classical bound $\gamma_0(\mathbf f,\boldsymbol \zeta) \le \frac{D^{3/2}}{2}\mu_{\text{norm}}(\mathbf f,\boldsymbol \zeta)$ for a homogeneous degree $D$ polynomial system
and $\boldsymbol \zeta \in \mathbb P^n$, see for instance
\ocite{BCSS}*{Th. 2 Sec.14.2}
or 
\ocite{Burgisser-Cucker}*{Prop. 16.45}.
With some further work, we will recover a more convenient version of
Theorem~\ref{toric:gamma}:

\begin{theorem}\label{toric:gamma:mu}
There is a constant $u_0 \simeq  0.090994609 \cdots$ 
with the following property.
Let $\mathbf z \in \mathscr M$ be a non-degenerate zero of $\mathbf f$.
If $\mathbf x_0 \in \mathscr M$ satisfies
\[
\frac{1}{2}\| \mathbf x_0-\mathbf z\|_{\mathbf z} \mu(\mathbf f,\mathbf z) \nu(\mathbf z) \le u_0
\]
then the sequence $\mathbf x_{i+1} =  {\mathbf N}_{\mathbf f}(\mathbf x_i)$ is well-defined
and
\[
\| \mathbf x_i - \mathbf z\|_{\mathbf z} \le 2^{-2^i+1} \| \mathbf x_0 - \mathbf z\|_{\mathbf z} . 
\]
\end{theorem}

Theorem \ref{th-alpha-toric} immediately becomes:

\begin{theorem}\label{toric-alpha-mu}
Let \[
\alpha 
\le
\alpha_0 = \frac{13 - 3 \sqrt{17}}{4}.
\]
Let 
\[
r_0 = 
\frac{1+\alpha-\sqrt{1-6\alpha+\alpha^2}}{4\alpha} 
\text{ and }
r_1 =
\frac{1-3\alpha-\sqrt{1-6\alpha+\alpha^2}}{4\alpha} 
.
\]

If $\mathbf x_0 \in \mathscr M$ satisfies $\frac{1}{2}\beta(\mathbf f,\mathbf x_0) \mu(\mathbf f,\mathbf x_0) \nu(\mathbf x_0) \le \alpha$,
then the sequence defined recursively by
$\mathbf x_{i+1} = \mathbf x_0+ {\mathbf N}_{\mathbf f,\mathbf x_0}(\mathbf x_i-\mathbf x_0)$
is well-defined and converges to a zero $\boldsymbol \zeta \in \mathscr M$
of $\mathbf f$. Furthermore,
\begin{enumerate}[(a)]
\item
$\| \mathbf x_i - \boldsymbol \zeta \|_{\mathbf x_0} \le 2^{-2^i+1} \| \mathbf x_1 - \mathbf x_0\|_{\mathbf x_0}$.
\item
$
\| \mathbf x_i - \boldsymbol \zeta \|_{\boldsymbol \zeta} \le 2^{-2^i+1} \| \mathbf x_1 - \mathbf x_0\|_{\mathbf x_0}
$
\item
$
\| \mathbf x_0 - \boldsymbol \zeta\|_{\mathbf x_0} \le r_0 \beta(\mathbf f, \mathbf x_0)
$
\item
$
\| \mathbf x_1 - \boldsymbol \zeta\|_{\boldsymbol \zeta} 
\le r_1 \beta(\mathbf f, \mathbf x_0).
$
\end{enumerate}
\end{theorem}

\begin{corollary}\label{gamma:remote}
There is a constant $\alpha_1\simeq 0.081239483\cdots$ with the following properties:
If
\[
\frac{1}{2} \mu(\mathbf f,\mathbf x_0) \nu(\mathbf x_0) \beta(\mathbf x_0) \le \alpha \le \alpha_1,
\]
then 
the sequence $\mathbf x_{i+1} =  {\mathbf N}_{\mathbf f}(\mathbf x_i)$ is well defined, converges
to a zero $\mathbf z$ of $\mathbf f$, and furthermore
\[
\| \mathbf x_i - \mathbf z\|_{\mathbf z} \le 2^{-2^{i-1}+1} r_1(\alpha) \beta(\mathbf f,\mathbf x_0)
.
\]
When $\alpha = \alpha_1$, $r_1(\alpha) \simeq 0.110020136\cdots$.
\end{corollary}

\begin{proof}
Let $\alpha_1$ be the smallest positive root of 
\[
\frac {\alpha r_1(\alpha)}{1-10 \alpha r_0(\alpha)} = u_0
\]
where $u_0 \simeq 0.090094609\cdots$
is the constant from Theorem~\ref{toric:gamma:mu}.
From Theorem~\ref{th-alpha-toric}(e), there is a zero $\mathbf z$ of $\mathbf f$
such that
\[
\|\mathbf x_1-\mathbf z\|_{\mathbf z} \le r_1(\alpha) \beta(\mathbf f,\mathbf x_0).
\]
Combining this with Theorem~\ref{main-bound},
\[
\frac{1}{2} \mu(\mathbf f,\mathbf z) \nu(\mathbf z) \|\mathbf x_1-\mathbf z\|_{\mathbf z} \le 
\frac{\alpha r_1(\alpha)}{1-10\alpha r_0(\alpha)}= u_0
.
\]
From Theorem~\ref{toric:gamma:mu}, the rest of the sequence
$\mathbf x_{i+1} =  {\mathbf N}_{\mathbf f}(\mathbf x_i)$ is well defined, converges to $\mathbf z$ and
\[
\| \mathbf x_i - \mathbf z\|_{\mathbf z} \le 
2^{-2^{i-1}+1} \|\mathbf x_1 - \mathbf x_2\|
\le
2^{-2^{i-1}+1} 
r_1(\alpha) \beta(\mathbf f,\mathbf x_0)
.
\]
\end{proof}

\begin{corollary}\label{cor:cont:path}
Let $(\mathbf f_{t})_{t \in [0,T]}$ be a $\mathcal C^1$ path in
$\mathbb P(\mathscr F_1) \times \cdots \times \mathbb P(\mathscr F_n)$.
Assume that the point $\mathbf x_0$ satisfies, for all $t \in [0,T]$, that
$\frac{1}{2}\mu(\mathbf f_t, \mathbf x) \nu(\mathbf f_t,\mathbf x_t) \beta(\mathbf f_t,\mathbf x_t) 
\le \alpha < 4/45=0.0888\cdots$.
Then for each $t \in [0,T]$,
the sequence $\mathbf x_0(t)=\mathbf x_0$, 
$\mathbf x_{i+1}(t) = \mathbf x_0+ {\mathbf N}_{\mathbf f_t,\mathbf x_0}(\mathbf x_i(t)-\mathbf x_0)$ converges uniformly
to some $\mathcal C^1$ path $\boldsymbol \zeta(t)$, 
\end{corollary}

\begin{proof}
By hypothesis $\mu(\mathbf f_t,\mathbf x_0)$ is bounded for $t \in [0,T]$. Hence
$\mu(\mathbf f_t,\mathbf x_0) \le \bar \mu$ for some finite $\bar \mu$.
By Th.~\ref{toric-alpha-mu}(c), $u=\mu(\mathbf f_t,\mathbf x_0) \nu(\mathbf f_t,\mathbf x_0) 
\|\mathbf x_0-\boldsymbol \zeta\|_{\mathbf x_0} \le 2 r_0(\alpha) \alpha < 1/5$. Thus,
$\mu(\mathbf f_t, \boldsymbol \zeta(t)) \le \bar \mu / (1-5 u)$ is finite.
By construction of the condition number,
\[
\left\| \frac{\partial}{\partial t} \boldsymbol \zeta(t) \right\|_{\boldsymbol \zeta(t)} \le
\frac{\bar \mu}{1-5u} 
\left\|\frac{\partial}{\partial t} [\mathbf f_t]\right\|_{[\mathbf f(t)]} 
\]
which is finite by compactness of the path $(\mathbf f_t)_{t \in [0,T]}$.
\end{proof}

\subsection{The cost of homotopy}

Recall that the {\em solution variety} is
\[
\mathscr S_0 = \left\{\ghostrule{3ex}(\mathbf f, \mathbf z) \in \mathbb P(\mathscr F_{A_1}) \times \cdots \times 
\mathbb P(\mathscr F_{A_n}) \times \mathscr M: \mathbf f(\mathbf z)=0 \right\}
\]
and that $\Sigma'$ is the set where $\mu(\mathbf f,\mathbf x) \nu(\mathbf x)=\infty$.
The {\em condition length} was defined as
\[
\mathscr L((\mathbf f_t, \mathbf z_t): t_0, t_1) =
\int_{t_0}^{t_1}   
\mu(\mathbf f_t,\mathbf z_t) \nu(\mathbf z_t) \sqrt{\| \dot{\mathbf f}_t \|_{f_t}^2 + \| \dot {\mathbf z}_t\|_{\mathbf z_t}^2}
\ \dd t
.
\]
We will need below the auxiliary quantity
\[
\mathscr L_1((\mathbf f_t, \mathbf z_t): t_0, t_1) =
\int_{t_0}^{t_1}   
\mu(\mathbf f_t,\mathbf z_t) \nu(\mathbf z_t) \left(\| \dot{\mathbf f}_t \|_{f_t} + \| \dot {\mathbf z}_t\|_{\mathbf z_t}\right)
\ \dd t
.
\]
that relates to the condition length by
\[
\mathscr L((\mathbf f_t, \mathbf z_t): t_0, t_1)
\le 
\mathscr L_1((\mathbf f_t, \mathbf z_t): t_0, t_1)
\le \sqrt{2} \mathscr L((\mathbf f_t, \mathbf z_t): t_0, t_1)
\]
\begin{proof}[Proof of Main Theorem A]
Assume that $0< u \le u_0 = \frac{3-\sqrt{7}}{2}$ is given.
Set $t_0=0$ and for $i=0, \dots, N-2$ 
choose $t_{i+1}$ so that ${\mathscr L_1(t_{i},t_{i+1})}=\delta$ for
some constant $\delta$ to be determined.
Then set $t_N=T$, and $\mathscr L(t_{N-1},t_{N})\le\delta$.

We consider the following induction hypothesis:
\begin{equation}\label{induction}
\frac{1}{2} \mu(\mathbf f_{t_i},\mathbf z_{t_i}) \nu(\mathbf z_{t_i}) 
\|\mathbf z_{t_i} - \mathbf x_i\|_{\mathbf z_{t_i}} 
\le u
\end{equation}
which is already satisfied for $i=0$. 
Theorem~\ref{toric:gamma:mu} implies that
$\mathbf x_{i+1}= {\mathbf N}(\mathbf f_{t_i}, \mathbf x_i)$ satisfies:
\[
\frac{1}{2} \mu(\mathbf f_{t_i},\mathbf z_{t_i}) \nu(\mathbf z_{t_i}) 
\|\mathbf z_{t_i} - \mathbf x_{i+1}\|_{\mathbf z_{t_i}} 
\le u/2
\]
To simplify notations, let $\mu=\mu(\mathbf f_{t_i},\mathbf z_{t_i})$, 
$\nu=\nu(\mathbf z_{t_i})$, $\mu'= \mu(\mathbf f_{t_{i+1}},\mathbf z_{t_{i+1}})$
and $\nu'= \nu(\mathbf z_{t_{i+1}})$.
Let $r=\max_{t_i \le t \le t_{i+1}} 
d_S(\mathbf f_t, \mathbf f_{t_i}) +
\| \mathbf z_t - \mathbf z_{t_i}\|_{\mathbf z_{t_i}}$.
Assume that the maximum is attained for $t=t^*$,
$t_i \le t^* \le t_{i+1}$. Then,
\begin{eqnarray*}
\mu\ \nu\ r &\le& 
\mu\ \nu 
\int_{t_i}^{t^*}
\left\| 
\frac{\partial}{\partial t} [\mathbf f_t]  
\right\|_{[\mathbf f_t]}
+
\left\|
\frac{\partial}{\partial t} z_t 
\right\|_{\mathbf z_t}
\dd t \\
&\le&
\frac{1}{1-5\mu \nu r} {\mathscr L_1(t_i, t^*)} \\
&\le&
\frac{1}{1-5\mu \nu r} {\mathscr L_1(t_i, t_{i+1})} \\
\end{eqnarray*}
Hence,
\[
\mu \nu r (1-5\mu \nu r) \le \delta
.
\]
The largest possible value of $\mu \nu r$ should therefore
satisfiy the quadratic equation
$\mu \nu r (1-5\mu \nu r) = \delta$. Solving the equation,
we deduce that
\[
\mu \nu r \le R(\delta) \defeq \frac{1}{10} \left(1-\sqrt{1 - 20 \delta}\right) = \delta (1+o(1)).
\]
Now we bound
\begin{eqnarray*}
\frac{1}{2}
\mu' \nu' \|\mathbf x_{t_{i+1}}-\mathbf z_{t_{i+1}}\|_{\mathbf z_{t_{i}}}&\le&
\frac{1}{2}
\mu' \nu' \left( 
\|\mathbf x_{t_{i+1}}-\mathbf z_{t_{i}}\|_{\mathbf z_{t_i}} + 
\|\mathbf z_{t_{i+1}}-\mathbf z_{t_{i}}\|_{\mathbf z_{t_i}} 
\right)
\\
&\le&
\frac{1}{1 - 5 \delta} \left( \frac{u}{2} + 
\frac{R(\delta)}{2}\right)
\end{eqnarray*}
and from Lemma \ref{lem:metric},
\[
\frac{1}{2} \mu' \nu' \|\mathbf x_{t_{i+1}}-\mathbf z_{t_{i+1}}\|_{\mathbf z_{t_{i+1}}}
\le
\frac{e^{R(\delta)}}{1 - 5 \delta} \left( \frac{u}{2} + 
\frac{R(\delta)}{2}\right)
\]
The induction hypothesis \eqref{induction} is guaranteed to hold
for $i+1$ as long as
\begin{equation}\label{delta-mu}
\frac{e^{R(\delta)}}{1 - 5 \delta} \left( \frac{u}{2} + 
\frac{R(\delta)}{2}\right)
\le
u
.
\end{equation}
When $u=u_0=\frac{3-\sqrt{7}}{2}$ we obtain numerically
the largest solution for this inequality, that is
$\delta \simeq 0.037391\cdots$.
In particular,
$N= 
{\lceil \frac{1}{\delta}\mathscr L_1(0,T) \rceil}
\le
{\lceil \frac{\sqrt{2}}{\delta}\mathscr L(0,T) \rceil}
\le \lceil {38}\mathscr L(0,T) \rceil
$.
\end{proof}

Before proving Main Theorem B, we need an extra result. Its proof is postponed.
\begin{proposition}\label{prop:alpha-remote}
Let $u<1/10$.
Assume that $\mathbf f(\mathbf z)=0$ and $\frac{1}{2}\mu(\mathbf f,\mathbf z) \nu(\mathbf f,\mathbf z) \|\mathbf z-\mathbf x\|_{\mathbf z} < u$.
Then, 
\[
\frac{1}{2}\mu(\mathbf f,\mathbf x) \nu(\mathbf x) \beta(\mathbf f,\mathbf x) \le 
u
e^{2u} 
\frac{1-u}{\psi(u)(1-10u)}
\]
\end{proposition}

\begin{proof}[Proof of Main Theorem B]
Let $u_1\simeq 0.003974518\cdots$ be the smallest root of 
\[
u
e^{2u} 
\frac{1-u}{\psi(u)(1-10u)}
=
\alpha_1
.
\]

Solving \eqref{delta-mu} for $u=u_1$ and $\delta$ by 
$\delta_1=0.024210342\cdots$
in the proof of the Main Theorem A,
Proposition ~\ref{prop:alpha-remote} implies that 
for all $t_i \le t \le t_{i+1}$,
\[
\frac{1}{2}
\mu(\mathbf f_{t}, \mathbf x_{i+1}) \nu(\mathbf x_{i+1}) \beta(\mathbf f_{t}, \mathbf x_{i+1}) \le \alpha_0.
\]

We replace the time mesh $0=t_0 \le t_1 \le \cdots$ by the one
in Main Theorem B. In that case $\mathscr L((f_t,z_t),t_i, t_{i+1}) \ge 
\delta/\sqrt{2} $ so the number of steps is still bounded above by
$\lceil \mathscr L((f_t, z_t),0,T) {\sqrt{2}}/ \delta \rceil
{\le 59 \mathscr L((f_t, z_t),0,T)}$.
Corollary~\ref{cor:cont:path} guarantees that each $x_i$ is indeed
an approximate root of $\mathbf f_{t_i}$ associated to $\mathbf z_{t_i}$.
Corollary~\ref{gamma:remote} then guarantees that the sequence 
$\mathbf y_0=\mathbf x_{N}, \mathbf y_{i+1} =  {\mathbf N}_{\mathbf f_T}(\mathbf y_i)$
converges quadratically to
$\mathbf z(t)$ and satisfies 
\[
\|\mathbf y_i-\mathbf z\|_{\mathbf z} \le 2^{-2^{i-1}+1}\times 0.100015909 
\ \beta(\mathbf f_T,\mathbf y_0)
\le 2^{-2^{i-1}-2} \beta(\mathbf f_T,\mathbf y_0).
\]  
\end{proof}

\begin{figure}
\centerline{\resizebox{\textwidth}{!}{\includegraphics{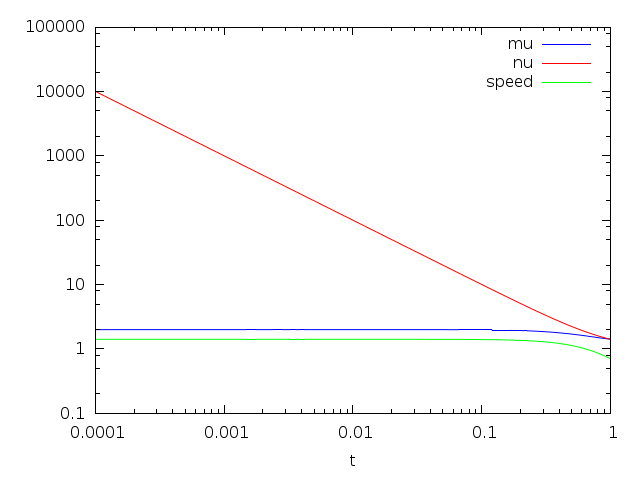}\includegraphics{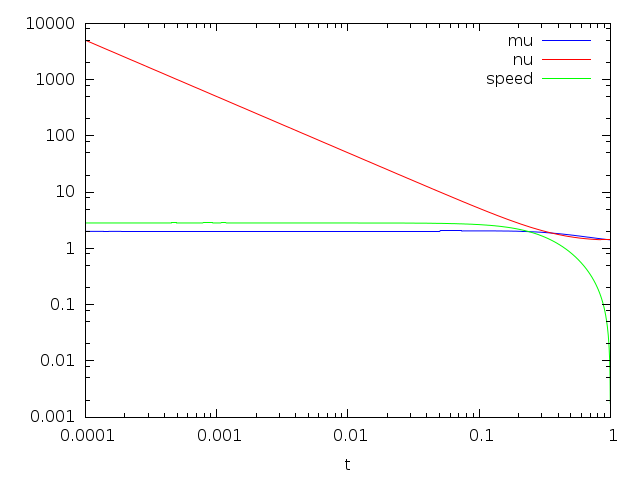}}}
\caption{Logarithmic plot for the invariants associated to each of
the solution paths, in the toric setting.\label{fig:roots}}
\end{figure}

\begin{rexample}
Recall that $\mathbf Z^{(i)}(t)=(X^{(i)}(t),Y^{(i)}(t))$, $i=1,2$ are the two roots for $f_t$
in the running example \eqref{eq-example}. Let $\mathbf z^{(i)}(t) = \log \mathbf Z^{(i)}(t)$
coordinatewise, and let $g^{(i)} = g(\mathbf z^{(i)})$ be the metric matrix for
$\langle \cdot, \cdot \rangle_{1,2}$. 

In order to obtain an approximation for the integral $\mathscr L$, 
we first compute the 
Taylor series of the Hermitian matrix
\[
\begin{split}
M^{(i)} = 
2 \|\mathbf V(\mathbf z^{(i)}(t)\|^{-2}
\begin{pmatrix}
\|f_1\|^{-1} & 0 \\
0 & \|f_2\|^{-1} 
\end{pmatrix}
& (\mathbf f_t \cdot D\mathbf V(\mathbf z^{(i)}(t))
 \ g^{(i)}_t\\
&(\mathbf f_t \cdot D\mathbf V(\mathbf z^{(i)}(t))^{*}
\ \begin{pmatrix}
\|f_1\|^{-1} & 0 \\
0 & \|f_2\|^{-1} 
\end{pmatrix}
.
\end{split}
\]
The factor of 2 comes from the
fact that we use the product metric
$\langle \cdot, \cdot \rangle = \langle \cdot, \cdot \rangle_{1} +
\langle \cdot, \cdot \rangle_{2}$ in the definition of $\mu$ and $\|\dot z\|$. 
Then, $\mu^{(i)} = \sqrt( \|M^{-1}\| )$.
The square of $\nu^{(i)}$ is the largest diagonal entry of 
the matrix
\[
N^{(i)} = \left( A - \begin{pmatrix}1\\1\\1\\1\end{pmatrix} \mathbf m(\mathbf z^{(i)})\right)
g^{(i)} \left( A - \begin{pmatrix}1\\1\\1\\1\end{pmatrix} \mathbf m(\mathbf z^{(i)})\right)
^T
,
\]
where $A=\begin{pmatrix} 1 & 0 \\ 0 & 2 \\ 1 & 1 \\ 0 & 3 \end{pmatrix}$ encodes
the support and $\mathbf m$ is the momentum map.
Computations for $\mu^{(i)}$, $\nu^{(i)}$ and the speed vector are
displayed in Table \ref{table-cond-length-t}. Actual values of the invariants
appear in Figure~\ref{fig:roots} 
We obtain in both cases that
\[
\mathscr L((\mathbf f_t,\mathbf z_t^{(i)}); \epsilon,1) = 
\int_{\epsilon}^1 \mu^{(i)}\nu^{(i)} \|
\cdots\|\dd t = 2 \log(\epsilon^{-1}) + O(1).
\]

\begin{table}
\centerline{
\begin{tabular}{||c||c|c||}
\hline \hline
$i$ & $1$ & $2$ \\
\hline \hline
$\|(f_1)_t\|^2$ & $(t^2+1)^2$ & $(t^2+1)^2$ \\
$\|(f_2)_t\|^2$ & 4 & 4 \\
$\|\mathbf V(\mathbf z^{(i)}(t))\|_{i}^2$& $2t^{-6}+O(t^{-4})$ & $\frac{1}{2}t^{-2} + O(1)$\\
$g^{(i)}(t)$ & 
$ \begin{pmatrix} \frac{1}{4} & -\frac{1}{2} 
\\ -\frac{1}{2} & 1+ t^2 +O(t^4) \end{pmatrix}$
&
$\begin{pmatrix} 4t^2&-8t^2\\-8t^2&\frac{1}{4}+16t^2\end{pmatrix} + O(t^4)$ \\
$f \cdot D\mathbf V(\mathbf z^{(i)})$ &
$\begin{pmatrix}
-t^{-2}+t^{-1} &
t^{-2}-3t^{-1} \\
t^{-3}+t^{-2} & 
-t^{-3}-t^{-2} 
\end{pmatrix}
$
&
$\begin{pmatrix}
-1-t^{2} &
\frac{3}{2} + \frac{5}{2}t^2 \\
0 & \frac{1}{2}t^{-1} + 1 + t/2
\end{pmatrix}$
\\
$M^{(i)}$ &
$\begin{pmatrix}
\frac{1}{4}+\frac{1}{2}t & -\frac{1}{2}t \\
-\frac{1}{2}t & \frac{1}{4}+\frac{1}{2}t 
\end{pmatrix}+O(t^2)$
&
$
\begin{pmatrix}
\frac{1}{4}& -\frac{1}{2}t \\
-\frac{1}{2}t & \frac{1}{4}+t 
\end{pmatrix}+O(t^2)$
\\
$(\mu^{i})^2$ & $4 + O(t)$ & $4 + O(t)$\\
\hline
$(\nu^{(i)})^2$ & $t^{-2} + 1$ & $\frac{1}{4}t^{-2}+\frac{3}{2}+\frac{1}{4}t^2$\\
$\left\| 
\frac{\partial}{\partial t} 
\left(\mathbf f_t, \mathbf z^{(i)}(t)\right) 
\right\|_{\left(\mathbf f_t, \mathbf z^{(i)}(t)\right) }^{2}$ &
$1-2t^2+O(t^4)$& $4-56t^2+O(t^4)$ \\
$\left(\mu^ {(i)} \nu^{(i)} \left\| 
\frac{\partial}{\partial t} 
\left(\mathbf f_t, \mathbf z^{(i)}(t)\right) 
\right\|_{\left(\mathbf f_t, \mathbf z^{(i)}(t)\right)}\right)^2 
$&$4t^{-2} + O(t^{-1})$&$4t^{-2} + O(t^{-1})$ \\
\hline
$\mu^ {(i)} \nu^{(i)} \left\| 
\frac{\partial}{\partial t} 
\left(\mathbf f_t, \mathbf z^{(i)}(t)\right) 
\right\|_{\left(\mathbf f_t, \mathbf z^{(i)}(t)\right) }
$ &
$2t^{-1}+O(1)$ & 
$2t^{-1}+O(1)$ 
\\
\hline \hline
$\mathscr L((\mathbf f_t,\mathbf z_t^{(i)}); \epsilon,1) = 
\int_{\epsilon}^1 \mu^{(i)}\nu^{(i)}  \sqrt{\| \dot{\mathbf f}_t \|_{\mathbf f_t}^2 + \| \dot {\mathbf z}_t\|_{\mathbf z_t}^2}
\dd t$ & $2 \log(\epsilon^{-1}) + O(1)$ &  $2 \log(\epsilon^{-1}) + O(1)$ 
\\
\hline \hline
\end{tabular}
}
\caption{\label{table-cond-length-t}Computation of the condition length
in the toric setting.}
\end{table}
\end{rexample}
\section{Distortion bounds}
\label{distortion}

Newton iteration is usually generalized to manifolds through the
use of geodesics and of the exponential map. Given a function or
a vector field $f$ defined on a manifold $M$, the Newton vector field
at this point evaluates to $\mathbf w=-D\mathbf f({\mathbf x})^{-1} \mathbf f(\mathbf x) \in T_{\mathbf x}M$. Then
$ {\mathbf N}_{\mathbf f,\mathbf x}$ is usually defined to be $\exp_{\mathbf x}(\mathbf w)$, where $\exp_{\mathbf x}(t\| \mathbf w\|^{-1} \mathbf w)$
is the geodesic passing at $\mathbf x$ for $t=0$ with tangent vector 
$\mathbf w/\|\mathbf w\|$ and
constant unit speed. This point can be found by solving the geodesic
differential equation, or by integrating it. 
\ocite{Dedieu-Priouret-Malajovich} generalized Smale's invariants to 
this context using high order covariant derivatives and parallel transport. 
A sharper analysis for equations defined by fiber bundles on a manifold
was carried out by \ocite{Li-Wang}.

Unfortunately, computing geodesics can be as hard as solving systems of 
equations. Indeed, let $\mathbf f: M \rightarrow \mathbb C^n$ be a 
holomorpic map from an $n$-dimensional complex manifold onto $\mathbb C^n$, 
and assume that the Hermitian structure of $M$ is the pull-back by $\mathbf f$ of
the canonical Hermitian structure. If $\mathbf x_0 \in M$ is an arbitrary point
and $\mathbf y_0 = \mathbf f(\mathbf x_0)$, then the segment $[\mathbf y_0, 0] \subset \mathbb C^n$ pulls
back to a minimizing geodesic $\mathbf x(t)$ with $\dot{\mathbf x}(0)=-D\mathbf f(\mathbf x_0)^{-1} \mathbf y_0$
and $\mathbf f ( \mathbf x(1) ) = 0$. Of course, one may be able to compute efficiently
geodesics on the sphere, on projective space and many interesting 
manifolds. No easy formula seems to be known for geodesics on toric
varieties. 

In this paper we traded the geodesics for straight lines in a unique
canonical chart. This is topologically equivalent outside toric infinity, and
is geometrically equivalent up to order 1. The Newton operator is much
easier to compute, and no covariant derivatives are needed.
There is a price to pay for bypassing geodesics. Parallel transport is
not available any more. Each point has a different Hermitian structure
associated to it. In this section we bound the distortion introduced by
this trivial transport operator. As usual, the momentum map is the key 
to bound this distortion.

\subsection{The momentum map}

Since the momentum map $\mathbf m_i(\mathbf x)$ plays such an important rôle in the theory,
we need to estimate how fast it changes with respect to
$\mathbf x$. The theorem below shows that the momentum is locally Lipschitz,
and allows to compute local Lipschitz constants.

\begin{theorem}\label{variation:moment}
Let $1 \le i \le n$ be fixed.
Let $\mathbf x, \mathbf y \in \mathscr M$.  
If $\nu_i(\mathbf x) \|\mathbf y-\mathbf x\|_{i,\mathbf x} \le s$ then
\begin{enumerate}[(a)]
\item For any $\mathbf w \in \mathbb R^n$, $|(\mathbf m_i(\mathbf y)-\mathbf m_i(\mathbf x))\mathbf w | \le 
\|\mathbf w\|_{i,\mathbf x} (e^{2s}-1) e^{e^{2s}-1-2s}$.
\item Let $d$ be the Riemannian distance in $(\mathscr M, \|\cdot\|_{i,\mathbf x})$. Then
$\|\mathbf m_i(\mathbf y)-\mathbf m_i(\mathbf x)\|_2 \le 2\,\diam{\conv{A_i}}\, d(\mathbf y,\mathbf x)$.
\end{enumerate}
\end{theorem}


Before proving the statement, we should point out an immediate
consequences of Theorem~\ref{variation:moment}(b).
A point $\mathbf v \in \mathscr V$ is said to be {\em at toric infinity} if it  
has no preimage in $\mathscr M$.

\begin{corollary}
Let $\mathbf x \in \mathscr M$ and let $\delta$ be the minimum over all $i$ of the
Euclidean distance from $\mathbf m_i(\mathbf x)$
to $\partial \conv{A_i}$, divided by the diamater of $\conv{A_i}$. 
Then, the open ball $B([\mathbf V(\mathbf x)],\delta/2)
\subset \mathscr V$ 
contains no point at toric infinity.
\end{corollary}

By dividing the induced Fubini-Study volume form by the total volume of
$\mathscr V$, one makes $\mathscr V$ into a probability space. The momentum
map is volume preserving, up to a constant. Therefore,
\begin{corollary}
The probability that $\mathbf v \in \mathscr V$ is at distance at most 
$\delta/2$
from a point at toric infinity is at most
\[
\delta
\sum_{i=1}^n
\frac{ \diam{\conv{A_i}} \vol(\partial \conv{A_i})}
{\vol(\conv{A_i})} .
\]
\end{corollary}

\begin{proof}[Proof of Theorem~\ref{variation:moment}]
Assume without loss of generality that $\mathbf m_i(\mathbf x)=0$.
Since the momentum $\mathbf m_i(\mathbf y)$
depends only on the real part of $\mathbf y$, assume also that $\mathbf x$ and
$\mathbf y$ are real.
For $k \ge 1$, define
\[
S_k(\mathbf y) = 2^{k-1} \sum_{a \in A_i} \rho_{\mathbf a}^2 e^{2\mathbf a\mathbf y} \underbrace
{\mathbf a \otimes \cdots \otimes \mathbf a}_{\text{$k$ times}}
.
\]
The momentum map is given by the formula
\[
\mathbf m_i(\mathbf y) = -\frac{1}{2}\phi(\mathbf y) S_1(\mathbf y)
\]
with $\phi(\mathbf y) = -2/\|V_i\|^2 = -2/ \sum_{\mathbf a\in A_i} |\rho_{\mathbf a} e^{\mathbf a\mathbf y}|^2$. 
The derivation rules for $\phi$ and $S_k$
are:
\begin{eqnarray*}
D\phi(\mathbf y) &=& \phi(\mathbf y)^2 S_1(\mathbf y) \\
DS_k(\mathbf y) &=& S_{k+1}(\mathbf y) \\
\end{eqnarray*}
The first derivatives of $\mathbf m(\mathbf y)$ are
\begin{eqnarray*}
D\mathbf m_i(\mathbf y)&=& -\frac{1}{2} \left(\ghostrule{3ex} \phi(x) S_2(\mathbf y) + \phi(\mathbf y)^2 S_1(\mathbf y)^2 \right)\\
D^2\mathbf m_i(\mathbf y)&=& -\frac{1}{2} \avg{} \left(\ghostrule{3ex} \phi(\mathbf y) S_3(\mathbf y) + 3 \phi(\mathbf y)^2 S_2(\mathbf y) S_1(\mathbf y) + \phi(\mathbf y)^3 S_1(\mathbf y)^3\right) \\
D^3\mathbf m_i(\mathbf y)&=& -\frac{1}{2} \avg{} \left( \ghostrule{3ex}\phi(\mathbf y) S_4(\mathbf y) + 4 \phi(\mathbf y)^2 S_3(\mathbf y) S_1(\mathbf y)
+ 3 \phi(\mathbf y)^3 S_2(\mathbf y) S_2(\mathbf y)
\right.\\ &&\left.
+ 6 \phi(\mathbf y)^3 S_2(\mathbf y) S_1(\mathbf y)^2
+ \phi(\mathbf y)^4 S_1(\mathbf y)^4\ghostrule{3ex}\right) 
\end{eqnarray*}
where average is taken over all permutations acting on
the arguments of the $j+1$-linear form within parentheses. 
Recall that
$D^j m_i(\mathbf y) = D^{j+1} \frac{1}{2} \log K(\mathbf y, \mathbf y)$ so
it should be a symmetric tensor.
The averaging above 
can be understood as a symmetrization operator. 
It is convenient to 
represent each term of the form $\avg{}
\left(\phi(\mathbf y)^k S_{i_1}(\mathbf y) \cdots S_{i_k}(\mathbf y)\right)$
by the Young diagram for the partition $j+1 =i_1 + i_2 + \cdots i_k$.
For instance when $j=3$ we write
\[
D^3\mathbf m_i(\mathbf y) = -\frac{1}{2}\left(
\ytableausetup{smalltableaux}
\ydiagram{4}
+
4 \ydiagram{3,1}
+
3 \ydiagram{2,2}
+
6 \ydiagram{2,1,1}
+
\ydiagram{1,1,1,1}
\right)
.
\]
The coefficients to each Young diagram 
are the number of ways to partition a set of $j+1$ labeled
elements into the corresponding partition. 
Indeed, the `derivative' of a Young diagram is obtained by
adding one box into every possible row, for instance
\[
D \left( \ydiagram{2,1,1} \right)
= \ydiagram{3,1,1} + \ydiagram{2,2,1} + \ydiagram{2,1,2}
+ \ydiagram{2,1,1,1}
= \ydiagram{3,1,1} + 2\ydiagram{2,2,1} + \ydiagram{2,1,1,1}
\]
Using this notation,
\[
D^j\mathbf m_i(\mathbf y) = -\frac{1}{2} \sum_Y c_{Y} Y
\]
where the sum ranges over all Young diagrams with $j+1$ boxes. A coarse
bound for the norm of $D^j\mathbf m(\mathbf y)$ is $\|D^j\mathbf m(\mathbf y)\| \le \frac{1}{2} \varpi_{j+1} \max \|Y\|$,
where  
$\varpi_{j+1} = \sum_Y c_Y$ is the 
number of partitions of a set with $j+1$ labelled elements,
known as the $j+1$-th {\em Bell number}, see Sloane's OEIS \ycite{OEIS}*{BELL sequence} and
Knuth's book \ycite{Knuth}.
\par
However we are actually bounding $\|D^j\mathbf m(\mathbf x)\|$ under the assumption
that $\mathbf m(\mathbf x)=0$. In particular,
$S_1(\mathbf x)=0$ and Young diagrams with at least one length one row should not be
counted. 
The number $\varpi'_j$ of Young diagrams with $j$ boxes and no row of length
one is also known as the number of complete rhyming schemes 
\cite{OEIS}*{sequence A000296}, \cite{Knuth}
and has exponential generating function $c(t)=e^{e^t-1-t}$. The first values
for $\varpi'_j$ are
\[
\varpi'_0=1
\hspace{1em}
\varpi'_1=0
\hspace{1em}
\varpi'_2=1
\hspace{1em}
\varpi'_3=1
\hspace{1em}
\varpi'_4=4
\hspace{1em}
\varpi'_5=11
\hspace{1em}
\varpi'_6=41
\hspace{1em}
\varpi'_7=162
\]
By using the fact that 
\[
\langle \mathbf w_1,\mathbf w_2 \rangle_{i,\mathbf x} =
\langle D[V_i](\mathbf x) \mathbf w_1, D[V_i](\mathbf x) \mathbf w_2 \rangle
=
\|V_i(\mathbf x)\|^{-2} \sum_{\mathbf a} \rho_{\mathbf a} e^{2\mathbf a\mathbf x} (\mathbf a\mathbf w_1) (\mathbf a\mathbf w_2),
\]
each $S_k$ can be bounded as follows:
\begin{eqnarray*}
|S_k(\mathbf x) (\mathbf w_1,\mathbf w_2,\mathbf w_3, \dots, \mathbf w_k)|
&=&
2^{k-1} \left| \sum_{\mathbf a} \rho_{\mathbf a}^2 e^{2\mathbf a\mathbf x} (\mathbf a\mathbf w_1)(\mathbf a\mathbf w_2) (\mathbf a\mathbf w_3) \cdots (\mathbf a\mathbf w_k)
\right|
\\
&\le&
2^{k-1} \left| \sum_{\mathbf a} \rho_{\mathbf a}^2 e^{2\mathbf a\mathbf x} (\mathbf a\mathbf w_1)(\mathbf a\mathbf w_2) \right|
\ \max_{\mathbf a} |\mathbf a\mathbf w_3| \cdots |\mathbf a\mathbf w_k|
\\
&\le&
2^{k-1} \|V_i(\mathbf x)\|^2 \left|\langle \mathbf w_1,\mathbf w_2\rangle_{i,\mathbf x}\right| \max_{\mathbf a} |\mathbf a\mathbf w_3| \cdots |\mathbf a\mathbf w_k|
\\
&\le&
2^{k-1} \|V_i(x)\|^2 \|\mathbf w_1\|_{i,\mathbf x} \|\mathbf w_2\|_{i,\mathbf x} \max_{\mathbf a} |\mathbf a\mathbf w_3| \cdots |\mathbf a\mathbf w_k|
\end{eqnarray*}
Hence,
\begin{equation}\label{bound:Sk}
|S_k(\mathbf w_1, \dots, \mathbf w_k)||\phi(\mathbf x)|
\le 2^{k} \|\mathbf w_1\|_{i,\mathbf x} \|\mathbf w_2\|_{i,\mathbf x}
\left( \max_{\mathbf a \in A} |\mathbf a\mathbf w_3| \cdots |\mathbf a\mathbf w_k| \right)
\end{equation}
If $Y$ is a Young diagram, let $r(Y)$ be its number of rows.
Let $w$ be an arbitrary vector.
Adding over all the Young diagrams with $j+1$ cases and
no row of length one, 
\[
\begin{split}
| D^j\mathbf m_i(\mathbf x)(\mathbf w,\mathbf y-\mathbf x, &\dots, \mathbf y-\mathbf x) | \le
\\ \le & \ 
2^j \|\mathbf w\|_{i,\mathbf x} 
\sum_Y  
\|\mathbf y-\mathbf x\|_{i,\mathbf x}^{2r(Y)-1}
\left( \max_{\mathbf a \in A_i} |\mathbf a(\mathbf y-\mathbf x)| \right)^{j+1-2r(Y)}
\\
\le & \ 
2^ j \varpi' _{j+1}
\|\mathbf w\|_{i,\mathbf x} \|\mathbf y-\mathbf x\|_{i,\mathbf x} 
\left( \max_{\mathbf a \in A_i} |\mathbf a(\mathbf y-\mathbf x)| \right)^{j-1}
\end{split}
\]
where the last inequality uses that $\|\mathbf y-\mathbf x\|_{i,\mathbf x} \le 
\max_{\mathbf a \in A_i} |\mathbf a(\mathbf y-\mathbf x)|$.
Recall that $\mathbf m_i(\mathbf x)=0$. The Taylor series of $\mathbf m_i(\mathbf y)\mathbf w$ around $\mathbf x$ is:
\[
\mathbf m_i(\mathbf y)(\mathbf w) = D\mathbf m_i(\mathbf x) (\mathbf w,\mathbf y-\mathbf x) + \sum_{j \ge 2} \frac{1}{j!}D^j\mathbf m_i(\mathbf x)(\mathbf w,(\mathbf y-\mathbf x)^j).
\]
Let $c(t) = e^{e^t-1-t}$ be
the exponential generating function for the number of complete rhyming schemes
$\varpi'_j$. 
We can bound:
\begin{eqnarray*}
|\mathbf m_i(\mathbf y)\mathbf w| 
&\le& \|\mathbf w\|_{i,\mathbf x} \|\mathbf y-\mathbf x\|_{i,\mathbf x} 
\sum_{j \ge 1} \frac{2}{j!} \varpi'_{j+1}  
\left( 2 s\right)^{j-1} 
\\
&\le& 
\frac{\|\mathbf w\|_{i,\mathbf x} \|\mathbf y-\mathbf x\|_{i,\mathbf x}}{s}
\sum_{j \ge 1} \frac{1}{j!} \varpi'_{j+1}  
\left( 2 s\right)^{j} 
\\
&\le&
\frac{\|\mathbf w\|_{i,\mathbf x} \|\mathbf y-\mathbf x\|_{i,\mathbf x}}{s}
(c'(2 s)-\varpi'_1)
\\
&\le&
\|\mathbf w\|_{i,\mathbf x} 
c'(2 s)
\end{eqnarray*}
because $\varpi'_1=0$ and $\|\mathbf y-\mathbf x\|_{i,\mathbf x} \le  
\nu_i(\mathbf x) \|\mathbf y-\mathbf x\|_{i,\mathbf x} = s$.
Explicitly, 
$c'_1(t) = (e^t-1) e^{e^t-1-t}$ so item (a) follows:
\[
|\mathbf m_i(\mathbf y)\mathbf w| \le \|\mathbf w\|_{\mathbf x} (e^{2s}-1) e^{e^{2s}-1-2s}
\]

In order to prove item (b), we 
apply the bound \eqref{bound:Sk} to the
formula $D\mathbf m_{i}(\mathbf y)=-\frac{1}{2}\phi(\mathbf y)S_2(\mathbf y)$. One obtains:
\[
|D\mathbf m_i(\mathbf y) (\mathbf w_1, \mathbf w_2)| \le \|\mathbf w_1\|_{i,\mathbf x} \|\mathbf w_2\|_{i,\mathbf x}.
\]
Let $\mathbf x(t)_{t \in [0,T]}$ be a minimizing geodesic with respect to
$\|\cdot\|_{i,\mathbf x}$ with boundary $\mathbf x(0)=\mathbf x$ and $\mathbf x(T)=\mathbf y$.
Then,
\begin{eqnarray*}
| (\mathbf m_i(\mathbf y)-\mathbf m_i(\mathbf x))(\mathbf w_2) | &=& 
\left|\int_{0}^T D\mathbf m_i(\mathbf x(t)) (\dot{\mathbf x}(t),\mathbf w_2)\right| \dd t
\\
&\le&
\int_{0}^T 2 \|\dot{\mathbf x}(t)\|_{i,\mathbf x(t)} \|w_2\|_{i,\mathbf x(t)} \dd t
\\
&\le&
2 \max( \|\mathbf w_2\|_{i,\mathbf x(t)} ) \int_{0}^T \|\dot{\mathbf x}(t)\|_{i,\mathbf x(t)} \dd t
\\
&\le&
2 \max( \|\mathbf w_2\|_{i,\mathbf x(t)} ) d_i(x,y)
\\
&\le&
2 \|\mathbf w_2\| \diam{\conv{A_i}} d_i(\mathbf x,\mathbf y)
\end{eqnarray*}
where the last bound follows from the inequality $\|u\|_{i,\mathbf x} \le \|\mathbf w\|_2\diam{\conv{A_i}}$.
\end{proof}

\subsection{The local norms and the circumscribed radii}

\begin{proof}[Proof of Lemma~\ref{lem:metric}]
Assume without loss of generality that $\mathbf m_i(\mathbf x)=0$ for all $i$.
Write $\|\mathbf w\|_{i,\mathbf y} = \|Dv_i(\mathbf y) \mathbf w\|$ where $v_i(\mathbf x)=V_i(\mathbf x)/\|V_i(\mathbf x)\|$.
In that case,
\[
Dv_i(\mathbf y)\mathbf w = Dv_i(\mathbf x)\mathbf w + \sum_{k \ge 2} \frac{1}{k-1!} D^kv_i(\mathbf w,\mathbf y-\mathbf x, \cdots \mathbf y-\mathbf x)
.
\]
Also,
\[
D^kv_i(\mathbf x) (\mathbf w, \mathbf y-\mathbf x, \cdots, \mathbf y-\mathbf x) = 
\frac{1}{\|V_i(\mathbf x)\|}
\begin{pmatrix}
\vdots\\
\rho_{\mathbf a} e^{\mathbf a \mathbf x} (\mathbf a\mathbf w) (\mathbf a(\mathbf y-\mathbf x))^{k-1} \\
\vdots
\end{pmatrix}
\]
so
\begin{eqnarray*}
\left\|
D^kv_i(\mathbf x) (\mathbf w, \mathbf y-\mathbf x, \cdots, \mathbf y-\mathbf x)  
\right\|
&\le&
\left\|
\frac{1}{\|V_i(\mathbf x)\|}
\begin{pmatrix}
\vdots\\
\rho_{\mathbf a} e^{\mathbf a \mathbf x} (\mathbf a \mathbf w) 
\\
\vdots
\end{pmatrix}
\right\|
\max_{\mathbf a \in A_i} |\mathbf a(\mathbf y-\mathbf x)|^{k-1} 
\\
&=&
\| D[V_i](\mathbf x) \mathbf w\| \max_{\mathbf a \in A_i} |\mathbf a(\mathbf y-\mathbf x)|^{k-1}
\end{eqnarray*}
Therefore,
\[
\|Dv_i(\mathbf y)\mathbf w - Dv_i(\mathbf x)\mathbf w\|
\le
\|\mathbf w\|_{i,\mathbf x} \sum_{k \ge 2} 
\frac{
(\max_{\mathbf a \in A_i} |\mathbf a(\mathbf y-\mathbf x)|)^{k-1}
}{k-1!} 
\le \|\mathbf w\|_{i,\mathbf x} \left( e^{s}-1\right).
\]
Triangular inequality yields 
\[
(2-e^{s}) \|\mathbf w\|_{i,\mathbf x} 
\le
\|\mathbf w\|_{i,\mathbf y}
\le
e^{s} \|\mathbf w\|_{i,\mathbf x} 
\] 
so the first statement follows.
The second statement is now obvious.
\end{proof}

The circumscribed radii $\nu_i$ were crucial in our previous bounds.
For later use, we also estimate their variation rate.
\begin{lemma}\label{var:nu}
Let $1 \le i \le n$. Let $\mathbf x \in \mathscr M$ and
$\nu_i(\mathbf x) \|\mathbf y-\mathbf x\|_{i,\mathbf x} \le s$. Then,
\[
\left(\nu_i(\mathbf x) - (e^{2s}-1) e^{e^{2s}-1-2s}\right)
e^{-s}
\le
\nu_i(\mathbf y) \le
\left(\nu_i(\mathbf x) + (e^{2s}-1) e^{e^{2s}-1-2s}\right)
\frac{1}{2-e^{s}}
.
\]
\end{lemma}

It follows immediately that if $\nu(\mathbf x) \|\mathbf y-\mathbf x\|_{\mathbf x} \le s$,
\[
\left(\nu(\mathbf x) - (e^{2s}-1) e^{e^{2s}-1-2s}\right)
e^{-s}
\le
\nu(\mathbf y) \le
\left(\nu(\mathbf x) + (e^{2s}-1) e^{e^{2s}-1-2s}\right)
\frac{1}{2-e^{s}}
\]
as well.

\begin{proof}
In the sequel we drop the $i$ index so $A = A_i$, $\nu=\nu_i$, 
$\| \cdot \|_{\mathbf x} = \| \cdot \|_{i,\mathbf x}$, etc...
We introduce the notation 
$\nu_{\mathbf x}(\mathbf y) = \max_{\|\mathbf u\|_{\mathbf x}=1} |(\mathbf a-\mathbf m(\mathbf y))\mathbf u|$ so $\nu_{\mathbf x}(\mathbf x)=\nu(\mathbf x)$.
From triangular inequality,
\[
\nu_{\mathbf x}(\mathbf x) -  \|\mathbf m(\mathbf y)-\mathbf m(\mathbf x)\|_{\mathbf x}
\le 
\nu_{\mathbf x}(\mathbf y) \le \nu_{\mathbf x}(\mathbf x) + \|\mathbf m(\mathbf y)-\mathbf m(\mathbf x)\|_{\mathbf x}
\]
so
\begin{equation}\label{eq:breaking:rho}
(\nu_{\mathbf x}(\mathbf x) -  \max\|\mathbf m(\mathbf y)-\mathbf m(\mathbf x)\|_{\mathbf x}) \frac{\nu_y(\mathbf y)}{\nu_{\mathbf x}(\mathbf y)}
\le 
\nu(\mathbf y) \le 
(\nu_{\mathbf x}(\mathbf x) + \max\|\mathbf m(\mathbf y)-\mathbf m(\mathbf x)\|_{\mathbf x}) \frac{\nu_y(\mathbf y)}{\nu_{\mathbf x}(\mathbf y)}
.
\end{equation}
From Lemma~\ref{lem:metric}, $(2-e^s) \|{\mathbf w}\|_{\mathbf x} \le \|{\mathbf w}\|_y \le e^s \|{\mathbf w}\|_{\mathbf x}$
so
\begin{eqnarray*}
\nu_{\mathbf y}(\mathbf y)
&=&
\sup_{\|\mathbf w\|_{\mathbf y} \le 1}  
\max_{\mathbf a \in A} |(\mathbf a-\mathbf m(\mathbf y))\mathbf w|
\\
&\le&
\frac{1}{2-e^{s} }
\sup_{\|\mathbf w\|_{\mathbf x} \le 1}  
\max_{\mathbf a \in A} |(\mathbf a-\mathbf m(\mathbf y))\mathbf w|
\\
&=&
\frac{1}{2-e^{s}}
\nu_{\mathbf x}(\mathbf y)
\end{eqnarray*}
and
\begin{eqnarray*}
\nu_{\mathbf y}(\mathbf y)
&=&
\sup_{\|\mathbf w\|_{\mathbf y} \le 1}  
\max_{\mathbf a \in A} |(\mathbf a-\mathbf m(\mathbf y))\mathbf w|
\\
&\ge&
\frac{1}{e^{s} }
\sup_{\|\mathbf w\|_{\mathbf x} \le 1}  
\max_{\mathbf a \in A} |(\mathbf a-\mathbf m(\mathbf y))\mathbf w|
\\
&=&
\frac{1}{e^{s}}
\nu_{\mathbf x}(\mathbf y)
\end{eqnarray*}
The last two bounds and Theorem~\ref{variation:moment}(a)
can be substituted into equation \eqref{eq:breaking:rho}:
\[
\left(\nu(\mathbf x) - (e^{2s}-1) e^{e^{2s}-1-2s}\right)
e^{-s}
\le
\nu(\mathbf y) \le
\left(\nu(\mathbf x) + (e^{2s}-1) e^{e^{2s}-1-2s}\right)
\frac{1}{2-e^{s}}
\]

\end{proof}
\subsection{The condition number}

Toward the proof of Theorem~\ref{main-bound}, 
we will show the following estimate. It should be compared to
\ocite{Burgisser-Cucker}*{Prop. 16.55}. The extra factor $1+O(s)$ comes
from the different local norms.
\begin{theorem}
\label{th:variation}
Let $[\mathbf f],[\mathbf g] \in \mathbb P(\mathscr F_1), \dots, \mathbb P(\mathscr F_n)$. Let $\mathbf x \in \mathscr M$. Assume that for all $i$, $\nu_i(\mathbf x) \|\mathbf y-\mathbf x\|_{i,\mathbf x} \le s$.
If $\mu(\mathbf f,\mathbf x) \left( d_P(\mathbf f,\mathbf g) + (e^{s}-1)\right) <1$, then
\[
\frac{
\left( 2 - e^{s} \right)
\mu(\mathbf f,\mathbf x)}
{1+\mu(\mathbf f,\mathbf x) \left( d_P(\mathbf f,\mathbf g) + (e^{s}-1) \right)}
\le
\mu(\mathbf g,\mathbf y)
\le
\frac{
e^s
\mu(\mathbf f,\mathbf x)}
{1-\mu(\mathbf f,\mathbf x) \left( d_P(\mathbf f,\mathbf g) + (e^{s}-1) \right)}
\]
where $d_P$ is the multiprojective (sine) distance.
\end{theorem}
In the proof of Theorem~\ref{th:variation} we will need
two well-known Lemmas about linear mappings between
normed spaces. The proofs are included for completeness.
\begin{lemma}
Let $A$ and $B$ be 
linear operators between finite dimensional normed spaces. 
Let $\sigma(X) = \inf_{\|\mathbf u\|\le 1} \|X u\|$ and
let $\|X\|$ denote the operator norm of $X$. 
Then,
\[
|\sigma(A) - \sigma(B) | \le \|A-B\|
\]
\end{lemma}

\begin{proof}
Assume that $\sigma(A) = \|A\mathbf u\|$ with $\|\mathbf u\|=1$.
Triangular inequality yields
\[
\sigma(A) =
\|A\mathbf u\| \ge \|B\mathbf u\| - \|(A-B) \mathbf u\| \ge 
\sigma(B) - \|A-B\|. 
\]
Replacing $A$ by $B$ one obtains that 
$\sigma(B) \ge \sigma(A) - \|A-B\|$.
\end{proof}

\begin{lemma}\label{lemma:invertible}
Let $A, B$ be invertible linear operators between finite dimensional 
normed spaces. 
If $\|A^{-1}\| \|A-B\|<1$, 
then
\[
\frac{\|A^{-1}\|}{1+\|A^{-1}\| \|A-B\|}
\le
\|B^{-1}\| 
\le \frac{\|A^{-1}\|}{1-\|A^{-1}\| \|A-B\|}
\]
\end{lemma}

\begin{proof}
From the previous Lemma,
\[
\frac{1}{\|A^{-1}\|}
-\|A-B\|
\le
\frac{1}{\|B^{-1}\|}
\le
\frac{1}{\|A^{-1}\|}
+
\|A+B\|
.\]
Multiplying by $\|A^{-1}\| \|B^{-1}\|$,
\[
\|B^{-1}\| (1 - \|A^{-1}\| \|A-B\|)
\le
\|A^{-1}\|
\le
\|B^{-1}\| (1 + \|A^{-1}\| \|A-B\|)
\]
and so
\[
\frac{\|A^{-1}\|}{1+\|A^{-1}\| \|A-B\|}
\le
\|B^{-1}\| 
\le \frac{\|A^{-1}\|}{1-\|A^{-1}\| \|A-B\|}
\]
\end{proof}

\begin{proof}[Proof of Theorem~\ref{th:variation}]
Assume without loss of generality that $\mathbf m_i(\mathbf x)=0$ for all $i$.
Also without loss of generality, scale the $f_i$ such that
$\|f_1\| = \cdots = \|f_n\|=1$ and the $g_i$ such that $\|f_i-g_i\|$ is
minimal, so $d_P(\mathbf f,\mathbf g)=\|\mathbf f-\mathbf g\|$.
Let $v_i(\mathbf x) = \frac{1}{\|V_i(\mathbf x)\|} V_i(\mathbf x)$. 
Because $\mathbf m_i(\mathbf x)=0$ for all $i$, we can write
\[
\mu(\mathbf f,\mathbf x) = \left\| (\mathbf f \cdot Dv(\mathbf x))^{-1} \right\|_{\mathbf x}
\]
where $\mathbf f \cdot D\mathbf v(\mathbf x)$ is an operator from $(\mathscr M, \| \cdot \|_{\mathbf x})$
into $\mathbb C^n$ with the canonical norm assumed.

From the previous Lemma,
\[
\frac{\mu(\mathbf f,\mathbf x)}
{1+\mu(\mathbf f,\mathbf x) T}
\le
\left\| (\mathbf g \cdot D\mathbf v(\mathbf y))^{-1} \right\|_{\mathbf x}
\le
\frac{\mu(\mathbf f,\mathbf x)}
{1-\mu(\mathbf f,\mathbf x) T}
\]
where $T = \left\| \mathbf f \cdot D\mathbf v(\mathbf x) - \mathbf g \cdot D\mathbf v(\mathbf y) \right\|_{\mathbf x}$.
We estimate $T=T'+T''$ where
\begin{eqnarray*}
T' &=& \left\| \mathbf f \cdot D\mathbf v(\mathbf x) - \mathbf g \cdot D\mathbf v(\mathbf x) \right\|_{\mathbf x} \\
&\le& 
\sup_{\|\mathbf w\|_{\mathbf x}\le 1} \left\| (\mathbf f-\mathbf g) \cdot D\mathbf v(\mathbf x) \mathbf w \right\| 
\\
&\le&
\sup_{\|\mathbf w\|_{\mathbf x} \le 1} 
\sqrt{ \sum \|f_i-g_i\|^2} \max_i \|\mathbf w\|_{i,\mathbf x}
\\
&\le&
\|\mathbf f-\mathbf g\|
\end{eqnarray*}
and
\begin{eqnarray*}
T'' &=& \left\| \mathbf g \cdot D\mathbf v(\mathbf x) - \mathbf g \cdot D\mathbf v(\mathbf y) \right\|_{\mathbf x} \\
&=&
\sup_{\|\mathbf w\|_{\mathbf x} \le 1}
\sqrt{
\sum_i |g_i (Dv_i(\mathbf x) - Dv_i(\mathbf y))\mathbf w|^2
}
\\
&\le&
\sup_{\|\mathbf w\|_{\mathbf x} \le 1}
\sqrt{
\sum_i \|g_i\|^2} \max_i \|(Dv_i(\mathbf x) - Dv_i(\mathbf y))\mathbf w\|
\\
&\le&
\sup_{\|\mathbf w\|_{\mathbf x}\le 1}
\max_i \|(Dv_i(\mathbf x) - Dv_i(\mathbf y))\mathbf w\|
\\
&\le&
\max_i 
\sup_{\|\mathbf w\|_{i,\mathbf x}\le 1}
\sum_{k \ge 2} \frac{1}{k-1!}\|D^kv_i(\mathbf x)(\mathbf w, \mathbf y-\mathbf x, \cdots, \mathbf y-\mathbf x)\|
\\
&\le&
\max_i 
\sum_{k \ge 2} \frac{1}{k-1!} \nu_i(\mathbf x)^{k-1} \|\mathbf y-\mathbf x\|_{i,\mathbf x}^{k-1}
\\
&\le& (e^s-1)
\end{eqnarray*}

Therefore,
\[
\frac{
\left( 2- e^{s}\right)
\mu(\mathbf f,\mathbf x)}
{1+\mu(\mathbf f,\mathbf x) \left( \|\mathbf f-\mathbf g\| + (e^{s}-1) \right)}
\le
\mu(\mathbf g,\mathbf y)
\le
\frac{
e^s
\mu(\mathbf f,\mathbf x)}
{1-\mu(\mathbf f,\mathbf x) \left( \|\mathbf f-\mathbf g\| + (e^{s}-1) \right)}
\]
\end{proof}
We are ready to prove Theorem~\ref{main-bound}.
\begin{proof}[Proof of Theorem~\ref{main-bound}]
Let $s=\max_i \nu_i(\mathbf x) \|\mathbf y-\mathbf x\|_{i,\mathbf x}$ and recall that
$\mu(\mathbf f,\mathbf x) \nu(\mathbf x)$ $(\|\mathbf y-\mathbf x\|_{\mathbf x} + d_P(\mathbf f,\mathbf g)) \le \theta$.
The right-hand sides of Lemma~\ref{var:nu}
can be bounded above by
\begin{eqnarray*}
\nu(\mathbf y) &\le&
\left(\nu(\mathbf x) + (e^{2s}-1) e^{e^{2s}-1-2s}\right)
\frac{1}{2-e^{s}}
\\
&\le&
\nu(\mathbf x) 
\left(1  + (e^{2\theta}-1) e^{e^{2\theta}-1-2\theta}\right)
\frac{1}{2-e^{\theta}}
\end{eqnarray*}
using $\nu(\mathbf x) \ge 1$ and $s=\nu(\mathbf x) \|\mathbf x-\mathbf y\|_{\mathbf x} \le \theta$. 
The right hand side
of Theorem ~\ref{th:variation}
satisfies
\begin{eqnarray*}
\mu(\mathbf g,\mathbf y) &\le&
\frac{e^s
\mu(\mathbf f,\mathbf x)}
{1-\mu(\mathbf f,\mathbf x) \left( d_P(\mathbf f-\mathbf g) + (e^{s}-1) \right)}
\\
&\le&
\frac{e^s \mu(\mathbf f,\mathbf x)}
{2-e^{\theta}}
\end{eqnarray*}
using $\mu(\mathbf f,\mathbf x) (e^{s}-1) \le e^{\mu(\mathbf f,\mathbf x) \nu(\mathbf x) \|\mathbf x-\mathbf y\|_{\mathbf x}}-1$ 
and hence
$\mu(\mathbf f,\mathbf x) \left( d_P(\mathbf f-\mathbf g) + (e^{s}-1)\right)$ $\le e^\theta-1$.
Putting all together,
\[
\mu(\mathbf g,\mathbf y) \nu(\mathbf y) \le 
\mu(\mathbf f,\mathbf x) \nu(\mathbf x) 
e^u
\frac{1  + (e^{2\theta}-1) e^{e^{2\theta}-1-2\theta}}
{\left(2-e^{\theta}\right)^2}
\le
\frac{\mu(\mathbf f,\mathbf x) \nu(\mathbf x)}{1-5\theta}.
\]
By a similar argument,
\[
\mu(\mathbf g,\mathbf y) \nu(\mathbf y) \ge 
\mu(\mathbf f,\mathbf x) \nu(\mathbf x)
\frac{( 1-(e^{2\theta}-1) e^{e^{2\theta}-1-2\theta}) (2-e^{\theta}) }
{e^{2\theta}}
\ge
(1-5\theta) 
\mu(\mathbf f,\mathbf x) \nu(\mathbf x).
\]
\end{proof}
\section{Proof of the technical results}
\label{proof:tech}

\subsection{Proof of the toric $\gamma$-theorem}
\label{proof:toric:gamma}

For the proof of Theorem~\ref{toric:gamma} we will need
the following fact, which can be stated as a general result
about the $\gamma$ invariant. Let $\kappa(X)=\|X\| \|X^{-1}\|$ be
the Wilkinson condition number for a square matrix $X$, where
operator norms are assumed:

\begin{lemma}\label{action:gamma} 
Let $\mathbf z \in \mathbb C^n$ be fixed, and let $\mathbf f: U
\rightarrow \mathbb C^n$ with $\mathbf f(\mathbf z)=0$ be holomorphic on a 
neighborhood of $\mathbf z$. 
Let $\mathbf m_1, \dots, \mathbf m_n \in (\mathbb C^n)^*$ and
set $g_i(\mathbf x) = e^{-m_i \cdot \mathbf x} f_i(\mathbf x)$. 
If $\mathbf f(\mathbf z)=0$ then
\[
\gamma(\mathbf g,\mathbf z) \le \kappa(D\mathbf f(\mathbf z)) \max_i \sup_{\|\mathbf w\|_{{\mathbf z}} \le 1} |\mathbf m_i(\mathbf w)| + \gamma(\mathbf f,\mathbf z)
.
\]
\end{lemma}

\begin{proof}[Proof of Lemma~\ref{action:gamma}]
We differentiate $g_i$ to obtain 
\[
Dg_i(\mathbf x) = e^{-\mathbf m_i \cdot \mathbf x} \left(Df_i(\mathbf x) - f_i(\mathbf x)\mathbf m_i\right)
.
\]
Since $\mathbf g$ vanishes at $\mathbf z$, we have $Dg_i(\mathbf z)=e^{-\mathbf m_i \cdot \mathbf z} Df_i(\mathbf z)$.
By induction,
\[
D^kg_i(\mathbf x) = e^{-\mathbf m_i \cdot \mathbf x} 
\sum_{l=0}^k 
(-1)^{l} \binomial{k}{l} \avg{} \left(D^{k-l}f_i(\mathbf x) \otimes \mathbf m_i^{\otimes l}\right)
\]
where the average is taken over all the permutations of the covariant
indices.  In order to bound $\gamma(\mathbf g,\mathbf z)$, we will produce a bound for
$\frac{\|D\mathbf g(\mathbf z)^{-1} D^k\mathbf g(\mathbf z)\|_{\mathbf z}}{k!}$.
For clarity, we examine first the case $k=2$. Assume that
the operator norm of $\frac{1}{2} D\mathbf g(\mathbf z)^{-1} D^2\mathbf g(\mathbf z)$ is attained at
unit vectors $\mathbf w_1$ and $\mathbf w_2$, that is
\[
\left\|
\frac{1}{2} D\mathbf g(\mathbf z)^{-1} D^2\mathbf g(\mathbf z)
\right\|_{{\mathbf z}} =
\left\|
\frac{1}{2} D\mathbf g(\mathbf z)^{-1} D^2\mathbf g(\mathbf z) (\mathbf w_1, \mathbf w_2)
\right\|_{{\mathbf z}} 
\]
where $\|\mathbf w_1\|_{{\mathbf z}}=\|\mathbf w_2\|_{{\mathbf z}}=1$. Expand
\begin{eqnarray*}
\frac{1}{2} D\mathbf g(\mathbf z)^{-1} D^2\mathbf g(\mathbf z)(\mathbf w_1,\mathbf w_2) &=& 
\frac{1}{2} D\mathbf f(\mathbf z)^{-1} D^2\mathbf f(\mathbf z)(\mathbf w_1,\mathbf w_2) \\
&-&
\frac{1}{2} D\mathbf f(\mathbf z)^{-1} 
\left[
\begin{matrix}
\mathbf m_1 \cdot \mathbf w_1 \\
& \ddots\\
& & \mathbf m_n \cdot \mathbf w_1 
\end{matrix}
\right]
D\mathbf f(\mathbf z) \mathbf w_2
\\
&-&
\frac{1}{2} D\mathbf f(\mathbf z)^{-1} 
\left[
\begin{matrix}
\mathbf m_1 \cdot \mathbf w_2 \\
& \ddots\\
& & \mathbf m_n \cdot \mathbf w_2 
\end{matrix}
\right]
D\mathbf f(\mathbf z) \mathbf w_1
.
\end{eqnarray*}
Taking norms, $\|\frac{1}{2} D\mathbf g(\mathbf z)^{-1} D^2\mathbf g(\mathbf z)(\mathbf w_1,\mathbf w_2)\|_{{\mathbf z}} \le
\gamma(\mathbf f,\mathbf z)+ \kappa(D\mathbf f(\mathbf z))\max_i \|\mathbf m_i\|_{{\mathbf z}}$.  
The general case is similar. Assume that the operator norm of
$\frac{1}{k!} 
D\mathbf g(\mathbf z)^{-1} D^k\mathbf g(\mathbf z)$ is attained at $\mathbf w_1, \dots, \mathbf w_k$, namely
\[
\left\|
\frac{1}{k!} 
D\mathbf g(\mathbf z)^{-1} D^k\mathbf g(\mathbf z)  
\right\|
_{{\mathbf z}}=
\left\|
\frac{1}{k!} 
D\mathbf g(\mathbf z)^{-1} D^k\mathbf g(\mathbf z) (\mathbf w_1, \dots, \mathbf w_k)
\right\|
_{{\mathbf z}}\]
with $\|\mathbf w_1\|_{{\mathbf z}}= \cdots = \|\mathbf w_k\|_{{\mathbf z}}$. Then,
\[
\begin{split}
\frac{1}{k!} 
D\mathbf g(\mathbf z)^{-1} D^k\mathbf g(\mathbf z)(\mathbf w_1, \dots, \mathbf w_k)& = \\
=
\frac{1}{k!} 
\sum_{l=0}^{k-1} 
(-1)^{l} \binomial{k}{l} 
\avg{}
&
\left( 
D\mathbf f(\mathbf z)^{-1}
M(\mathbf w_1, \dots, \mathbf w_l) 
D^{k-l}\mathbf f(\mathbf z)(\mathbf w_{l+1}, \dots, \mathbf w_{k})
\right)
\end{split}
\]
with
\[
M(\mathbf w_1, \dots, \mathbf w_l) \defeq
\left[
\begin{matrix}
\prod_{j=1}^l \mathbf m_1 \cdot \mathbf w_{j} \\
& \ddots \\
& & \prod_{j=1}^l \mathbf m_n \cdot \mathbf w_{j} 
\end{matrix}
\right]
.
\]
Taking norms,
{\small
\begin{eqnarray*}
\frac{1}{k!} 
\left\| D\mathbf g(\mathbf z)^{-1} D^k\mathbf g(\mathbf z) \right\|_{{\mathbf z}}& \le &
\frac{1}{k!}
\sum_{l=0}^{k-1} 
\binomial{k}{l} 
\avg{}
\left\| 
D\mathbf f(\mathbf z)^{-1}
M(\mathbf w_1, \dots, \mathbf w_l)\right. \\
&&
\hspace{10em}\left.
D^{k-l}\mathbf f(\mathbf z)(\mathbf w_{l+1}, \dots, \mathbf w_{k})
\right\|
_{{\mathbf z}}\\
&\le&
\frac{1}{k!}
\sum_{l=0}^{k-1} 
\binomial{k}{l} 
\avg{}
\left\| 
D\mathbf f(\mathbf z)^{-1}
M(\mathbf w_1, \dots, \mathbf w_l) 
D\mathbf f(\mathbf z)
\right\|_{{\mathbf z}} \\
&&\hspace{10em}
\left\|
D\mathbf f(\mathbf z)^{-1}
D^{k-l}\mathbf f(\mathbf z)(\mathbf w_{l+1}, \dots, \mathbf w_{k})
\right\|
_{{\mathbf z}}\end{eqnarray*}
When $l=0$ we have $\left\| D\mathbf f(\mathbf z)^{-1} M(\mathbf w_1, \dots, \mathbf w_l) D\mathbf f(\mathbf z) \right\|_{{\mathbf z}}=1$.
Otherwise, its value can be bounded above by
$\kappa(D\mathbf f(\mathbf z)) \max_{i,j}(|\mathbf m_i \cdot \mathbf w_j|)^{l}$. Using the fact that
$\kappa(D\mathbf f(\mathbf z)) \ge 1$, we bound
\begin{eqnarray*}
\frac{1}{k!} 
\left\| D\mathbf g(\mathbf z)^{-1} D^k\mathbf g(\mathbf z) \right\|
_{{\mathbf z}}&\le&
\sum_{l=0}^{k-1} 
\frac{1}{l!}
\kappa(D\mathbf f(\mathbf z))
\max_{i,j}(|\mathbf m_i \cdot \mathbf w_j|)^{l}
\frac{
\left\|
D\mathbf f(\mathbf z)^{-1}
D^{k-l}\mathbf f(\mathbf z)
\right\|
_{{\mathbf z}}}{k-l!}
\\
&\le&
\sum_{l=0}^{k-1} 
\binomial{k-1}{l}
\kappa(D\mathbf f(\mathbf z))^l
(\max_{ij} |\mathbf m_i \cdot \mathbf w_j|)^l \gamma(\mathbf f,\mathbf z)^{k-l-1}
\\
&\le&
\left(
\kappa(D\mathbf f(\mathbf z))
\max_{ij}|\mathbf m_i \cdot \mathbf w_j| + \gamma(\mathbf f,\mathbf z) \right)^{k-1}
\end{eqnarray*}
}
Taking $k-1$-th roots, we obtain:
\[
\gamma(\mathbf g,\mathbf z) \le \kappa(D\mathbf f(\mathbf x)) \max_{ij} |\mathbf m_i \cdot \mathbf w_j| + \gamma(\mathbf f,\mathbf z).
\]
\end{proof}

We will need the following, well-known Lemma. Since the proof 
is short, it is included for completeness.
\begin{lemma}\label{lem:bound:derivative}
Let $\mathbf g: (\mathbb E, \|\cdot\|) \rightarrow (\mathbb F, \|\cdot\|)$  
be a holomorphic map between Banach spaces.  
Let ${u}=\|\mathbf z-\mathbf x\| \gamma(\mathbf g,\mathbf z)< 1-\frac{\sqrt{2}}{2}$. Then, $D\mathbf g(\mathbf x)$ is invertible and
\begin{equation}\label{bound:derivative}
\| D\mathbf g (\mathbf x)^{-1} D\mathbf g(\mathbf z)\|  \le 
\frac{(1-u^2)}{\psi(u)}
\end{equation}
where $\psi(u)=1-4u+2u^2$.
\end{lemma}

\begin{proof}
\begin{eqnarray*}
\|
(D\mathbf g(\mathbf z))^{-1}
D\mathbf g(\mathbf x) - I
\|
&\le&
\sum_{k \ge 2} \frac{
\left\| (D\mathbf g(\mathbf z))^{-1}
D^k\mathbf g(\mathbf z)\right\| 
}{k-1!} 
{\|\mathbf x-\mathbf z\|}^{k-1}\\
&\le& 
\sum_{k \ge 2} k \gamma(\mathbf g, \mathbf z)^{k-1} \|\mathbf x-\mathbf z\|^{k-1}
\\
&=& \frac{1}{(1 - u)^2} -1.
\end{eqnarray*}
Therefore $D\mathbf g(\mathbf z))^{-1}
D\mathbf g(\mathbf x)$ is invertible and
\[
\| D\mathbf g (\mathbf x)^{-1} D\mathbf g(\mathbf z)\|  \le \frac{1}{1- \left(\frac{1}{(1 - u)^2} -1\right)}
= \frac{(1-u^2)}{\psi(u)}
\]
and equation \eqref{bound:derivative} holds.
\end{proof}

\begin{proof}[Proof of Theorem~\ref{toric:gamma}]
We assume without loss of generality that $\mathbf m_i(\mathbf z)=0$.
For each $i$, we use the $i$-th momentum map
to produce an `integrating factor' at $\mathbf x_0$:
Set $W_i(\mathbf x) = e^{-\mathbf m_i(\mathbf x_0) ({\mathbf x})}V_i(\mathbf x)$. 
Then 
\begin{eqnarray*}
f_i \cdot\frac{1}{\|V_i(\mathbf x_0)\|}  P_{V_i^\perp} DV_i(\mathbf x_0) &=& 
f_i \cdot \frac{1}{\|V_i(\mathbf x_0)\|} \left(I-\frac{1}{\|V_i(\mathbf x_0)\|^2}V_i(\mathbf x_0) V_i(\mathbf x_0)^*\right)DV_i(\mathbf x_0)\\
&=& f_i \cdot \frac{1}{\|V_i(\mathbf x_0)\|} DV_i(\mathbf x_0) - f_i \cdot \frac{1}{\|V_i(\mathbf x_0)\|} V_i(\mathbf x_0) \mathbf m_i(\mathbf x_0)\\
&=& e^{{\mathbf m_i}(\mathbf x_0) \mathbf x_0} f_i \cdot \frac{1}{\|V_i(\mathbf x_0)\|} DW_i(\mathbf x_0) 
\\
&=& f_i \cdot \frac{1}{\|W_i(\mathbf x_0)\|} DW_i(\mathbf x_0) 
\end{eqnarray*}

The toric Newton operator
takes $\mathbf x_0$ to $\mathbf x_1 =  {\mathbf N}_{\mathbf f}(\mathbf x_0)$ where
\begin{eqnarray*}
\mathbf x_1  
&=&
\mathbf x_0 - 
\left({\mathbf f} \cdot (I - P_{V(\mathbf x_0)^{\perp}}) D\mathbf V(\mathbf x_0) \right)^{-1}
\mathbf f \cdot \mathbf V(\mathbf x_0)
\\
&=&
\mathbf x_0 - 
(\mathbf f\cdot D\mathbf W(\mathbf x_0))^{-1}
\mathbf f\cdot  \mathbf W(\mathbf x_0)
\end{eqnarray*}

Thus, the toric Newton operator at $\mathbf x_0$ is the same as
the usual Newton operator at $\mathbf x_0$ for the function
$\mathbf g(\mathbf x) = \mathbf f \cdot \|\mathbf W(\mathbf x_0)\|^{-1} \mathbf W(\mathbf x)$. 
This differs from the local section by a ratio
\[
\mathbf g(\mathbf x) = S_{\mathbf f,\mathbf x_0}(\mathbf x-\mathbf z) e^{-\mathbf m_i(\mathbf x_0) \mathbf x}
\]
Also, $\mathbf g(\mathbf z)=0$.

From now on we use the metric structure of $T_{\mathbf z}\mathscr M$. All 
norms, operator norms and the invariant $\gamma$ are computed
with the norm $\|\cdot\|_{\mathbf z}$. Lemma~\ref{action:gamma} provides
the bound
\[
\gamma(\mathbf g,\mathbf z) = 
\kappa(DS_{\mathbf f,\mathbf z}(0)) 
\max_i \|\mathbf m_i(\mathbf x_0)\|_{\mathbf z}
+
\gamma(\mathbf f,\mathbf z) 
\]
where $\kappa(DS_{\mathbf f,\mathbf z}(0)) = \|DS_{\mathbf f,\mathbf z}(0)\|_{\mathbf z} 
\|DS_{\mathbf f,\mathbf z}(0)^{-1}\|_{\mathbf z} \le \mu(\mathbf f,\mathbf x)$ using operator norms. 
Above, $\|\mathbf m_i(\mathbf x_0)\|_{\mathbf z} = \max_{\|\mathbf w\|_{\mathbf z} \le 1} |\mathbf m_i(\mathbf x_0) \mathbf w|$ is
the norm of $\mathbf m_i(\mathbf x_0)$ as a covector. Since we took $\mathbf m_i(\mathbf z)=0$,
$\|\mathbf m_i(\mathbf x_0)\|_{\mathbf z}  
=\|\mathbf m_i(\mathbf x_0)-\mathbf m_i(\mathbf z)\|_{\mathbf z}$.
Therefore,
\[
\|\mathbf x_0-\mathbf z\|_{\mathbf z} \gamma(\mathbf g, \mathbf z) 
\le
\gamma(\mathbf f,\mathbf z) + \mu(\mathbf f,\mathbf z) \|\|\mathbf m_i(\mathbf x_0)-\mathbf m_i(\mathbf z)\|_{\mathbf z}
\le \frac{3-\sqrt{7}}{2}
.
\]

By Theorem \ref{th-gamma} applied to $\mathbf g$ one would
achieve quadratic convergence yet for a different Newton
operator, namely $\mathbf x \mapsto \mathbf x - (\mathbf f \cdot P_{\mathbf V(\mathbf x_0)^{\perp}} D\mathbf V(\mathbf x))^{-1} 
 \cdot \mathbf V(\mathbf x)$. Instead, we just claim that for $\mathbf x_1= {\mathbf N}_{\mathbf f}(\mathbf x_0)$,
\begin{equation}\label{gamma:induction}
\|\mathbf x_1-\mathbf z\|_{\mathbf z} \le \|\mathbf x_0-\mathbf z\|_{\mathbf z} \frac{u}{\psi(u)}
\end{equation}
where $\mathbf u=\gamma(\mathbf g,\mathbf z) \|\mathbf x_0-\mathbf z\|_{\mathbf z}$
and $\psi(u)=1-4u+2u^2$.
If we define the sequence
$u_i = \gamma(\mathbf g,\mathbf z) \|\mathbf x_i-\mathbf z\|_{\mathbf z}$, we deduce from \eqref{gamma:induction}
that
\[
u_{i+1} \le \frac{u_i^2}{\psi(u)}
.
\]
This is enough to deduce that the $u_i$ decrease faster than
the iterates of $t_0=0$, $t_{i+1}= {\mathbf N}_{h_{\gamma}}(t_i)$,
for $h_{\gamma}(t) = t - \frac{\gamma t^2}{1-\gamma t}$,
$\gamma = \gamma(\mathbf f,\mathbf z)$. This in turn implies that
\[
u_{i} \le 2^{-2^i+1} u_0
\]
and hence
\[
\| \mathbf x_i-\mathbf z\|_{\mathbf z} \le 2^{-2^i+1} \| \mathbf x_0-\mathbf z\|_{\mathbf z}.
\]

\medskip
\par
It remains to prove \eqref{gamma:induction}.
Set $W_i(\mathbf x) = e^{-\mathbf m_i(\mathbf x_0) (\mathbf x)}V_i(\mathbf x)$.
As before, $u=\gamma(\mathbf g,\mathbf z)\|\mathbf x_0-\mathbf z\|$.
Then
\begin{eqnarray*}
\mathbf x_1 - \mathbf z 
&=&
\mathbf x_0 - \mathbf z - 
\left(\mathbf f \cdot (I - P_{\mathbf V(\mathbf x_0)^{\perp}}) D\mathbf V(\mathbf x_0) \right)^{-1}
\mathbf f \cdot \mathbf V(\mathbf x_0)
\\
&=&
\mathbf x_0 - \mathbf z - 
D\mathbf g(\mathbf x_0))^{-1}
\mathbf g(\mathbf x_0)
\\
&=&
(D\mathbf g(\mathbf x_0))^{-1}
\left(
D\mathbf g(\mathbf x_0)
(\mathbf x_0-\mathbf z)
- 
\mathbf g(\mathbf x_0)
\right)
\\
&=&
(D\mathbf g(\mathbf x_0))^{-1}
(D\mathbf g(\mathbf z))
(D\mathbf g(\mathbf z))^{-1}
\left(
D\mathbf g(\mathbf x_0)
(\mathbf x_0-\mathbf z)
- 
\mathbf g(\mathbf x_0)
\right)
\end{eqnarray*}

For all vector $\mathbf w$, we can expand
\[
D\mathbf g(\mathbf x_0) \mathbf w = 
D\mathbf g(\mathbf z) \mathbf w + 
\sum_{k \ge 2} \frac{1}{k-1!} 
D^k\mathbf g(\mathbf z) ((\mathbf x_0-\mathbf z)^{k-1}, \mathbf w).
\]
Lemma~\ref{lem:bound:derivative} applied to $\mathbf g:
(T_{\mathbf z}\mathscr M,\|\cdot\|_{\mathbf z}) \rightarrow (\mathbb C^n, \| \cdot \|_2)$
implies that
\[
\| D\mathbf g (\mathbf x_0)^{-1} D\mathbf g(\mathbf z)\|_{\mathbf z}  \le
\frac{(1-u)^2}{\psi(u)}
\]
with $\psi(u)=1-4u+2u^2$.
It remains to bound
\[
(D\mathbf g(\mathbf z) )^{-1}
\left(
D\mathbf g(\mathbf x_0)
(\mathbf x_0-\mathbf z)
- 
\mathbf g(\mathbf x_0)
\right)
=
\sum_{k \ge 2}
\frac{k-1}{k!}
(D\mathbf g(\mathbf z) )^{-1}
D^k\mathbf g(\mathbf z) (\mathbf x_0-\mathbf z)^{k} 
\]
by
\[
\left\|
(D\mathbf g(\mathbf z) )^{-1} ( \cdots )
\right\|_{\mathbf z}
\le
\sum_{k \ge 2}
(k-1) u^{k-1} \|\mathbf x_0-\mathbf z\|_{\mathbf z}
=
\frac{u \|\mathbf x_0-\mathbf z\|_{{\mathbf z}}}{(1-u)^2}.
\]
This shows that $\|\mathbf x_1-\mathbf z\|_{\mathbf z} \le \|\mathbf x_0-\mathbf z\|_{\mathbf z} \frac{u}{\psi(u)}$,
establishing \eqref{gamma:induction}.

\end{proof}

\subsection{The higher derivative estimate}
\begin{proof}[Proof of Theorem \ref{higher}]
Assume without loss of generality that $\mathbf m_i(\mathbf x)=0$ for all $i$.
\begin{eqnarray*}
\frac{1}{k!} 
\left\|
DS_{\mathbf f,\mathbf x}(0)^{-1} D^ kS_{\mathbf f,\mathbf x}(0)
\right\|_{\mathbf x} \le \hspace{-7em}&&
\\
&\le&
\frac{1}{k!}
\left\|
DS_{\mathbf f,\mathbf x}(0)^{-1} 
\begin{pmatrix}
\|f_1\| \\
& \ddots \\
& & \|f_n\|
\end{pmatrix}
\right\|_{\mathbf x}
\left\|
\begin{pmatrix}
\frac{1}{\|f_1\|} f_1 \cdot \frac{1}{\|V_1(\mathbf x)\|} D^k V_1(\mathbf x)
\\
\vdots 
\\
\frac{1}{\|f_n\|} f_n \cdot \frac{1}{\|V_n(\mathbf x)\|} D^k V_n(\mathbf x)
\end{pmatrix}
\right\|_{\mathbf x}
\\
&\le&
\frac{1}{k!}
\mu(\mathbf f,\mathbf x) 
\left\|
\begin{pmatrix}
\frac{1}{\|V_1(\mathbf x)\|} D^k V_1(\mathbf x)
\\
\vdots \\
\frac{1}{\|V_n(\mathbf x)\|} D^k V_n(\mathbf x)
\end{pmatrix}
\right\|_{\mathbf x}
\\
&\le&
\frac{1}{k!}\mu(\mathbf f,\mathbf x) \nu(\mathbf x)^{k-1}
\end{eqnarray*}
as in the proof of Lemma~\ref{lem:metric}.
Then use the fact that $\mu(\mathbf f,\mathbf x) \ge 1$ to bound the expression above
by 
\[
\frac{1}{k!} 
\left\|
DS_{\mathbf f,\mathbf x}(0)^{-1} D^ kS_{\mathbf f,\mathbf x}(0)
\right\|_{\mathbf x} \le 
\frac{1}{2^{k-1}}\mu(\mathbf f,\mathbf x)^{k-1} \nu(\mathbf x)^{k-1},
\]
before taking $k-1$-th roots.
\end{proof}

\subsection{Proof of the modified gamma theorem}

\begin{proof}[Proof of Theorem~\ref{toric:gamma:mu}]
Assume that $u=\frac{1}{2}\|\mathbf x_0-\mathbf z\|_{\mathbf z} \mu(\mathbf f,\mathbf z) \nu(\mathbf z) \le \frac{3-\sqrt{7}}{2}$.
From Theorem~\ref{higher} we can bound
\[
\|\mathbf x_0-\mathbf z\|_{\mathbf z} \gamma(\mathbf f,\mathbf z) \le u
\]
From Theorem ~\ref{variation:moment}(a) and bounding $s \le 2u$,
\[
\|\mathbf x_0-\mathbf z\|_{\mathbf z} \mu(\mathbf f,\mathbf z) 
\max_i \sup_{\|\mathbf w\|_{\mathbf z}=1}|(\mathbf m_i(\mathbf z)-\mathbf m_i(\mathbf x))\mathbf w|
\le
2u (e^{4u}-1) e^{e^{4u}-1-4u}
.
\]
The inequality $u + 2u (e^{4u}-1) e^{e^{4u}-1-4u} \le  \frac{3-\sqrt{7}}{2}$
holds for $u \le u_0 = 0.090994609 \cdots$, where $u_0$ was obtained numerically.
\end{proof}

\subsection{Proof of Proposition~\ref{prop:alpha-remote}}

\begin{proof}
Since $\mathbf f(\mathbf z)=0$,
\[
D\mathbf f(\mathbf z)^{-1}\mathbf f(\mathbf x) = \mathbf x-\mathbf z + \sum_{k \ge 2} \frac{1}{k!} D\mathbf f(\mathbf z)^{-1}D^k\mathbf f(\mathbf z)(\mathbf x-\mathbf z)^k
\]
so that
\begin{eqnarray*}
\| D\mathbf f(\mathbf z)^{-1}\mathbf f(\mathbf x) \|_{\mathbf z} &\le&
 \|\mathbf x-\mathbf z\|_{\mathbf z} \left(1 + \sum_{k \ge 2} 
\gamma(\mathbf f,\mathbf z)^{k-1} \|\mathbf x-\mathbf z\|_{\mathbf z}^{k-1}\right)
\\
&=&
\frac{\|\mathbf x-\mathbf z\|_{\mathbf z}}{1- \gamma(\mathbf f,\mathbf z) \|\mathbf x-\mathbf z\|_{\mathbf z}}
\end{eqnarray*}
Since $\|\mathbf x-\mathbf z\|_{\mathbf z} \gamma(\mathbf f,\mathbf z)\|\le u \le 1/10 < 1-\sqrt{2}/2$,
Lemma~\ref{lem:bound:derivative} 
allows us to bound
\begin{eqnarray*}
\| D\mathbf f(\mathbf x)^{-1}\mathbf f(\mathbf x) \|_{\mathbf z} &\le& 
\| D\mathbf f(\mathbf x)^{-1}D\mathbf f(\mathbf z) \|_{\mathbf z}  
\| D\mathbf f(\mathbf z)^{-1}\mathbf f(\mathbf x) \|_{\mathbf z} \\
&\le& 
\frac{(1-u)^2}{\psi(u)}
\frac{\|\mathbf x-\mathbf z\|_{\mathbf z}}{1- \gamma(\mathbf f,\mathbf z) \|\mathbf x-\mathbf z\|_{\mathbf z}}
\end{eqnarray*}
with $\psi(u)=1-4u+2u^2$.
Theorem~\ref{higher} and
Lemma~\ref{lem:metric}
with $s\le 2u$ imply
\[
\beta(\mathbf f,\mathbf x) = \| D\mathbf f(\mathbf x)^{-1}\mathbf f(\mathbf x) \|_{\mathbf x} \le 
e^{2u} 
\frac{1-u}{\psi(u)}
\|\mathbf x-\mathbf z\|_{\mathbf z}
\]
Also, Theorem~\ref{main-bound} with $\theta \le 2u$ implies that
\[
\mu(\mathbf f,\mathbf x)\nu(\mathbf x) \le \frac{ \mu(\mathbf z) \nu(\mathbf z) }{1 - 10 u}
\]
so
\[
\frac{1}{2}
\mu(\mathbf f,\mathbf x)\nu(\mathbf x) 
\beta(\mathbf f,\mathbf x) 
\le
u
e^{2u} 
\frac{1-u}{\psi(u)(1-10u)}
\]

\end{proof}

\section{Finsler structure}
\label{sec:Finsler}

The toric variety associated to an unmixed 
system of sparse polynomial equations 
has $n$ natural Hermitian metrics, each one 
induced by the support of one of the equations.
In Section~\ref{subsec:hermitian} we added up all those
Hermitian metrics to produce one Hermitian metric, namely
\[
\langle \cdot , \cdot \rangle_{\mathbf x} =
\langle \cdot , \cdot \rangle_{1,\mathbf x} +
\cdots +
\langle \cdot , \cdot \rangle_{n,\mathbf x}
.
\]
This metric cannot be a natural object. Each of the $n$ Hermitian metrics
is actually induced by a Kahler symplectic form, and the mixed volume is
the integral over the toric variety of the wedge product of those $n$
forms, up to a constant. 
By adding the Hermitian metrics, information is lost.
Instead, a formal linear combination
\[
\lambda_1 \langle \cdot , \cdot \rangle_{1,\mathbf x} +
\cdots +
\lambda_n \langle \cdot , \cdot \rangle_{n,\mathbf x}
\]
would preserve the mixed volume information, the mixed volume being
proportional to the coefficient in $\lambda_1 \cdots \lambda_n$
of the total volume. Those linear combinations are induced by a
semigroup structure on the space of spaces of fewnomials, see
\cite{Malajovich-Fewspaces} and the discussion therein.
\medskip
\par
Therefore, it may be more natural to measure lengths on
$\mathscr V$ and $\mathscr M$ in some way that is invariant
of the coefficients $\lambda_1, \dots, \lambda_n > 0$.
Instead of using the Hermitian norm
\[
\| \cdot \|_{\mathbf x} = \sqrt{\langle \cdot , \cdot \rangle_{\mathbf x}}.
\]
we can also use
\[
\finsler{\mathbf w}{\mathbf x} = \max_i \| \mathbf w\|_{i,\mathbf x} = 
\max_i \sqrt{ \langle \mathbf w , \mathbf w \rangle_{i,\mathbf x} }
.
\]
This associates a norm to each $\mathbf x$. Because each $\|\cdot\|_{i,\mathbf x}$ is
rescaling invariant, $\finsler{\cdot}{\mathbf x}$ is 
independent of the $\lambda_i$.
We always have $\finsler{\mathbf w}{\mathbf x} \le \|\mathbf w\|_{\mathbf x}$. 
In the running
example, $\finsler{\mathbf w}{\mathbf x} = \frac{\sqrt{2}}{2} \|\mathbf w\|_{\mathbf x}$.

\begin{remark}
Most authors define a Finsler structure as a  
function $F:T\mathcal M \rightarrow \mathbb R$
so that $F(\mathbf x, \cdot)$ is a norm and $F(\mathbf x, \dot{\mathbf x})$ is smooth
or $\mathcal C^1$ for $\dot{\mathbf x} \ne 0$. The norm $\finsler{\mathbf x}{\dot{\mathbf x}}$
is only guaranteed to be continuous and subdifferentiable.  
Properly speaking, one might call it a {\em subdifferentiable} Finsler
structure.
\end{remark}

Smale's alpha-theory was originally stated for holomorphic mappings between
Banach spaces. 
The definition of invariants $\beta$, $\gamma$ and $\alpha$ for a Newton
operator $\mathbb B \rightarrow \mathbb B$ only uses 
the norm on $\mathbb B$ and the induced operator norm for multilinear
maps. In the context of this paper, the invariants become

\[
\beta(\mathbf f,\mathbf x) = \finsler{ {\mathbf N}_{\mathbf f}(\mathbf x) - \mathbf x}{\mathbf x} =
\finsler{DS_{\mathbf f,\mathbf x}(0) ^{-1} S_{\mathbf f,\mathbf x}(0)}{\mathbf x},
\]
\[
\gamma(\mathbf f,\mathbf x) \defeq \max_{k \ge 2} 
\left(
\frac{1}{k!}
\sup_{ \finsler{\mathbf w_1}{\mathbf x}, \dots, \finsler{\mathbf w_k}{\mathbf x} \le 1}
\finsler{DS_{\mathbf f,\mathbf x}(0) ^{-1} D^k S_{\mathbf f,\mathbf x}(0) (\mathbf w_1, \dots, \mathbf w_k)}{\mathbf x}
\right)^{\frac{1}{k-1}}
\]
and $\alpha(\mathbf f,\mathbf x) = \beta(\mathbf f,\mathbf x)\gamma(\mathbf f,\mathbf x)$.

The invariant $\mu$ is more delicate. It 
was defined as the operator norm of the map
\[
DG(f): 
\left(T_{[f_1]} \mathbb P(\mathscr F_1) \times \cdots \times T_{[f_n]} 
\mathbb P(\mathscr F_n),
\|\cdot\|_{\mathbf x}\right)
\rightarrow (T_{\mathbf x} \mathscr M, \|\cdot \|_{\mathbf x}) 
\]
where the product norm was assumed in the domain of $DG(\mathbf f)$.
We redefine $\mu$ as the operator norm of the same map between
different spaces. In the manifold
\[
 \mathbb P(\mathscr F_1) \times \cdots \times  \mathbb P(\mathscr F_n)
\]
we also define a Finsler structure,
\[
\finsler{\dot{\mathbf f}}{[\mathbf f]} = \max_i \|\dot{\mathbf f}_i\|_{[f_i]} .
\]
Now,
\[
DG(f): \left( T_{f_1} \mathbb P(\mathscr F_1) \times \cdots \times T_{f_n} 
\mathbb P(\mathscr F_n)
, \finsler{\cdot}{x} \right)
\rightarrow
(T_{\mathbf x} \mathscr M, \finsler{\cdots}{\mathbf x})  
\]
and the norm on the domain is
\[
\finsler{\dot{\mathbf f}}{[\mathbf f]} = \max_i \|\dot{\mathbf f}_i\|_{[f_i]}
.
\]
An alternative formulation is
\[
\begin{split}
\mu(\mathbf f,\mathbf x) &= \finsler{
DS_{\mathbf f,\mathbf x}(0)^{-1}
\begin{pmatrix}
\|f_1\| \\
& \ddots \\
& & \|f_n\|
\end{pmatrix}
}{\infty,\mathbf x}
\\
&=\left(
\inf_{\finsler {\mathbf w}{i,\mathbf x} \le 1} \max_i 
\frac {|f_i \cdot DV_i(\mathbf x) \mathbf w|}{\|\mathbf f\|_i}{\|V_i(\mathbf x)\|} 
\right)^{-1}
.
\end{split}
\]
The expression above guarantees that $\mu(\mathbf f,\mathbf x) \ge 1$ always.
The invariant $\nu$ is already defined in terms of the $\|\cdot\|_{i,\mathbf x}$
so it does not change.

In the proof of Theorem~\ref{variation:moment}, only the inner products
$\langle \cdot , \cdot \rangle_{i,\mathbf x}$ appear, and this is the only place
in the proof of Main Theorems A and B where an Hermitian structure is
used. 

The definition of the multiprojective metric in Theorem~\ref{main-bound}
should be modified to be compatible with the Finsler structure. Now,
\[
d_P(\mathbf f,\mathbf g) = \max_i \inf_{\lambda \in \mathbb C} \frac{\| f_i - \lambda g_i \|}{\|f_i\|}.
\]
As usual, $d_P(\mathbf f,\mathbf g) \le d(\mathbf f,\mathbf g)$ where $d$ is the Finslerian distance from
$\mathbf f$ to $\mathbf g$.

The proofs of Theorem~\ref{th:variation} and
\ref{higher} must be modified because of the operator norm
$\finsler{\cdot}{\infty, \mathbf x}$.

\begin{proof}[Proof of Theorem~\ref{th:variation} for the Finsler structure]
We assume without loss of generality that $\mathbf m_i(\mathbf x)=0$ for all $i$,
scale the $f_i$ such that
$\|f\|_1 = \cdots = \|f_n\|=1$ and then scale 
the $g_i$ such that $\|f_i-g_i\|$ is
minimal. The sine distance now is the sine distance for the
Finsler metric, that is $d_P(\mathbf f,\mathbf g)=\max_i \|f_i-g_i\|$.
Let $v_i(\mathbf x) = \frac{1}{\|V_i(\mathbf x)\|} V_i(\mathbf x)$. 
Because $\mathbf m_i(\mathbf x)=0$ for all $i$, we can write
\[
\mu(\mathbf f,\mathbf x) = \finsler{ (\mathbf f \cdot D\mathbf v(\mathbf x))^{-1} }{\infty,\mathbf x}.
\]

Lemma~ \ref{lemma:invertible} provides us with the inequality
\[
\frac{\mu(\mathbf f,\mathbf x)}
{1+\mu(\mathbf f,\mathbf x) T}
\le
\left\| (\mathbf g \cdot D\mathbf v(\mathbf y))^{-1} \right\|_{\mathbf x}
\le
\frac{\mu(\mathbf f,\mathbf x)}
{1-\mu(\mathbf f,\mathbf x) T}
\]
where now, $T = \finsler{\mathbf f \cdot D\mathbf v(\mathbf x) - \mathbf g \cdot D\mathbf v(\mathbf y)}{\mathbf x,\infty}$.
We estimate $T=T'+T''$ where
\begin{eqnarray*}
T' &=& \finsler{ \mathbf f \cdot Dv(\mathbf x) - \mathbf g \cdot D\mathbf v(\mathbf x) }{\mathbf x,\infty} \\
&\le& 
\sup_{\finsler{\mathbf w}{\mathbf x}\le 1} \max_i \left| (f_i-g_i) \cdot Dv_i(\mathbf x) \mathbf w \right\| 
\\
&\le&
\max_i \|f_i - g_i\|
\end{eqnarray*}
and
\begin{eqnarray*}
T'' &=& \finsler{\mathbf g \cdot D\mathbf v(\mathbf x) - \mathbf g \cdot D\mathbf v(\mathbf y)}{\mathbf x,\infty} \\
&=&
\sup_{\finsler{\mathbf w}{\mathbf x} \le 1}
\max_i \left|g_i (Dv_i(\mathbf x) - Dv_i(\mathbf y))\mathbf w\right|
\\
&\le&
\sup_{\finsler{\mathbf w}{\mathbf x}\le 1}
\max_i \|(Dv_i(\mathbf x) - Dv_i(\mathbf y))\mathbf w\|
\\
&\le&
\sup_{\finsler{\mathbf w}{\mathbf x}\le 1}
\max_i 
\sum_{k \ge 2} \frac{1}{k-1!}\|D^kv_i(\mathbf x)(\mathbf w, \mathbf y-\mathbf x, \cdots, \mathbf y-\mathbf x)\|
\\
&\le&
\max_i 
\sum_{k \ge 2} \frac{1}{k-1!} \nu_i(\mathbf x)^{k-1} \|\mathbf y-\mathbf x\|_{i,\mathbf x}^{k-1}
\\
&\le& (e^s-1)
\end{eqnarray*}
As before,
\[
\frac{
\left( 2- e^{s}\right)
\mu(f,x)}
{1+\mu(\mathbf f,\mathbf x) \left( \|\mathbf f-\mathbf g\| + (e^{s}-1) \right)}
\le
\mu(\mathbf g,\mathbf y)
\le
\frac{
e^s
\mu(\mathbf f,\mathbf x)}
{1-\mu(\mathbf f,\mathbf x) \left( \|\mathbf f-\mathbf g\| + (e^{s}-1) \right)}
\]
\end{proof}

\begin{proof}[Proof of Theorem \ref{higher} for the Finsler structure]
As before, assume without loss of generality that $\mathbf m_i(\mathbf x)=0$ for all $i$.
\begin{eqnarray*}
\frac{1}{k!} 
\finsler{
DS_{\mathbf f,\mathbf x}(0)^{-1} D^ kS_{\mathbf f,\mathbf x}(0)
}{\mathbf x} \le \hspace{-12.5em}&&
\\
&\le&
\frac{1}{k!}
\finsler{
DS_{\mathbf f,\mathbf x}(0)^{-1} 
\begin{pmatrix}
\|f_1\| \\
& \ddots \\
& & \|f_n\|
\end{pmatrix}
}{\infty,\mathbf x}
\finsler{
\begin{pmatrix}
\frac{1}{\|f_1\|} f_1 \cdot \frac{1}{\|V_1(\mathbf x)\|} D^k V_1(\mathbf x)
\\
\vdots 
\\
\frac{1}{\|f_n\|} f_n \cdot \frac{1}{\|V_n(\mathbf x)\|} D^k V_n(\mathbf x)
\end{pmatrix}
}{\mathbf x,\infty}
\\
&\le&
\frac{1}{k!}
\mu(\mathbf f,\mathbf x) 
\max_{i}
\sup_{\finsler{\mathbf w_1}{\mathbf x}, \dots, \finsler{\mathbf w_k}{\mathbf x} \le 1}
\left|
\frac{1}{\|V_i(\mathbf x)\|} D^k V_i(\mathbf x)(\mathbf w_1, \dots, \mathbf w_k)
\right|
\\
&\le&
\frac{1}{k!}\mu(\mathbf f,\mathbf x) \nu(\mathbf x)^{k-1}
.
\end{eqnarray*}
as in the proof of Lemma~\ref{lem:metric}.
We can still use $\mu(\mathbf f,\mathbf x) \ge 1$ to bound the expression above
by 
\[
\frac{1}{k!} 
\finsler{
DS_{\mathbf f,\mathbf x}(0)^{-1} D^ kS_{\mathbf f,\mathbf x}(0)
}{\mathbf x} \le 
\frac{1}{2^{k-1}}\mu(\mathbf f,\mathbf x)^{k-1} \nu(\mathbf x)^{k-1},
\]
and take $k-1$-th roots.
\end{proof}

\section{Conclusions and future work}
\label{sec:conclusions}
The theory of condition numbers and homotopy for sparse systems
proposed in this paper shares many of the features of the
theory of homotopy algorithms for dense polynomial systems:
there are effective criteria for quadratic convergence, a 
Lipschitz condition number, a higher derivative estimate
and the toric {\em condition length} is an upper
bound for the cost of homotopy algorithms. 
\par
This bound is possibly
sharper from what we would obtain from the theory of dense 
homogeneous or multi-homogeneous equations, as illustrated by
the running example. On the other hand, this theory has some
distinctive features.
\par
The higher derivative estimate for $\gamma(\mathbf f,\mathbf x)$ 
is less sharp as $x$ goes to toric infinity. This is to be expected,
since in the toric case `infinity' means a supporting facet of the
support. Therefore it may be necessary to `switch charts' at some
point and appromiate roots going to infinity by points at infinity.
In the mean time, we are left with the undesirable features of the
non-homogenized, later discarded version of the theory in~\ocite{Bezout1}.
\par
Nothing was said about implementation issues. Some of them may
require experimentation. For instance, it is not clear if the
extra sharpness provided by the Finsler structure does offset
the extra cost of computing it. This may depend on how many variables
appear on each polynomial. 
\par
Then we need a probabilistic analysis of the condition
of sparse polynomial systems.
This may be a challenging
problem. Previous results obtained by ~\ocite{Malajovich-Rojas} depend
on polynomial systems being unmixed or on a {\em mixed dilation} which
is only finite for {\em nondegenerate} fewnomial spaces as in 
Definition~\ref{fewspaces}(iii). This is an inconvenient hypothesis. Removing
it is a topic for future research.

\renewcommand{\MR}[1]{}

\appendix
\section{Proof of Lemma~\ref{plane-swap}}

We start with a real version of Lemma~\ref{plane-swap}.
This will be used to recover the complex version. The notation
$\langle \cdot . \cdot \rangle$ stands for the canonical Hermitian
inner product in $\mathbb C^n$, and $\langle \cdot . \cdot \rangle_{\mathbb R^n}$ is the real canonical inner product. Identifying $\mathbb C^n$ to
$\mathbb R^{2n}$ we can write
\[
\Re \left(\langle \cdot . \cdot \rangle \right) =  \langle \cdot . \cdot \rangle_{\mathbb R^{2n}}
.
\]
Since the same norm arises from those two inner products, we use the notation
$\| \cdot \|$ for it. Here is the real Lemma:

\begin{lemma}\label{A1}
Suppose that $\mathbf x, \mathbf y, \boldsymbol \zeta \in \mathbb R^{n+1}$ with
$\boldsymbol \zeta-\mathbf x \perp \mathbf x$, $\mathbf y-\mathbf x \perp \mathbf x$ and $\|\mathbf y-\boldsymbol \zeta\|\le\|\mathbf x-\boldsymbol \zeta\|$.
Then,
\[
\frac{\| \pi_{\mathbb R}(\mathbf y)-\boldsymbol \zeta \|}{\|\boldsymbol \zeta\|}
\le
\frac{\|\mathbf y-\boldsymbol \zeta\|}{\|\mathbf x\|}
\]
where $\pi_{\mathbb R}(\mathbf y) = 
\frac{\|\boldsymbol \zeta\|^2}{\langle \mathbf y,\boldsymbol \zeta \rangle_{\mathbb R^{n+1}}} \mathbf y
$ is the radial projection
onto the real affine plane $\boldsymbol \zeta+\boldsymbol \zeta^{\perp}$.
\end{lemma}

\begin{proof}
Rescaling the three vectors $\mathbf x, \mathbf y$ and $\boldsymbol \zeta$
simultaneously we can assume that $\|\mathbf x\|=1$. Then we can choose
an orthonormal basis $(\mathbf e_0, \dots, \mathbf e_n)$ so that
$\mathbf x = \mathbf e_0$, $\boldsymbol\zeta$ is in the span of $\mathbf e_0$ and
$\mathbf e_1$ and $y$ is in the span of $\mathbf e_0, \mathbf e_1$
and $\mathbf e_2$. In coordinates,
\[
\mathbf x = \begin{pmatrix}
1 \\
0 \\
0 \\
0 \\
\vdots
\end{pmatrix}
\hspace{1em}
\text{ , }
\hspace{1em}
\boldsymbol \zeta = \begin{pmatrix}
1 \\
t \\
0 \\
0 \\
\vdots
\end{pmatrix}
\hspace{1em}
\text{ and }
\hspace{1em}
\mathbf y = \begin{pmatrix}
1 \\
s \\
r \\
0 \\
\vdots
\end{pmatrix}
.
\]
We can further assume that $t \ge 0$ and $r \ge 0$. Squaring both sides
of the hypothesis
$\|\mathbf y-\boldsymbol \zeta\|\le\|\mathbf x-\boldsymbol \zeta\|$
we obtain
\[
r^2 + (s-t)^2 \le t^2
\]
that is
\begin{equation}\label{r-bound}
r^2 \le 2 st - s^2
\end{equation}
which implies $s \ge 0$.

We claim first that
\begin{equation}
\label{first-bound}
\frac{ \|\pi_{\mathbb R}(\mathbf y)-\boldsymbol \zeta\| }{\|\mathbf y-\boldsymbol \zeta\|}
\le
\frac{ \|\pi_{\mathbb R}(\mathbf x)-\boldsymbol \zeta\| }{\|\mathbf x-\boldsymbol \zeta\|}
.
\end{equation}
We compute
\begin{eqnarray*}
\|\pi_{\mathbb R}(\mathbf y)-\boldsymbol \zeta\|^2&=&
{{\left(t^2+1\right)\,\left(r^2\,t^2+t^2-2\,s\,t+s^2+r^2\right) }\over{\left(s\,t+1\right)^2}}
\\
\|\mathbf y-\boldsymbol \zeta\|^2 &=& 
t^2-2\,s\,t+s^2+r^2
\\
\|\pi_{\mathbb R}(\mathbf x)-\boldsymbol \zeta\|^2 &=&
t^2\,\left(t^2+1\right)\\
\|\mathbf x-\boldsymbol \zeta\|^2
&=&
t^2
\end{eqnarray*}
To show inequation \eqref{first-bound}, we just need to verify that
\[
K = 
\|\pi_{\mathbb R}(\mathbf y)-\boldsymbol \zeta\|^2
\|\mathbf x-\boldsymbol \zeta\|^2
-
\|\pi_{\mathbb R}(\mathbf x)-\boldsymbol \zeta\|^2 
\|\mathbf y-\boldsymbol \zeta\|^2 
\le 0
\]
Using the Maxima computer algebra system \cite{MAXIMA},
\[
K = - \frac{t^3 (t^2+1)}{(st+1)^2} \left( A r^2 + B\right) 
\]
with
\[
A=(s^2-1)t + 2s
\hspace{1em}
\text{ and }
\hspace{1em}
B=s (t-s)^2 (st+2)
\]
From the factorization above, $K$ is negative if and only
if $A r^2 + B \ge 0$. Clearly $B \ge 0$.
If $A \ge 0$ we are done, so assume
$A<0$. Then multiplying both sides of \eqref{r-bound} by $A$,
one obtains
\[
A r^2 \ge 2 A st - A s^2
\]
and
\[
A r^2 +B \ge s^2 t(1+t^2) \ge 0
.
\]
This shows \eqref{first-bound}.
Also,
\[
\frac{ \|\pi_{\mathbb R}(\mathbf x)-\boldsymbol \zeta\|^2}
{\|\mathbf x-\boldsymbol \zeta\|^2}
=
1+t^2 = \frac{\|\zeta\|^2}{\|x\|^2}
.
\]
Taking square roots and combining with \eqref{first-bound},
\[
\frac{\| \pi_{\mathbb R}(\mathbf y)-\boldsymbol \zeta \|}{\|\boldsymbol \zeta\|}
\le
\frac{\|\mathbf y-\boldsymbol \zeta\|}{\|\mathbf x\|}
.
\]
\end{proof}

\renewcommand{\thetheorem}{\ref{plane-swap}}
\begin{lemma}
Suppose that $\mathbf x, \mathbf y, \boldsymbol \zeta \in \mathbb C^{n+1}$ with
$\boldsymbol \zeta-\mathbf x \perp \mathbf x$, $\mathbf y-\mathbf x \perp \mathbf x$ and $\|\mathbf y-\boldsymbol \zeta\|\le\|\mathbf x-\boldsymbol \zeta\|$.
Then,
\[
\frac{\| \pi(\mathbf y)-\boldsymbol \zeta \|}{\|\boldsymbol \zeta\|}
\le
\frac{\|\mathbf y-\boldsymbol \zeta\|}{\|\mathbf x\|}
\]
where $\pi(\mathbf y) = 
\changed{ 
\frac{\|\boldsymbol \zeta\|^2}{\langle \mathbf y,\boldsymbol \zeta \rangle} \mathbf y}
$ is the radial projection
onto the affine plane $\boldsymbol \zeta+\boldsymbol \zeta^{\perp}$.
\end{lemma}

\begin{proof}
We identify $\mathbb C^{n+1}$ with $\mathbb R^{2n+2}$ and claim that
\begin{equation}\label{second-bound}
\| \pi(\mathbf y)-\boldsymbol \zeta \|
\le
\| \pi_{\mathbb R}(\mathbf y)-\boldsymbol \zeta \|
.
\end{equation}
Since complex orthogonal vectors are also real orthogonal,
inequation \eqref{second-bound} and Lemma~\ref{A1} imply
\[
\frac{\| \pi(\mathbf y)-\boldsymbol \zeta \|}{\|\boldsymbol \zeta\|}
\le
\frac{\| \pi_{\mathbb R}(\mathbf y)-\boldsymbol \zeta \|}{\|\boldsymbol \zeta\|}
\le
\frac{\|\mathbf y-\boldsymbol \zeta\|}{\|\mathbf x\|}
.
\]

To show \eqref{second-bound} we choose coordinates so that
\[
\boldsymbol \zeta = \begin{pmatrix}
1 \\
0 \\
0 \\
\vdots
\end{pmatrix}
\hspace{1em}
\text{ and }
\hspace{1em}
\mathbf y = \begin{pmatrix}
a+b i \\
c \\
0 \\
\vdots
\end{pmatrix}
\]
with $c \ge 0$. A straight-forward computation gives
\[
\pi (y) = \begin{pmatrix}
1 \\
\frac{c}{a+bi} \\
0 \\
\vdots
\end{pmatrix}
\hspace{1em}
\text{ and }
\hspace{1em}
\pi_{\mathbb R} (\mathbf y) = \begin{pmatrix}
1+\frac{b}{a} i \\
\frac{c}{a} \\
0 \\
\vdots
\end{pmatrix}
\]

We have
\[
\| \pi(\mathbf y)-\boldsymbol \zeta \|^2
=
\frac{c^2}{a^2+b^2}
\le
\frac{b^2}{a^2} 
+ 
\frac{c^2}{a^2}
=
\| \pi_{\mathbb R} (\mathbf y)-\boldsymbol \zeta \|^2
\]
with equality if $b=0$. This finishes the proof.
\end{proof}

\end{document}